\documentclass[10pt,a4paper]{amsart}

\usepackage[left=2.8cm, right=2.8cm, top=2.2cm, bottom=2.2cm]{geometry}

\usepackage{amsmath,amssymb,amsthm,mathrsfs,stmaryrd,latexsym}
\usepackage{accents}

\usepackage{relsize}
\usepackage[bbgreekl]{mathbbol}
\usepackage{amsfonts}
\DeclareSymbolFontAlphabet{\mathbb}{AMSb} 
\DeclareSymbolFontAlphabet{\mathbbl}{bbold}
\newcommand{\prism}{{\mathlarger{\mathbbl{\Delta}}}}

\usepackage{tcolorbox}
\tcbuselibrary{skins, breakable, theorems}

\usepackage{enumerate} 
\usepackage{tikz}
\usepackage{etoolbox}

\def\Luoma#1{\uppercase\expandafter{\romannumeral#1}}
\def\luoma#1{\romannumeral#1}

\usepackage{url}

\newtheorem{mythm}{Theorem}[section]
\newtheorem{mylem}[mythm]{Lemma}
\newtheorem{myprop}[mythm]{Proposition}
\newtheorem{mycor}[mythm]{Corollary}

\theoremstyle{definition}
\newtheorem{mydefn}[mythm]{Definition}

\theoremstyle{remark}
\newtheorem{myrem}[mythm]{Remark}
\newtheorem{mypara}[mythm]{}

\usepackage[all]{xy}
\usepackage{graphicx}
\usepackage{subfigure}

\newcommand{\bb}{\mathbb}
\newcommand{\ca}{\mathcal}
\newcommand{\ak}{\mathfrak}
\newcommand{\scr}{\mathscr}

\newcommand{\mbf}{\mathbf}
\newcommand{\mrm}{\mathrm}

\newcommand{\trm}{\textrm}

\def\op#1{\mathop{\mathrm{#1}}}

\newcommand{\ho}{\mrm{Hom}}

\newcommand{\ke}{\mrm{Ker}}
\newcommand{\cok}{\mrm{Coker}}
\newcommand{\im}{\mrm{Im}}
\newcommand{\df}{\mrm{d}}
\newcommand{\id}{\mrm{id}}

\newcommand{\spec}{\op{Spec}}
\newcommand{\colim}{\op{colim}}

\newcommand{\rr}{\mrm{R}}
\newcommand{\dl}{\mrm{L}}

\newcommand{\iso}{\stackrel{\sim}{\longrightarrow}}

\newcommand{\plim}{\varprojlim}

\newcommand{\al}{\mrm{al}}

\newcommand{\gal}{\mrm{Gal}}

\newcommand{\et}{\mrm{\acute{e}t}}
\newcommand{\fet}{\mrm{f\acute{e}t}}

\newcommand{\profet}{\mrm{prof\acute{e}t}}

\newcommand{\fal}{\mbf{E}}
\newcommand{\falh}{\mbf{I}}
\newcommand{\falb}{\overline{\scr{B}}}

\title[The $p$-adic Galois Cohomology of Valuation Fields]{The $p$-adic Galois Cohomology of Valuation Fields}%
\author{Tongmu He}
\date{\today}
\address{Tongmu He, Princeton University, Fine Hall, Washington Road, Princeton, New Jersey 08544, the United States}
\email{hetm@princeton.edu}

\usepackage{hyperref}
\hypersetup{colorlinks,
	citecolor=black,
	filecolor=black,
	linkcolor=black,
	urlcolor=black,
	pdftitle={The p-adic Galois Cohomology of Valuation Fields},
	pdfauthor={Tongmu He},
	pdftex}

\numberwithin{equation}{mythm}

\begin{document}
	\maketitle
	
\begin{abstract}
	We compute the Galois cohomology of any $p$-adic valuation field extension of a pre-perfectoid field. Moreover, we obtain a generalization and also a new proof of the classical results of Tate and Hyodo on discrete valuation fields, without using higher ramification group, local class field theory or Epp's elimination of ramifications. A key ingredient is Gabber-Ramero's computation of cotangent complexes for valuation rings.
\end{abstract}
	
\tableofcontents

\section{Introduction}

\begin{mypara}
	The development of $p$-adic Hodge theory can be traced back to Tate's foundational work \cite{tate1967p} in the 1960's. One of the most profound results is his computation of the $p$-adic Galois cohomology of a $p$-adic local field. More precisely, let $K$ be a complete discrete valuation field extension of $\bb{Q}_p$ with absolute Galois group $G_K=\gal(\overline{K}/K)$. Tate computed the continuous cohomology of $G_K$ with coefficients in the $p$-adic completed algebraic closure $\widehat{\overline{K}}$, under the assumption that the residue field of $K$ is perfect. This result was later generalized by Hyodo in the 1980's to the imperfect residue field case as follows:
\end{mypara}

\begin{mythm}[{\cite[\textsection3.3]{tate1967p}, \cite[Theorem 1]{hyodo1986hodge}}]\label{thm:tate-hyodo}
	For any integers $q\geq 0$ and $n$, we have
	\begin{align}
		H^q(G_K,\widehat{\overline{K}}(n))=\left\{\begin{array}{ll}
			\widehat{\Omega}^q_K&\textrm{if }n=q,\\
			\widehat{\Omega}^{q-1}_K&\textrm{if }n=q-1,\\
			0&\textrm{otherwise},
		\end{array}\right.
	\end{align}
	where $(n)$ is taking the $n$-th Tate twist, $\widehat{\Omega}^q_{\ca{O}_K/\bb{Z}_p}$ is the $p$-adic completion of the module of $q$-th differentials $\Omega^q_{\ca{O}_K/\bb{Z}_p}$, and $\widehat{\Omega}^q_K=\widehat{\Omega}^q_{\ca{O}_K/\bb{Z}_p}[1/p]$.
\end{mythm}

\begin{mypara}
	Tate and Hyodo first applied this result to study the splitting of the Hodge-Tate filtration for abelian varieties over $K$. Their approach later had a profound influence on $p$-adic geometry. In particular, Faltings \cite{faltings1988p,faltings2002almost} adapted and further developed their proof strategy of \ref{thm:tate-hyodo} to study the $p$-adic \'etale cohomology of smooth algebraic varieties, establishing deep connections with various $p$-adic cohomologies arising from differential forms. These fundamental results in $p$-adic Hodge theory turn out to exhibit an intrinsically rigid-analytic nature, as originally conjectured by Tate \cite{tate1967p} and later proved by Scholze \cite{scholze2013hodge}. Since each point of a rigid analytic space corresponds to a valuation field that is not necessarily discretely valued, it is natural to extend the study of $p$-adic Hodge theory to general valuation fields.
	
	However, Tate's approach relies on higher ramification groups and local class field theory, while Hyodo's extension further invokes Epp's theorem on eliminating wild ramification. These techniques are inherently tied to the discrete valuation structure, posing significant obstacles to any direct generalization to the non-discrete setting.
	
	In a recent work by the author \cite{he2024perfd}, a perfectoidness criterion for general valuation fields was established. A key ingredient in this result is the extension of Tate’s strategy to the non-discrete setting, utilizing Gabber-Ramero's computation of cotangent complexes \cite{gabber2003almost}. This criterion ultimately leads to the resolution of a long-standing conjecture of Calegari-Emerton \cite{calegariemerton2012completed} for all Shimura varieties. Moreover, it suggests that the techniques developed therein could be employed to compute Galois cohomology and develop (both abelian and non-abelian) $p$-adic Hodge theory for general valuation fields. The present article is devoted to a systematic study of the $p$-adic Galois cohomology of valuation fields, laying the groundwork for further advances in this direction.
	
	Our main result is stated as follows:
\end{mypara}

\begin{mythm}[{see \ref{cor:ari} and \ref{cor:geo}}]\label{thm:intro-coh}
	Let $F$ be a Henselian valuation field of height $1$ extension of $\bb{Q}_p$ with absolute Galois group $G_F=\gal(\overline{F}/F)$.
	\begin{enumerate}
		\renewcommand{\labelenumi}{{\rm(\theenumi)}}
		\item (Arithmetic case) Assume that $\Omega^1_{\ca{O}_F/\bb{Z}_p}[p^\infty]$ is killed by a certain power of $p$. Then, for any integers $q\geq 0$ and $n$, we have
		\begin{align}
			H^q(G_F,\widehat{\overline{F}}(n))=\left\{\begin{array}{ll}
				\widehat{\Omega}^q_F&\textrm{if }n=q,\\
				\widehat{\Omega}^{q-1}_F&\textrm{if }n=q-1,\\
				0&\textrm{otherwise}.
			\end{array}\right.
		\end{align}\label{item:thm:intro-coh-1}
		\item (Geometric case) Assume that $\bigcap_{n\in\bb{N}}\scr{D}_{F_n/F}\ca{O}_{F_\infty}\neq 0$, where $\scr{D}_{F_n/F}$ is the different ideal of the finite extension $F_n/F$ by adjoining a primitive $p^n$-th root of unity $\zeta_{p^n}$. Then, for any integers $q\geq 0$ and $n$, we have
		\begin{align}
			H^q(G_F,\widehat{\overline{F}}(n))=\widehat{\Omega}^q_F(n-q).
		\end{align}\label{item:thm:intro-coh-2}
	\end{enumerate}
\end{mythm}
\begin{mypara}
	We say that $F$ is \emph{arithmetic} (resp. \emph{geometric}) if it satisfies the corresponding assumption in \ref{thm:intro-coh} (see \ref{defn:ari-geo}). There are some typical examples for arithmetic and geometric valuation fields (see \ref{lem:ari-geo-basic}):
	\begin{enumerate}
		\renewcommand{\labelenumi}{{\rm(\theenumi)}}
		\item If the valuation on $F$ is discrete, then $F$ is arithmetic.
		\item If there is a pre-perfectoid field valuation subextension of $F/\bb{Q}_p$ (e.g., $F=F_\infty$), then $F$ is geometric.
	\end{enumerate}
	Therefore, the arithmetic case is a generalization of Tate and Hyodo's computation \ref{thm:tate-hyodo}, and the geometric case can be regarded as a (rigid-analytically) local version of Faltings' computation of the $p$-adic \'etale cohomology of smooth varieties. These two cases are mostly common encountered in the practice of $p$-adic geometry. It enables us to construct an analogue of Hodge-Tate spectral sequence for (non-smooth) proper $p$-adic varieties via a completely different approach from Guo's work \cite{guo2019hodgetate} (see \ref{rem:htss}). We also obtain a more conceptual proof of the perfectoidness criterion established in \cite{he2024perfd} (see \ref{prop:fal-ext-delta}). Moreover, these two cases are sufficiently broad to cover all the valuation fields finitely generated over $\bb{Q}_p$ with transcendental degree $\leq 1$.
\end{mypara}

\begin{myprop}[{see \ref{cor:sen-perfd} and \ref{prop:trans-ari-geo}}]\label{prop:intro-ari-geo}
	Let $K$ be a Henselian discrete valuation field extension of $\bb{Q}_p$ with perfect residue field. 
	\begin{enumerate}
		\renewcommand{\labelenumi}{{\rm(\theenumi)}}
		\item Let $L$ be an algebraic extension of $K$. Then, $L$ is either arithmetic or pre-perfectoid (thus geometric).\label{item:prop:intro-ari-geo-1}
		\item Let $F$ be a valuation field of height $1$ finitely generated extension of $L$ of transcendental degree $\mrm{trdeg}_L(F)= 1$. Then, the Henselization of $F$ is either arithmetic or geometric.\label{item:prop:intro-ari-geo-2}
	\end{enumerate}
\end{myprop}

\begin{mypara}
	We remark that if $L$ is a $p$-adic analytic Galois extension of $K$, then a theorem of Sen on the higher ramification groups implies that $L$ is either finitely ramified or pre-perfectoid (see \ref{rem:sen-perfd}). Thus, \ref{prop:intro-ari-geo}.(\ref{item:prop:intro-ari-geo-1}) can be regarded as an analogue of this result, but our proof does not make any use of the local class field theory.
\end{mypara}

\begin{mypara}
	Let's sketch the proof of our main theorem \ref{thm:intro-coh}. The arithmetic case can be reduced to the geometric case by computing the Galois cohomology of the cyclotomic extension $H^i(\gal(F_\infty/F),\widehat{\Omega}^j_{F_\infty}(n))$. For the latter, we show that there exists a certain element $\pi\in \ca{O}_F$ such that for any $n\in\bb{N}$, we have
	\begin{align}
		\pi\ca{O}_{F_n} \subseteq \ca{O}_F[\zeta_{p^n}]\subseteq \ca{O}_{F_n}.
	\end{align}
	This follows from a ramified base change theorem that we establish in \ref{thm:ramified-bc} and that $F$ is arithmetic (see \ref{lem:ari-bc}). We remark that the different ideals $(\scr{D}_{F_n/F})_{n\in\bb{N}}$ are asymptotically equivalent to $(p^n\ca{O}_{F_n})_{n\in\bb{N}}$ as in Tate's work  \cite[\textsection3.1]{tate1967p}.
	
	The geometric case follows from the following integral version:
\end{mypara}

\begin{mythm}[{see \ref{thm:geo}}]\label{thm:intro-geo}
	Assume that $F$ is geometric. Then, for any $q,r\in\bb{N}$, there is a canonical morphism
	\begin{align}\label{eq:thm:intro-geo}
		\Omega^q_{\ca{O}_F/\bb{Z}_p}/p^r\Omega^q_{\ca{O}_F/\bb{Z}_p}\longrightarrow H^q(G_F,(\ca{O}_{\overline{F}}/p^r\ca{O}_{\overline{F}})\{q\}),
	\end{align}
	whose kernel and cokernel are killed by a certain power of $p$ independent of $r$.
\end{mythm}

\begin{mypara}
	Here, $\{q\}$ is taking the $q$-th Breuil-Kisin-Fargues twist over $\ca{O}_{\widehat{\overline{F}}}$-modules, which is an integral version of the $q$-th Tate twist. In fact, such a twist is defined over the geometric (or more generally, non-arithmetic) valuation field $F$ (see \ref{defn:int-tate-twist}). Therefore, taking Galois cohomology over geometric valuation fields commutes with the twists (see \ref{lem:non-ari-coh-twist}).
	
	The construction of the canonical morphism \eqref{eq:thm:intro-geo} follows from the same strategy of Bhatt-Morrow-Scholze for smooth formal schemes (see \ref{lem:comp}). It allows us to reduce to the case where $F$ is of finite transcendental degree over the cyclotomic field $\bb{Q}_p(\zeta_{p^\infty})$ by a limit argument (see \ref{thm:coh}). Then, we can pick up a $\pi$-basis of $\widehat{\Omega}^1_{\ca{O}_F/\bb{Z}_p}$ for any $\pi\in\ak{m}_F$ (see \ref{lem:differential}), i.e., $t_1,\dots,t_d\in \ca{O}_F^\times$ such that
	\begin{align}
		\pi\cdot \widehat{\Omega}^1_{\ca{O}_F/\bb{Z}_p}\subseteq \ca{O}_{\widehat{F}}\cdot\df t_1\oplus \cdots\oplus \ca{O}_{\widehat{F}}\cdot\df t_d\subseteq \widehat{\Omega}^1_{\ca{O}_F/\bb{Z}_p}.
	\end{align}
	This is actually a general feature for any $p$-adic completion of torsion-free $\ca{O}_F$-modules of finite rank (see \ref{thm:completion} and \ref{prop:generator-basis}).
	
	Then, we consider the perfectoid tower defined by adjoining $p$-power roots of $t_1,\dots,t_d$ as in Hyodo's work \cite[\textsection1]{hyodo1986hodge}, and we reduce to the Galois cohomology of the Kummer extension by adjoining $p$-power roots of a single element $t$. The latter was already studied in \cite{he2024perfd} and we include it in Section \ref{sec:kummer}.
\end{mypara}

\begin{mypara}
	We remark that during the preparation of this article, Bhatt explained to the author a prismatic approach to Theorem \ref{thm:intro-geo}, building on Bouis' computation \cite{bouis2023cartier} of the prismatic cohomology of valuation fields extension over a pre-perfectoid field. The latter relies not only on the Hodge-Tate and de Rham comparisons for prismatic cohomology, but also on Gabber-Ramero's work on valuation rings (see \ref{thm:cartier-smooth} and \ref{thm:cartier-smooth-val}). We present Bhatt's approach in the Appendix \ref{sec:pris}.
\end{mypara}

\begin{mypara}
	This article is structured as follows. In Section \ref{sec:mod}, we study the structure of a torsion-free module of finite rank over a valuation ring, using Kaplansky's structure theorem over maximal valuation rings. In Section \ref{sec:diff}, we discuss the general properties of different ideals, traces and differentials for (non-discrete) valuation rings, and we establish the ramified base change theorem in the end. In Section \ref{sec:ari-geo}, we introduce the notion of arithmetic and geometric valuation fields. Then, we compute the Galois cohomology for valuation field extension of the cyclotomic field in Sections \ref{sec:kummer} and \ref{sec:perfd}. A prismatic approach of Bhatt to this result is included in the Appendix \ref{sec:pris}. We establish the canonical comparison map between differentials and Galois cohomology in Section \ref{sec:comp} and finish the computation/comparison for arithmetic and geometric valuation fields in Section \ref{sec:final}.
\end{mypara}

\subsection*{Acknowledgements}

I am grateful to Bhargav Bhatt for his explaination of a prismatic approach to Galois cohomology, and I thank Tess Bouis for helpful discussions. I would like to extend my gratitude to Ahmed Abbes and Takeshi Tsuji for their invaluable feedback and comments. This work was carried out during my appointment as an instructor at Princeton University.

\section{Notation and Conventions}\label{sec:notation}
\begin{mypara}
	We fix a prime $p$ throughout this article.
\end{mypara}

\begin{mypara}\label{para:product}
	Let $d\in\bb{N}$, i.e., $d$ is a natural number. We endow the set $(\bb{N}\cup\{\infty\})^d$ with the partial order defined by $\underline{n}\leq \underline{n'}$ if ether $n'_i=\infty$ or $n_i\leq n'_i<\infty$ for any $1\leq i\leq d$, where $\underline{n}=(n_1,\dots,n_d)$ and $\underline{n'}=(n_1',\dots,n_d')$. For any $r\in \bb{N}\cup\{\infty\}$, we set $\underline{r}=(r,\dots,r)\in (\bb{N}\cup\{\infty\})^d$.
\end{mypara}

\begin{mypara}\label{para:torsion}
	Let $A$ be a ring, $I$ an ideal of $A$, $M$ an $A$-module. We denote by $M[I]=\{x\in M\ |\ I\cdot x=0\}$ the submodule of $I$-torsion elements and by $M[I^\infty]=\bigcup_{n\in\bb{N}}M[I^n]$ the submodule of $I$-power torsion elements.
\end{mypara}

\begin{mypara}\label{para:notation-val}
	A \emph{valuation field} is a pair $(K,\ca{O}_K)$ (usually denoted by $K$ for simplicity) where $\ca{O}_K$ is a valuation ring with fraction field $K$ (\cite[\Luoma{6}.\textsection1.2, D\'efinition 2]{bourbaki2006commalg5-7}). We denote by $\ak{m}_K$ the maximal ideal of $\ca{O}_K$ and call the cardinality of the nonzero prime ideals of $\ca{O}_K$ the \emph{height} of $K$ (\cite[\Luoma{6}.\textsection4.4, Proposition 5]{bourbaki2006commalg5-7}). We also refer to \cite[\href{https://stacks.math.columbia.edu/tag/00I8}{00I8}]{stacks-project} for basic properties on valuation rings.  
\end{mypara}

\begin{mypara}\label{para:notation-almost}
	Let $K$ be a discrete (resp. non-discrete) valuation field of height $1$, i.e., its value group $K^\times/\ca{O}_K^\times$ identifies with a discrete (resp. dense) subgroup of the ordered topological group of real numbers $\bb{R}$. We put $\widetilde{\ak{m}}_K=\ca{O}_K$ (resp. $\widetilde{\ak{m}}_K=\ak{m}_K$) so that
	\begin{align}\label{eq:para:notation-almost-1}
		\widetilde{\ak{m}}_K^2=\widetilde{\ak{m}}_K,\quad \widetilde{\ak{m}}_K\ak{m}_K=\ak{m}_K.
	\end{align}
	Moreover, for any extension $K\to F$ of valuation fields of height $1$, we have
	\begin{align}\label{eq:para:notation-almost-2}
		\ak{m}_K\cdot\ca{O}_F\subseteq\widetilde{\ak{m}}_{F}\subseteq\widetilde{\ak{m}}_K\cdot \ca{O}_F.
	\end{align}
	When referring to almost modules over $\ca{O}_K$, we always take $(\ca{O}_K,\widetilde{\ak{m}}_K)$ as the basic setup (\cite[6.1.15]{gabber2003almost}). Note that if $K$ is discrete, then the category of almost modules over $\ca{O}_K$ is equivalent to the category of $\ca{O}_K$-modules. We refer to \cite{gabber2003almost} and \cite[\Luoma{5}]{abbes2016p} for a systematic development of almost ring theory, and to \cite[\textsection5]{he2024coh} for a brief review on basic notions.
\end{mypara}

\begin{mypara}\label{para:notation-norm}
	Let $K$ be a valuation field of height $1$ and let $|\cdot|:K\to\bb{R}_{\geq 0}$ be an associated ultrametric absolute value (\cite[\Luoma{6}.\textsection6.2, Proposition 3]{bourbaki2006commalg5-7}). For any subset $S\subseteq K$, we define the \emph{norm} of $S$ to be 
	\begin{align}\label{eq:para:notation-norm-1}
		|S|=\sup_{s\in S}|s|\in\bb{R}_{\geq 0}\cup\{\infty\}.
	\end{align}
	We note that $|\widetilde{\ak{m}}_K|=|\ca{O}_K|=1$. We say that an $\ca{O}_K$-submodule $\ak{a}$ of $K$ is \emph{principal} if there exists $x\in\ak{a}$ such that $|\ak{a}|=|x|$ (so that $\ak{a}=\ca{O}_K\cdot x$). A \emph{fractional ideal} of $\ca{O}_K$ is an $\ca{O}_K$-submodule $\ak{a}$ of $K$ with finite norm $|\ak{a}|<\infty$ (i.e., $\ak{a}\neq K$). Note that if $K$ is discrete, then any fractional ideal is principal.
	
	For any two subsets $S$ and $T$ of $K$, we put $ST=\{st\in K\ |\ s\in S,\ t\in T\}$. Note that for any $\ca{O}_K$-submodule $\ak{a}$ of $K$, $S\ak{a}$ is also an $\ca{O}_K$-submodule of $K$. We remark that 
	\begin{align}\label{eq:para:notation-norm-2}
		|ST|=|S|\cdot |T|.
	\end{align}
	
	For any two $\ca{O}_K$-submodules $\ak{a}$ and $\ak{b}$ of $K$, we see that 
	\begin{align}\label{eq:para:notation-norm-3}
		|\ak{a}|\leq |\ak{b}|\Leftrightarrow\widetilde{\ak{m}}_K\ak{a}\subseteq \ak{b}\Leftarrow \ak{a}\subseteq \ak{b}.
	\end{align}
	Moreover, the last implication is an equivalence when $\ak{b}$ is principal. On the other hand, we have
	\begin{align}\label{eq:para:notation-norm-4}
		|\ak{a}|<|\ak{b}|\Rightarrow \ak{a}\subseteq \ak{m}_K\ak{b}.
	\end{align}
	
	For any nonzero fractional ideal $\ak{a}$ of $\ca{O}_K$, we put $\ak{a}^{-1}=\{x\in K\ |\ x\ak{a}\subseteq \ca{O}_K\}$, which is still a nonzero fractional ideal of $\ca{O}_K$. Then, we have (see \cite[4.4]{he2024perfd})
	\begin{align}\label{eq:para:notation-norm-5}
		\widetilde{\ak{m}}_K\subseteq \ak{a}\ak{a}^{-1}\subseteq \ca{O}_K.
	\end{align}
	In particular, $|\ak{a}^{-1}|=|\ak{a}|^{-1}$.
\end{mypara}

\begin{mypara}\label{para:notation-completion}
	Let $K$ be a valuation field of height $1$. There is a canonical topology on $K$, induced by any associated ultrametric absolute value or by the $\pi$-adic topology on $\ca{O}_K$ for any $\pi\in\ak{m}_K\setminus 0$. The completion of $\ca{O}_K$, $\ca{O}_{\widehat{K}}=\lim_{n\in\bb{N}}\ca{O}_K/\pi^n\ca{O}_K$, is a valuation ring of height $1$ extension of $\ca{O}_K$, and the completion of $K$ coincides with the fraction field $\widehat{K}$ of $\ca{O}_{\widehat{K}}$ (\cite[\Luoma{6}.\textsection5.3, Proposition 5]{bourbaki2006commalg5-7}). For any $\ca{O}_K$-module $M$, we endow it with the $\pi$-adic topology (which does not depend on the choice of $\pi$) and denote its completion by $\widehat{M}=\lim_{n\in\bb{N}}M/\pi^nM$.
\end{mypara}

\begin{mypara}\label{para:notation-perfd}
	Following \cite[\textsection5]{he2024coh}, a \emph{pre-perfectoid field} is a valuation field $K$ of height $1$ with non-discrete valuation of residue characteristic $p$ such that the Frobenius map on $\ca{O}_K/p\ca{O}_K$ is surjective (\cite[5.1]{he2024coh}). We note that any pre-perfectoid field is deeply ramified in the sense of \cite[6.6.1]{gabber2003almost} (see \cite[6.6.6]{gabber2003almost}). We also note that the completion $\widehat{K}$ is a perfectoid field in the sense of \cite[3.1]{scholze2012perfectoid} (see \cite[5.2]{he2024coh}). 
	
	Given a pre-perfectoid field $K$, we say that an $\ca{O}_K$-algebra $R$ is (resp. \emph{almost}) \emph{pre-perfectoid} if there is a nonzero element $\pi\in\ak{m}_K$ such that $p\in \pi^p\ca{O}_K$, that the $\pi$-adic completion $\widehat{R}$ is (resp. almost) flat over $\ca{O}_{\widehat{K}}$ and that the Frobenius induces an (resp. almost) isomorphism $R/\pi R\to R/\pi^p R$ (see \cite[5.19]{he2024coh}). This definition does not depend on the choice of $\pi$ by \cite[5.23]{he2024coh} and is equivalent to that the $\ca{O}_{\widehat{K}}$-algebra $\widehat{R}$ (resp. almost $\ca{O}_{\widehat{K}}$-algebra $\widehat{R}^{\al}$ associated to $\widehat{R}$) is perfectoid in the sense of \cite[3.10.(\luoma{2})]{bhattmorrowscholze2018integral} (resp. \cite[5.1.(\luoma{2})]{scholze2012perfectoid}, see \cite[5.18]{he2024coh}).
\end{mypara}

\section{Completion of Torsion-Free Modules with Finite Rank}\label{sec:mod}
Kaplansky \cite{kaplansky1952mod} clarifies the structure of a torsion-free module of countable rank over a maximal valuation ring. We use it to study the structure of a torsion-free module of finite rank over a general valuation ring. In particular, we prove that its completion is almost finite projective (see \ref{thm:completion} and \ref{prop:generator-basis}). This result will be applied in \ref{lem:differential}.

\begin{mydefn}\label{defn:val-rank}
	Let $K$ be a valuation field. For any $\ca{O}_K$-module $M$, we call the dimension of the $K$-module $K\otimes_{\ca{O}_K}M$ the \emph{rank} of $M$, and denote by $\mrm{rank}_{\ca{O}_K}(M)$.
\end{mydefn}

	We note that for any extension of valuation fields $K\to L$ and $M'=M\otimes_{\ca{O}_K}\ca{O}_L$, we have
	
\begin{align}\label{eq:defn:val-rank-1}
	\mrm{rank}_{\ca{O}_K}(M)=\mrm{rank}_{\ca{O}_L}(M').
\end{align}

\begin{mydefn}\label{defn:div-rank}
	Let $K$ be a valuation field of height $1$. For any $\ca{O}_K$-module $M$, we put  
	\begin{align}
		M^{\mrm{div}}=\bigcap_{\pi\in\ak{m}_K\setminus 0}\pi M,\quad M^{\mrm{sep}}=M/\bigcap_{\pi\in\ak{m}_K\setminus 0}\pi M,
	\end{align}
	called the \emph{divisible submodule} and the \emph{separated quotient} of $M$ respectively. We say that $M$ is \emph{divisible} (resp. \emph{separated}) if $M^{\mrm{div}}=M$ (resp. $M^{\mrm{div}}=0$).
\end{mydefn}

	Note that for a torsion-free $\ca{O}_K$-module $M$, $M^{\mrm{div}}$ is naturally a $K$-module and $M^{\mrm{sep}}$ is torsion-free.
	
\begin{mydefn}\label{defn:sep-bound}
	Let $K$ be a valuation field of height $1$. 
	\begin{enumerate}
		\renewcommand{\labelenumi}{{\rm(\theenumi)}}
		\item We say that a torsion-free $\ca{O}_K$-module $M$ is \emph{bounded} if $M$ is isomorphic to an $\ca{O}_K$-submodule of $\ca{O}_K^{\oplus I}$ for some set $I$.\label{item:defn:sep-bound-1}
		\item We say that a torsion $\ca{O}_K$-module $M$ is \emph{bounded} if $\pi M=0$ for some nonzero element $\pi\in\ca{O}_K$.\label{item:defn:sep-bound-2}
	\end{enumerate}
	
\end{mydefn}

Note that for a torsion-free $\ca{O}_K$-module $M$, it is bounded and of finite rank if and only if it admits a finitely generated $\ca{O}_K$-submodule $M_0$ containing $\pi M$ for some $\pi\in\ak{m}_K\setminus 0$.

\begin{mylem}\label{lem:div-rank-inv-base}
	Let $K\to L$ be an extension of valuation fields of height $1$, $M$ a torsion-free $\ca{O}_K$-module with $M^{\mrm{sep}}$ bounded, $M'=M\otimes_{\ca{O}_K}\ca{O}_L$. Then, the canonical morphisms of $\ca{O}_L$-modules
	\begin{align}\label{eq:lem:div-rank-inv-base-1}
		M^{\mrm{div}}\otimes_{\ca{O}_K}\ca{O}_L\longrightarrow M'^{\mrm{div}},\quad M^{\mrm{sep}}\otimes_{\ca{O}_K}\ca{O}_L\longrightarrow M'^{\mrm{sep}}
	\end{align}
	are isomorphisms.
\end{mylem}
\begin{proof}
	Since $M^{\mrm{sep}}\subseteq \ca{O}_K^{\oplus I}$ for some set $I$, we have $M^{\mrm{sep}}\otimes_{\ca{O}_K}\ca{O}_L\subseteq (\ca{O}_L)^{\oplus I}$ as $\ca{O}_K\to\ca{O}_L$ is flat. Thus, $M^{\mrm{sep}}\otimes_{\ca{O}_K}\ca{O}_L$ is still separated. Hence, $M'^{\mrm{div}}$ lies in the kernel of $M'\to M^{\mrm{sep}}\otimes_{\ca{O}_K}\ca{O}_L$, which is $M^{\mrm{div}}\otimes_{\ca{O}_K}\ca{O}_L$ by flat base change. On the other hand, it is clear that $M^{\mrm{div}}\otimes_{\ca{O}_K}\ca{O}_L$ is still divisible and thus contained in $M'^{\mrm{div}}$. Hence, we have $M^{\mrm{div}}\otimes_{\ca{O}_K}\ca{O}_L=M'^{\mrm{div}}$ and thus $M^{\mrm{sep}}\otimes_{\ca{O}_K}\ca{O}_L= M'^{\mrm{sep}}$.
\end{proof}

\begin{myrem}\label{rem:sep-unbounded}
	If $K$ is complete, then we will prove every separated torsion-free $\ca{O}_K$-module of finite rank is bounded in \ref{cor:sep-comp}. But if $K$ is not complete, there exists an unbounded but separated torsion-free $\ca{O}_K$-module of rank $2$. Indeed, let $\pi$ be a nonzero element of $\ak{m}_K$ and assume that $(s_n)_{n\in\bb{N}}\in \ca{O}_{\widehat{K}}=\lim_{n\in\bb{N}}\ca{O}_K/\pi^n\ca{O}_K$ does not lie in $\ca{O}_K$ (where $s_n\in\ca{O}_K$). Consider the $\ca{O}_K$-submodule $M$ of $K\oplus K$ generated by $\{(\pi^{-n},s_n\pi^{-n})\}_{n\in\bb{N}}$. It is clear that $M$ is unbounded. Let's verify that $M$ is separated. Suppose that $(a,b)\in M^{\mrm{div}}$. Then, for any $k\in\bb{N}$, there is $n\in\bb{N}_{\geq k}$ large enough such that the equations
	\begin{align}\label{eq:rem:sep-unbounded-1}
		\left\{\begin{array}{l}
			x_0+x_1/\pi+\cdots +x_n/\pi^n=\pi^{-k}a\\
			x_0s_0+x_1s_1/\pi+\cdots +x_ns_n/\pi^n=\pi^{-k}b
		\end{array}\right.
	\end{align}
	admit solutions $x_0,\dots,x_n\in \ca{O}_K$. Thus, we have
	\begin{align}\label{eq:rem:sep-unbounded-2}
		\pi^n(s_0-s_n)x_0+\cdots+\pi(s_{n-1}-s_n)x_{n-1}=\pi^{n-k}(b-s_na).
	\end{align}
	Since $s_i-s_n\in \pi^i\ca{O}_K$ for any $0\leq i<n$ by assumption, we see that 
	\begin{align}\label{eq:rem:sep-unbounded-3}
		b\equiv s_na \mod \pi^k\ca{O}_K.
	\end{align}
	If $a\neq 0$, then we see that $b/a\in K$ is equal to $(s_n)_{n\in\bb{N}}\in\widehat{K}$, which is a contradiction. Therefore, $a=0$ and thus $b\in \ca{O}_K^{\mrm{div}}$ is also zero. In other words, $M$ is separated. Moreover, we also see that $M\otimes_{\ca{O}_K}\ca{O}_{\widehat{K}}$ is not separated, since the equations \eqref{eq:rem:sep-unbounded-1} admit solutions in $\ca{O}_{\widehat{K}}$ for $n$ large enough if we take $b=(s_n)_{n\in\bb{N}}$ and $a=1$. This also shows that the ``bounded" assumption in \ref{lem:div-rank-inv-base} can not be removed. 
\end{myrem}

\begin{mypara}\label{para:maximal-val}
	Recall that a valuation field $K$ is \emph{maximal} (or \emph{maximally complete}) if any immediate valuation field extension $K\to L$ (i.e., a valuation field extension inducing isomorphisms of value groups and residue fields) is trivial (\cite[page 191]{krull1932val}). In particular, a maximal valuation field is complete with respect to the topology induced by the valuation (\cite[\Luoma{6}.\textsection5.3, Proposition 5]{bourbaki2006commalg5-7}). We refer to \cite[\textsection11]{he2024purity} for a brief review of Kaplansky's structure theorem on torsion-free modules of countable rank over maximal valuation rings.
\end{mypara}

\begin{mylem}\label{lem:val-R}
	Let $K$ be a valuation field whose value group $K^\times/\ca{O}_K^\times$ is isomorphic to the ordered group of real numbers $\bb{R}$. Then, any torsion-free $\ca{O}_K$-module of rank $1$ is isomorphic to exactly one of the three $\ca{O}_K$-modules: $K$, $\ca{O}_K$ and $\ak{m}_K$.
\end{mylem}
\begin{proof}
	Let $v:K\to \bb{R}\cup\{\infty\}$ be the valuation map constructed by the assumption and $v(0)=\infty$. As $M$ is torsion-free and of rank $1$, it is a nonzero $\ca{O}_K$-submodule of $K$. If $\inf_{x\in M}v(x)=-\infty$, then it is clear that $M=\bigcup_{x\in M}\ca{O}_K\cdot x=K$. If $\inf_{x\in M}v(x)\in \bb{R}$, then there exists $x_0\in K$ such that $v(x_0)= \inf_{x\in M}v(x)$ by assumption. If moreover $x_0\in M$, then $M=\ca{O}_K\cdot x_0\cong \ca{O}_K$. Otherwise, we see that $M=\ak{m}_K\cdot x_0\cong \ak{m}_K$. It is clear that $K$, $\ca{O}_K$ and $\ak{m}_K$ are not isomorphic to each other.
\end{proof}

\begin{mylem}\label{lem:max-comp-exist}
	Let $K$ be a valuation field of height $1$. Then, there exists a maximal valuation field $L$ of height $1$ extension of $K$ with value group isomorphic to $\bb{R}$.
\end{mylem}
\begin{proof}
	As the value group of $K$ is isomorphic to a subgroup of $\bb{R}$ (\cite[\Luoma{6}.\textsection4.5, Proposition 7]{bourbaki2006commalg5-7}), there exists a valuation field extension $K\to K'$ such that the value group of $K'$ is isomorphic to $\bb{R}$ (\cite{fukawa1965val}). Then, it suffices to take an immediate valuation field extension $K'\to L$ with $L$ maximal, which exists by \cite[Satz 24]{krull1932val}. See also \cite[7.4]{escassut1995analysis} for another proof using ultraproducts following Diarra \cite{diarra1984ultra}.
\end{proof}

\begin{mythm}[{\cite[Theorem 12]{kaplansky1952mod}}]\label{thm:kaplansky}
	Let $K$ be a valuation field of height $1$ and let $L$ be a maximal valuation field of height $1$ extension of $K$ with value group isomorphic to $\bb{R}$ (which exists by {\rm\ref{lem:max-comp-exist}}). Then, for any torsion-free $\ca{O}_K$-module $M$ of countable rank, there is an isomorphism of $\ca{O}_L$-modules
	\begin{align}\label{eq:thm:kaplansky-1}
		M\otimes_{\ca{O}_K}\ca{O}_L\cong L^{\oplus I_1}\oplus \ca{O}_L^{\oplus I_2}\oplus \ak{m}_L^{\oplus I_3},
	\end{align}
	where $I_1,I_2,I_3$ are countable sets.
\end{mythm}
\begin{proof}
	By Kaplansky's structure theorem \cite[Theorem 12]{kaplansky1952mod}, the torsion-free $\ca{O}_L$-module $M\otimes_{\ca{O}_K}\ca{O}_L$ of countable rank is isomorphic to a countable direct sum of torsion-free $\ca{O}_L$-modules of rank $1$. Thus, the conclusion follows from \ref{lem:val-R}. 
\end{proof}

\begin{mycor}\label{cor:kaplansky}
	Let $K$ be a valuation field of height $1$, $M$ a torsion-free $\ca{O}_K$-module. Then, $M$ is of countable rank if and only if it is countably generated.
\end{mycor}
\begin{proof}
	Firstly, we take a sequence $\{\pi_n\}_{n\in\bb{N}}$ in $\ak{m}_K$ such that $\lim_{n\to\infty} v(\pi_n)=\inf_{\pi\in\ak{m}_K}v(\pi)$, where $v:K\to\bb{R}\cup\{\infty\}$ is a valuation map. We see that this sequence generates the maximal ideal $\ak{m}_K$. In particular, $\ak{m}_K$ is countably generated. We need to verify that $M$ is countably generated if it is of countable rank. Following \ref{thm:kaplansky}, we see that $M\otimes_{\ca{O}_K}\ca{O}_L$ is countably generated. Thus, there is a countably generated submodule $M_0$ of $M$ which generates $M\otimes_{\ca{O}_K}\ca{O}_L$ over $\ca{O}_L$. By faithfully flat descent along $\ca{O}_K\to\ca{O}_L$, we see that $M_0=M$. 
\end{proof}

\begin{mypara}
	Let $K$ be a valuation field of height $1$, $R$ an $\ca{O}_K$-algebra. Recall that an $R$-module $M$ is \emph{almost finitely generated} if for any $\pi\in\widetilde{\ak{m}}_K$ (\ref{para:notation-almost}), there exists a finitely generated $R$-submodule $M_\pi\subseteq M$ containing $\pi M$ (\cite[2.3.10.(\luoma{1})]{gabber2003almost}, see also \cite[\Luoma{5}.3.5]{abbes2016p}).
\end{mypara}

\begin{mylem}\label{lem:pi-surj}
	Let $R$ be a ring, $\pi$ an element of $R$, $f:M\to N$ a homomorphism of $R$-modules. If the cokernel of $M/\pi^nM\to N/\pi^nN$ is killed by $\pi$ for any $n\in\bb{N}$, then the cokernel of the $\pi$-adic completions $\widehat{M}\to\widehat{N}$ is also killed by $\pi$.
\end{mylem}
\begin{proof}
	Let $Q$ be the cokernel of $f:M\to N$ and let $C_n=\{x\in M\ |\ f(x)\in \pi^nN\}$. Then, there is a canonical exact sequence of $R$-modules for any $n\in\bb{N}$,
	\begin{align}\label{eq:lem:pi-surj-1}
		0\longrightarrow C_n/\pi^nM\longrightarrow M/\pi^nM\longrightarrow N/\pi^nN\longrightarrow Q/\pi^nQ\longrightarrow 0.
	\end{align}
	
	We claim the the inverse system $(C_n/\pi^nM)_{n\in\bb{N}}$ satisfies the Mittag-Leffler condition. Let $\varphi_{mn}:C_m/\pi^mM\to C_n/\pi^nM$ denote the transition morphism for any integers $m\geq n\geq 0$. For any $m\geq n\geq 1$ and $x\in C_n$, we write $f(x)=\pi^ny$ with $y\in N$. Since $\pi \cdot (Q/\pi^mQ)=0$ by assumption, we have $\pi y=f(x')+\pi^my'$ for some $x'\in M$ and $y'\in N$. Thus, $f(x)=\pi^{n-1}f(x')+\pi^{m+n-1}y'$. This shows that $x-\pi^{n-1}x'\in C_{m+n-1}\subseteq C_m$. As $\pi^{n-1}x'\in \pi^kM$ for any integer $0\leq k\leq n-1$, we have
	\begin{align}
		\varphi_{mk}(x-\pi^{n-1}x')=\varphi_{nk}(x).
	\end{align}
	This shows that $\im(\varphi_{nk})\subseteq\im(\varphi_{mk})$ (and thus $\im(\varphi_{nk})=\im(\varphi_{mk})$) for any integers $m\geq n>k\geq 0$, which proves the claim.
	
	Then, applying the derived functor $\rr\lim_{n\in\bb{N}}$ to the exact sequence of Mittag-Leffler inverse systems associated to \eqref{eq:lem:pi-surj-1}, we obtain an exact sequence (\cite[\href{https://stacks.math.columbia.edu/tag/091D}{091D}]{stacks-project})
	\begin{align}\label{eq:lem:pi-surj-2}
		0\longrightarrow \lim_{n\in\bb{N}}C_n/\pi^nM\longrightarrow \widehat{M}\longrightarrow \widehat{N}\longrightarrow \widehat{Q}\longrightarrow 0.
	\end{align}
	In particular, the cokernel of $\widehat{M}\to\widehat{N}$ is killed by $\pi$.
\end{proof}

\begin{mylem}\label{lem:al-finite-comp}
	Let $K$ be a valuation field of height $1$, $R$ an $\ca{O}_K$-algebra, $M$ an almost finitely generated $R$-module. Then, the canonical morphism $M\otimes_R\widehat{R}\to \widehat{M}$ is almost surjective, where the completions are $\pi$-adic for any nonzero element $\pi\in\ak{m}_K$. In particular, $\widehat{M}$ is also almost finitely generated over $\widehat{R}$.
\end{mylem}
\begin{proof}
	If the valuation on $K$ is discrete, then the conclusion follows from \cite[\href{https://stacks.math.columbia.edu/tag/0315}{0315}]{stacks-project}. Hence, we may assume that $K$ is non-discrete. For any $\pi\in\ak{m}_K$, there is $n\in\bb{N}$ and a homomorphism $R^{\oplus n}\to M$ with cokernel killed by $\pi$. Then, we see that the cokernel of $\widehat{R}^{\oplus n}\to \widehat{M}$ is still killed by $\pi$ by \ref{lem:pi-surj} and thus so is the cokernel of $M\otimes_R\widehat{R}\to \widehat{M}$. Therefore, we conclude that $M\otimes_R\widehat{R}\to \widehat{M}$ is almost surjective and that $\widehat{M}$ is also almost finitely generated.
\end{proof}

\begin{mythm}\label{thm:completion}
	Let $K$ be a valuation field of height $1$, $M$ a torsion-free $\ca{O}_K$-module of finite rank. Then, the completion $\widehat{M}$ is an almost finitely generated torsion-free $\ca{O}_{\widehat{K}}$-module with $\mrm{rank}_{\ca{O}_{\widehat{K}}}(\widehat{M})\leq \mrm{rank}_{\ca{O}_K}(M)$. Moreover, the canonical morphism
	\begin{align}
		M\otimes_{\ca{O}_K}\ca{O}_{\widehat{K}}\longrightarrow \widehat{M}
	\end{align}
	is surjective.
\end{mythm}
\begin{proof}
	Following \ref{thm:kaplansky}, we take a finitely generated $\ca{O}_K$-submodule $M_0$ of $M$ such that $M_0\otimes_{\ca{O}_K}\ca{O}_L$ contains the finitely generated submodule $\ca{O}_L^{\oplus I_2}\oplus (\pi\ca{O}_L)^{\oplus I_3}$ for some $\pi\in\ak{m}_K\setminus 0$ under the isomorphism $
	M\otimes_{\ca{O}_K}\ca{O}_L\cong L^{\oplus I_1}\oplus \ca{O}_L^{\oplus I_2}\oplus \ak{m}_L^{\oplus I_3}$ \eqref{eq:thm:kaplansky-1}, where $I_1,I_2,I_3$ are finite sets. As $(M/\pi^nM)\otimes_{\ca{O}_K}\ca{O}_L\cong (\ca{O}_L/\pi^n \ca{O}_L)^{\oplus I_2}\oplus (\ak{m}_L/\pi^n\ak{m}_L)^{\oplus I_3}$, we see that the cokernel of $M_0/\pi^nM_0\to M/\pi^nM$ is killed by $\pi$ for any $n\in\bb{N}$ by faithfully flat descent along $\ca{O}_K\to\ca{O}_L$. Therefore, the cokernel of $\widehat{M_0}\to \widehat{M}$ is also killed by $\pi$ by \ref{lem:pi-surj}.
	
	As $M_0$ is a finitely generated torsion-free $\ca{O}_K$-module, it is finite free by \cite[\Luoma{6}.\textsection3.6, Lemme 1]{bourbaki2006commalg5-7}. Hence, $\widehat{M_0}=M_0\otimes_{\ca{O}_K}\ca{O}_{\widehat{K}}$ is finite free over $\ca{O}_{\widehat{K}}$. In particular, $\widehat{M}$ is bounded (\ref{defn:sep-bound}.(\ref{item:defn:sep-bound-1})). Note that $\widehat{M}$ is torsion-free as $M$ is so (cf. \cite[5.7.(1)]{tsuji2018localsimpson}).
	
	Notice that $\ca{O}_K\to \ca{O}_L$ uniquely factors through $\ca{O}_{\widehat{K}}$ as $\ca{O}_L$ is complete (see \ref{para:maximal-val}). Since $\ak{m}_L$ is an open ideal of $\ca{O}_L$, it is also complete. Thus, the composition of the canonical morphisms
	\begin{align}
		L^{\oplus I_1}\oplus \ca{O}_L^{\oplus I_2}\oplus \ak{m}_L^{\oplus I_3}\cong M\otimes_{\ca{O}_K}\ca{O}_L\longrightarrow \widehat{M}\otimes_{\ca{O}_{\widehat{K}}}\ca{O}_L\longrightarrow M\widehat{\otimes}_{\ca{O}_K}\ca{O}_L\cong \ca{O}_L^{\oplus I_2}\oplus \ak{m}_L^{\oplus I_3}
	\end{align}
	is surjective. Let $D$ be the kernel of the surjective map $ \widehat{M}\otimes_{\ca{O}_{\widehat{K}}}\ca{O}_L\to M\widehat{\otimes}_{\ca{O}_K}\ca{O}_L$. Modulo $\pi$, we see that $D/\pi D=0$ as $M\widehat{\otimes}_{\ca{O}_K}\ca{O}_L$ is flat over $\ca{O}_K$. But $\widehat{M}\otimes_{\ca{O}_{\widehat{K}}}\ca{O}_L$ is separated by \ref{lem:div-rank-inv-base} due to the boundedness of $\widehat{M}$. We have $D=0$. Therefore, $\widehat{M}\otimes_{\ca{O}_{\widehat{K}}}\ca{O}_L\to M\widehat{\otimes}_{\ca{O}_K}\ca{O}_L$ is an isomorphism. In particular, we see that $M\otimes_{\ca{O}_K}\ca{O}_L\to \widehat{M}\otimes_{\ca{O}_{\widehat{K}}}\ca{O}_L$ is surjective (and thus so is $M\otimes_{\ca{O}_K}\ca{O}_{\widehat{K}}\to \widehat{M}$) and that $\widehat{M}$ is almost finitely generated as $\ca{O}_L^{\oplus I_2}\oplus \ak{m}_L^{\oplus I_3}$ is so by faithfully flat descent (\cite[3.2.26.(\luoma{2})]{gabber2003almost}, see also \cite[\Luoma{5}.8.1]{abbes2016p}).
\end{proof}

\begin{mycor}\label{cor:sep-comp}
	Let $K$ be a complete valuation field of height $1$, $M$ a torsion-free $\ca{O}_K$-module of finite rank. Then, the following four properties on $M$ are equivalent to each other: bounded, separated, complete, and almost finitely generated.
\end{mycor}
\begin{proof}
	If $M$ is almost finitely generated, then it is bounded by definition \ref{defn:sep-bound}.(\ref{item:defn:sep-bound-1}). If $M$ is bounded, then it is separated as any free $\ca{O}_K$-module is so. If $M$ is separated, then the canonical morphism $M\to \widehat{M}$ is injective. But it is also surjective by \ref{thm:completion}. Hence, $M$ is complete. Finally, if $M$ is complete, then it is almost finitely generated by \ref{thm:completion}.
\end{proof}

\begin{myrem}\label{rem:sep-comp}
	Let $K$ be a complete valuation field of height $1$, $M$ a torsion-free $\ca{O}_K$-module of finite rank. Then, in the canonical exact sequence (\ref{defn:div-rank})
	\begin{align}
		0\longrightarrow M^{\mrm{div}}\longrightarrow M\longrightarrow M^{\mrm{sep}}\longrightarrow 0,
	\end{align}
	$M^{\mrm{div}}$ is a $K$-module of finite rank and $M^{\mrm{sep}}=\widehat{M}$ (\ref{thm:completion}) is an almost finitely generated complete $\ca{O}_K$-module (\ref{cor:sep-comp}). We will see that $M^{\mrm{sep}}$ is almost finite projective in \ref{prop:generator-basis}.
\end{myrem}

\begin{mydefn}\label{defn:pi-basis}
	Let $K$ be a valuation field of height $1$, $\ak{a}$ an ideal of $\ca{O}_K$, $M$ a torsion-free $\ca{O}_K$-module. We say that a subset $S$ of $M$ is an \emph{$\ak{a}$-generating subset} if the submodule generated by $S$ contains $\ak{a} M$. We say that an $\ak{a}$-generating subset $S\subset M$ is an \emph{$\ak{a}$-basis} if its elements are $\ca{O}_K$-linearly independent. If $\ak{a}$ is generated by a subset $A\subseteq \ca{O}_K$, then we simply call an $\ak{a}$-generating subset (resp. $\ak{a}$-basis) an \emph{$A$-generating subset} (resp. \emph{$A$-basis}).
\end{mydefn}

We note that an $\ak{a}$-generating subset (resp. $\ak{a}$-basis) is also an $\ak{b}$-generating subset (resp. $\ak{b}$-basis) for any ideal $\ak{b}\subseteq \ak{a}$.

\begin{mylem}\label{lem:pi-basis}
	Let $K$ be a valuation field of height $1$, $\ak{a}$ an ideal of $\ca{O}_K$, $M$ a torsion-free $\ca{O}_K$-module admitting an $\ak{a}$-basis $S$. Then, any element of $S$ does not lie in $\ak{m}_K \ak{a} M$. 
\end{mylem}
\begin{proof}
	Let $M_0$ be the submodule of $M$ generated by $S$.	By definition, $M_0$ is a free $\ca{O}_K$-module with basis $S$ and we have $\ak{a} M\subseteq M_0$. Then, we have $ \ak{m}_K\ak{a} M\subseteq \ak{m}_K M_0$. Since any element of $S$ is not zero in $M_0/\ak{m}_K M_0\cong (\ca{O}_K/\ak{m}_K)^{\oplus S}$, it is not zero in $M_0/ \ak{m}_K \ak{a} M$, which completes the proof.
\end{proof}

\begin{mylem}\label{lem:generator-basis}
	Let $R$ be a local ring, $\kappa$ its residue field. For any integer $n>0$, $n$ elements $x_1,\dots,x_n$ of the finite free $R$-module $R^{\oplus n}$ form a basis if and only if their images in $\kappa^{\oplus n}$ form a basis over $\kappa$.
\end{mylem}
\begin{proof}
	The ``only if" part is clear. For the ``if" part, firstly the $R$-linear map $R^{\oplus n}\to R^{\oplus n}$ sending the standard basis $(e_1,\dots,e_n)$ to $(x_1,\dots,x_n)$ is surjective by Nakayama's lemma (\cite[\href{https://stacks.math.columbia.edu/tag/00DV}{00DV}]{stacks-project}). Furthermore, it is an isomorphism by Cayley-Hamilton theorem (\cite[\href{https://stacks.math.columbia.edu/tag/05G8}{05G8}]{stacks-project}).
\end{proof}

\begin{myprop}\label{prop:generator-basis}
	Let $K$ be a valuation field of height $1$, $M$ an almost finitely generated torsion-free $\ca{O}_K$-module, $S$ a generating subset of $M$. Then, for any $\pi\in\widetilde{\ak{m}}_K$, there are $\mrm{rank}_{\ca{O}_K}(M)$ elements of $S$ forming a $\pi$-basis of $M$. In particular, $M$ is almost finite projective.
\end{myprop}
\begin{proof}
	For any $\pi\in\widetilde{\ak{m}}_K$, there exists a finitely generated $\ca{O}_K$-submodule $M_\pi$ of $M$ containing $\pi M$. Since $S$ generates $M$, there is a finite subset $S_\pi$ of $S$ generating a submodule $M_\pi'$ containing $M_\pi$. Since $M_\pi'$ is a finitely generated torsion-free $\ca{O}_K$-module, it is finite free by \cite[\Luoma{6}.\textsection3.6, Lemme 1]{bourbaki2006commalg5-7}. We take a subset $S'_\pi$ of $S_\pi$ such that its images in $M_\pi'/\ak{m}_KM_\pi'$ form a basis over the residue field of $\ca{O}_K$. Thus, $S'_\pi$ is a basis of the finite free $\ca{O}_K$-module $M_\pi'$ by \ref{lem:generator-basis}. As $\pi M\subseteq M_\pi'$, we see that $S'_\pi$ is a $\pi$-basis of $M$. It is clear that $M$ has the same rank as $M'_\pi$ and thus the cardinality of $S'_\pi$ is $\mrm{rank}_{\ca{O}_K}(M)$. On the other hand, it implies that $M$ is almost finite projective (\cite[2.4.15.(\luoma{2})]{gabber2003almost}, see also \cite[\Luoma{5}.3.7]{abbes2016p}).
\end{proof}

\section{Different Ideals and Ramified Base Change Theorem}\label{sec:diff}
Following \cite{gabber2003almost} and \cite[\textsection4]{he2024perfd}, we discuss the relation between traces, different ideals and differentials of an extension of valuation fields, generalizing the classical results for discrete valuation fields. Moreover, we establish a ramified base change theorem \ref{thm:ramified-bc}, which will be used to analyze the structure of cyclotomic extension of a general valuation ring in \ref{lem:ari-bc}.

\begin{mypara}\label{para:different}
	Let $K\to K'$ be a finite separable extension of Henselian valuation fields of height $1$ (so that $\ca{O}_{K'}$ is the integral closure of $\ca{O}_K$ in $K'$). Recall that its trace morphism $\mrm{Tr}_{K'/K}:K'\to K$ induces a $K'$-linear isomorphism $\tau_{K'/K}:K'\iso \ho_K(K',K),\ x\mapsto (y\mapsto \mrm{Tr}_{K'/K}(xy))$ (\cite[4.1.14]{gabber2003almost}). In particular, the $\ca{O}_{K'}$-submodule
	\begin{align}\label{eq:para:different-1}
		\ca{O}_{K'}^*=\{x\in K'\ |\ \mrm{Tr}_{K'/K}(xy)\in\ca{O}_K,\ \forall y\in\ca{O}_{K'}\}
	\end{align}
	is not equal to $K'$, and thus a fractional ideal of $K'$ containing $\ca{O}_{K'}$, called the \emph{codifferent ideal} of the extension $K'/K$. We call the ideal of $\ca{O}_{K'}$,
	\begin{align}\label{eq:para:different-2}
		\scr{D}_{K'/K}=\{x\in\ca{O}_{K'}\ |\ x\ca{O}_{K'}^*\subseteq \ca{O}_{K'}\},
	\end{align}
	the \emph{different ideal} of the extension $K'/K$.
\end{mypara}

	As $\ca{O}_{K'}$ is almost finite projective over $\ca{O}_K$ (\cite[6.3.8]{gabber2003almost}), Gabber-Ramero \cite[4.1.22]{gabber2003almost} define the different ideal $\scr{D}_{\ca{O}_{K'}^{\al}/\ca{O}_K^{\al}}$ of the associated almost module $\ca{O}_{K'}^{\al}$ over $\ca{O}_K^{\al}$. One can check by the same argument as in \cite[4.3, 4.5]{he2024perfd} that $\scr{D}_{\ca{O}_{K'}^{\al}/\ca{O}_K^{\al}}$ coincides with the associated almost module $\scr{D}_{K'/K}^{\al}$ of $\scr{D}_{K'/K}$. We remark that $\scr{D}_{K'/K}^{\al}=\ca{O}_{K'}^{\al}$ (i.e., $\widetilde{\ak{m}}_{K'}\subseteq \scr{D}_{K'/K}$) if and only if $\ca{O}_{K'}$ is almost finite \'etale over $\ca{O}_K$ (\cite[4.1.27]{gabber2003almost}).

\begin{myprop}\label{prop:ramified-bc}
	Let $K\to F$ be an extension of Henselian valuation fields of height $1$, and let $K'$ be a finite separable field extension of $K$ such that $F'=K'\otimes_KF$ is a field. 
	\begin{align}\label{eq:prop:ramified-bc-1}
		\xymatrix{
			F'&F\ar[l]\\
			K'\ar[u]&K\ar[l]\ar[u]
		}
	\end{align}
	Then, $\widetilde{\ak{m}}_K\scr{D}_{K'/K}\subseteq \scr{D}_{F'/F}$. Moreover, for any $\pi\in\ca{O}_{F'}$ such that $\pi\scr{D}_{F'/F}\subseteq \scr{D}_{K'/K}\ca{O}_{F'}$, we have
	\begin{align}\label{eq:prop:ramified-bc-2}
		\pi\widetilde{\ak{m}}_{F'}\subseteq \ca{O}_{K'}\otimes_{\ca{O}_K}\ca{O}_F\subseteq \ca{O}_{F'}.
	\end{align}
\end{myprop}
\begin{proof}
	Firstly, note that the integral closures $\ca{O}_{K'}$ (resp. $\ca{O}_{F'}$) of $\ca{O}_K$ in $K'$ (resp. $\ca{O}_F$ in $F'$) are valuation rings since $\ca{O}_K$ (resp. $\ca{O}_F$) is Henselian (see \cite[\Luoma{6}.\textsection8.6, Proposition 6]{bourbaki2006commalg5-7} and \cite[\href{https://stacks.math.columbia.edu/tag/04GH}{04GH}]{stacks-project}). 
	
	Since $\ca{O}_{K'}\otimes_{\ca{O}_K}\ca{O}_F$ is still torsion-free over $\ca{O}_K$ and $F'=K'\otimes_KF$, the canonical morphism of $\ca{O}_F$-algebras $\ca{O}_{K'}\otimes_{\ca{O}_K}\ca{O}_F\to\ca{O}_{F'}$ is injective. Consider the $\ca{O}_{K'}\otimes_{\ca{O}_K}\ca{O}_F$-submodule of $F'$,
	\begin{align}\label{eq:prop:ramified-bc-3}
		(\ca{O}_{K'}\otimes_{\ca{O}_K}\ca{O}_F)^*=\{x\in F'\ |\ \mrm{Tr}_{F'/F}(xy)\in \ca{O}_F,\ \forall y\in\ca{O}_{K'}\otimes_{\ca{O}_K}\ca{O}_F\},
	\end{align}
	which clearly contains
	\begin{align}\label{eq:prop:ramified-bc-4}
		\ca{O}_{F'}^*=\{x\in F'\ |\ \mrm{Tr}_{F'/F}(xy)\in\ca{O}_F,\ \forall y\in\ca{O}_{F'}\}.
	\end{align}
	In particular, we obtain inclusions of $\ca{O}_{K'}\otimes_{\ca{O}_K}\ca{O}_F$-submodule of $F'$,
	\begin{align}\label{eq:prop:ramified-bc-5}
		\ca{O}_{K'}\otimes_{\ca{O}_K}\ca{O}_F\subseteq \ca{O}_{F'}\subseteq \ca{O}_{F'}^*\subseteq (\ca{O}_{K'}\otimes_{\ca{O}_K}\ca{O}_F)^*.
	\end{align}
	
	We claim that the canonical morphism
	\begin{align}\label{eq:prop:ramified-bc-6}
		\ca{O}_{K'}^*\otimes_{\ca{O}_K}\ca{O}_F\longrightarrow (\ca{O}_{K'}\otimes_{\ca{O}_K}\ca{O}_F)^*
	\end{align}
	is an almost isomorphism of $\ca{O}_K$-modules. Indeed, it follows from \cite[4.1.24]{gabber2003almost} but can be also checked directly as follows. Since the trace morphism $\mrm{Tr}_{F'/F}=\mrm{Tr}_{K'/K}\otimes\id_F:F'\to F$ induces an $F'$-linear isomorphism $\tau_{F'/F}=\tau_{K'/K}\otimes\id_F:F'\iso \ho_F(F',F),\ x\mapsto (y\mapsto \mrm{Tr}_{F'/F}(xy))$, we see that \eqref{eq:prop:ramified-bc-6} identifies with the canonical morphism (see the proof of \cite[4.5]{he2024perfd})
	\begin{align}\label{eq:prop:ramified-bc-7}
		\ho_{\ca{O}_K}(\ca{O}_{K'},\ca{O}_K)\otimes_{\ca{O}_K}\ca{O}_F\longrightarrow \ho_{\ca{O}_F}(\ca{O}_{K'}\otimes_{\ca{O}_K}\ca{O}_F,\ca{O}_F).
	\end{align}
	It is an almost isomorphism since $\ca{O}_{K'}$ is almost finite projective over $\ca{O}_K$ (\cite[2.4.31.(\luoma{1})]{gabber2003almost}), which verifies the claim.
	
	The claim implies that $\widetilde{\ak{m}}_K\cdot(\ca{O}_{K'}\otimes_{\ca{O}_K}\ca{O}_F)^*\subseteq \ca{O}_{K'}^*\otimes_{\ca{O}_K}\ca{O}_F$. Thus, by \eqref{eq:prop:ramified-bc-5} we see that 
	\begin{align}
		\widetilde{\ak{m}}_K\scr{D}_{K'/K}\cdot \ca{O}_{F'}^*\subseteq\ca{O}_{K'}\otimes_{\ca{O}_K}\ca{O}_F\subseteq \ca{O}_{F'}.
	\end{align}
	In particular, we have $\widetilde{\ak{m}}_K\scr{D}_{K'/K}\subseteq \scr{D}_{F'/F}$.
	Moreover, using the facts that $\widetilde{\ak{m}}_{F'}^2=\widetilde{\ak{m}}_{F'}\subseteq\scr{D}_{F'/F}\cdot\ca{O}_{F'}^*\subseteq\ca{O}_{F'}$ (\eqref{eq:para:notation-almost-1} and \eqref{eq:para:notation-norm-5}) and $\widetilde{\ak{m}}_{F'}\subseteq \widetilde{\ak{m}}_K\ca{O}_{F'}$ \eqref{eq:para:notation-almost-2}, we see that 
	\begin{align}
		\pi\widetilde{\ak{m}}_{F'}\subseteq \pi\widetilde{\ak{m}}_{F'}\scr{D}_{F'/F}\cdot\ca{O}_{F'}^*\subseteq \widetilde{\ak{m}}_{F'} \scr{D}_{K'/K}\cdot\ca{O}_{F'}^*\subseteq \ca{O}_{K'}\otimes_{\ca{O}_K}\ca{O}_F,
	\end{align}
	for any $\pi\in\ca{O}_{F'}$ such that $\pi\scr{D}_{F'/F}\subseteq \scr{D}_{K'/K}\ca{O}_{F'}$.
\end{proof}

\begin{myprop}[{\cite[4.5]{he2024perfd}}]\label{prop:different-trace}
	Let $K\to K'$ be a finite separable extension of Henselian valuation fields of height $1$. For any nonzero fractional ideal $\ak{a}$ (resp. $\ak{a}'$) of $\ca{O}_K$ (resp. $\ca{O}_{K'}$), $\mrm{Tr}_{K'/K}(\widetilde{\ak{m}}_K\ak{a}')\subseteq \ak{a}$ if and only if $\widetilde{\ak{m}}_K\ak{a}'\scr{D}_{K'/K}\subseteq \ak{a}\ca{O}_{K'}$.
\end{myprop}
\begin{proof}
	If the valuation on $K$ is discrete, it follows from \cite[\Luoma{3}.\textsection3, Proposition 7]{serre1979local}. Otherwise, it follows from the same arguments of \cite[4.5]{he2024perfd}. Indeed, one can check the following equivalences 
	\begin{align}
		\mrm{Tr}_{K'/K}(\widetilde{\ak{m}}_K\ak{a}')\subseteq \ak{a}\Leftrightarrow \mrm{Tr}_{K'/K}(\widetilde{\ak{m}}_K\ak{a}^{-1}\ak{a}')\subseteq \ca{O}_K \Leftrightarrow \widetilde{\ak{m}}_K\ak{a}'\subseteq \ak{a}\ca{O}_{K'}^* \Leftrightarrow \widetilde{\ak{m}}_K\ak{a}'\scr{D}_{K'/K}\subseteq \ak{a}\ca{O}_{K'}
	\end{align}
	using the fact that $\widetilde{\ak{m}}_K=\widetilde{\ak{m}}_K^2$ \eqref{eq:para:notation-almost-1}, $\widetilde{\ak{m}}_K\subseteq \ak{a}\ak{a}^{-1}\subseteq \ca{O}_K$ \eqref{eq:para:notation-norm-5} and $\widetilde{\ak{m}}_{K'}\subseteq \scr{D}_{K'/K}\cdot \ca{O}_{K'}^*\subseteq \ca{O}_{K'}$ \eqref{eq:para:different-2}. 
\end{proof}

\begin{myrem}\label{rem:prop:different-trace}
	In \ref{prop:different-trace}, if we fix an associated ultrametric absolute value $|\cdot|:K'\to \bb{R}_{\geq 0}$, then we see that $|\mrm{Tr}_{K'/K}(\ak{a}')|\leq |\ak{a}|$ if and only if $|\ak{a}'\scr{D}_{K'/K}|\leq |\ak{a}|$ by \eqref{eq:para:notation-norm-3}.
\end{myrem}

\begin{mycor}[{\cite[4.6]{he2024perfd}}]\label{cor:different-trace}
	Let $K\to K'$ be a finite separable extension of Henselian valuation fields of height $1$, $|\cdot|:K'\to \bb{R}_{\geq 0}$ an ultrametric absolute value associated to the valuation on $K'$. Then, for any $x\in K'$, we have
	\begin{align}\label{eq:cor:different-trace-1}
		|\mrm{Tr}_{K'/K}(x)|\leq |\ak{m}_K^{-1}\ak{m}_{K'}\scr{D}_{K'/K}|\cdot |x|.
	\end{align}
	In particular, we have $[K':K]\cdot\ak{m}_K\ak{m}_{K'}^{-1}\subseteq \scr{D}_{K'/K}$.
\end{mycor}
\begin{proof}
	Let $\ak{a}\subseteq \ca{O}_K$ be the ideal generated by $\mrm{Tr}_{K'/K}(x\ca{O}_{K'})$, i.e., $\ak{a}=\mrm{Tr}_{K'/K}(x\ca{O}_{K'})\cdot \ca{O}_K$ (\ref{para:notation-norm}). Then, by \eqref{eq:para:notation-norm-2} we have  
	\begin{align}
		|\mrm{Tr}_{K'/K}(x)|\leq |\mrm{Tr}_{K'/K}(x\ca{O}_{K'})|=|\ak{a}|.
	\end{align}
	We may assume that $\mrm{Tr}_{K'/K}(x)\neq 0$ so that $\ak{a}\neq 0$. Then, for any $\pi\in\ak{m}_K$, as $|\pi\ak{a}|< |\ak{a}|$, we see that $|\pi\ak{a}|< |x\scr{D}_{K'/K}|\leq |\ak{a}|$ by \ref{prop:different-trace} and \ref{rem:prop:different-trace}. Thus, we see that $\pi\ak{a}\subseteq \ak{m}_{K'}x\scr{D}_{K'/K}$ by \eqref{eq:para:notation-norm-4} so that $\ak{m}_K\ak{a}\subseteq \ak{m}_{K'}x\scr{D}_{K'/K}$. This verifies  \eqref{eq:cor:different-trace-1} by \eqref{eq:para:notation-norm-2} and \eqref{eq:para:notation-norm-5}. Moreover, taking $x=1$, we obtain $[K':K]\cdot \ak{m}_K=\mrm{Tr}_{K'/K}(1)\cdot\ak{m}_K\subseteq \ak{m}_K\ak{a}\subseteq \ak{m}_{K'}\scr{D}_{K'/K}$. The conclusion follows from multiplying $\ak{m}_{K'}^{-1}$ on both sides.
\end{proof}

\begin{myrem}\label{rem:different-trace}
	In \ref{cor:different-trace}, assume that $[K':K]$ is invertible in $\ca{O}_K$. If moreover the valuation on $K$ is non-discrete (so that $|\ak{m}_K|=|\ak{m}_{K'}|=1$), then we see that $|\scr{D}_{K'/K}|=1$ and thus $\ca{O}_{K'}$ is almost finite \'etale over $\ca{O}_K$. On the other hand, if the valuation on $K$ is discrete, then we see that $|\scr{D}_{K'/K}|\geq |\ak{m}_K\ak{m}_{K'}^{-1}|$.
\end{myrem}

\begin{mycor}\label{cor:coh-torsion}
	Let $K\to K'$ be a finite Galois extension of Henselian valuation fields of height $1$, $\pi\in\ca{O}_K$, $q\in\bb{N}_{>0}$. Then, the cokernel of the canonical injection $\ca{O}_K/\pi\ca{O}_K\to H^0(\gal(K'/K),\ca{O}_{K'}/\pi\ca{O}_{K'})$ and the $q$-th Galois cohomology group $H^q(\gal(K'/K),\ca{O}_{K'}/\pi\ca{O}_{K'})$ are both killed by $\ca{O}_K\cap \widetilde{\ak{m}}_K\scr{D}_{K'/K}$.
\end{mycor}
\begin{proof}
	Firstly, note that $\ca{O}_K/\pi\ca{O}_K\to \ca{O}_{K'}/\pi\ca{O}_{K'}$ is injective whose image contains $\mrm{Tr}_{K'/K}(\ca{O}_{K'}/\pi\ca{O}_{K'})$ (as $\mrm{Tr}_{K'/K}(\ca{O}_{K'})\subseteq \ca{O}_K$). Therefore, any element of $\mrm{Tr}_{K'/K}(\ca{O}_{K'})$ kills $H^q(\gal(K'/K),\ca{O}_{K'}/\pi\ca{O}_{K'})$ and $\cok(\ca{O}_K/\pi\ca{O}_K\to H^0(\gal(K'/K),\ca{O}_{K'}/\pi\ca{O}_{K'}))$ by \cite[\Luoma{5}.12.1.(1)]{abbes2016p}. Let $\ak{a}\subseteq \ca{O}_K$ be the ideal generated by $\mrm{Tr}_{K'/K}(\ca{O}_{K'})$. Then, $\widetilde{\ak{m}}_K\scr{D}_{K'/K}\subseteq\ak{a}\ca{O}_{K'}$ by \ref{prop:different-trace}, which completes the proof.
\end{proof}

\begin{myprop}[{\cite[4.15]{he2024perfd}, cf. \cite[\textsection3.1, Proposition 6]{tate1967p}}]\label{prop:trace-tau}
	Let $K\to K'$ be a finite Galois extension of Henselian valuation fields of height $1$ extensions of $\bb{Q}_p$ with Galois group isomorphic to $\bb{Z}/p^n\bb{Z}$ for some $n\in\bb{N}$, $\tau\in\gal(K'/K)$ a generator, $|\cdot|:K'\to\bb{R}_{\geq0}$ an associated ultrametric absolute value. Then, for any $x\in K'$, we have
	\begin{align}\label{eq:prop:trace-tau-1}
		|p^{-n}\mrm{Tr}_{K'/K}(x)-x|\leq |p^{-(n+1)}\ak{m}_K^{-1}\ak{m}_{K'}\scr{D}_{K'/K}|\cdot |\tau(x)-x|.
	\end{align}
\end{myprop}
\begin{proof}	
	We take induction on $n\in\bb{N}$. The case when $n=0$ is clear. For $n\geq 1$, let $K''$ be the subfield of $K'$ fixed by $\tau^{p^{n-1}}$ so that $\gal(K''/K)\cong \bb{Z}/p^{n-1}\bb{Z}$ (which is still generated by $\tau$). We have
	\begin{align}\label{eq:prop:trace-tau-2}
		&|p^{-n}\mrm{Tr}_{K'/K}(x)-x|\\
		=\ &|p^{-(n-1)}\mrm{Tr}_{K''/K}(p^{-1}\mrm{Tr}_{K'/K''}(x))-p^{-1}\mrm{Tr}_{K'/K''}(x)+p^{-1}\mrm{Tr}_{K'/K''}(x)-x|\nonumber\\
		\leq\ &\max\{|p^{-(n-1)}\mrm{Tr}_{K''/K}(p^{-1}\mrm{Tr}_{K'/K''}(x))-p^{-1}\mrm{Tr}_{K'/K''}(x)|,\ |p^{-1}\mrm{Tr}_{K'/K''}(x)-x|\}.\nonumber
	\end{align}
	By induction hypothesis, we have
	\begin{align}\label{eq:prop:trace-tau-3}
		&|p^{-(n-1)}\mrm{Tr}_{K''/K}(p^{-1}\mrm{Tr}_{K'/K''}(x))-p^{-1}\mrm{Tr}_{K'/K''}(x)|\\
		\leq\ &|p^{-n}\ak{m}_K^{-1}\ak{m}_{K''}\scr{D}_{K''/K}|\cdot|(\tau-1)(p^{-1}\mrm{Tr}_{K'/K''}(x))|\nonumber\\
		=\ &|p^{-n}\ak{m}_K^{-1}\ak{m}_{K''}\scr{D}_{K''/K}|\cdot|p^{-1}\mrm{Tr}_{K'/K''}(\tau(x)-x)|\quad \trm{(as $\gal(K'/K'')$ is commutative)}\nonumber\\
		\leq \ &|p^{-n}\ak{m}_K^{-1}\ak{m}_{K''}\scr{D}_{K''/K}|\cdot|p^{-1}\ak{m}_{K''}^{-1}\ak{m}_{K'}\scr{D}_{K'/K''}|\cdot|\tau(x)-x|\quad \trm{(by \eqref{eq:cor:different-trace-1})}\nonumber\\
		= \ &|p^{-(n+1)}\ak{m}_K^{-1}\ak{m}_{K'}\scr{D}_{K'/K}|\cdot|\tau(x)-x|,\nonumber
	\end{align}
	where the last equality follows from the relation of the almost ideals $\scr{D}_{K'/K}^\al=\scr{D}_{K'/K''}^\al\cdot\scr{D}_{K''/K}^\al$ (\cite[4.1.25]{gabber2003almost}, see also \cite[(4.15.3)]{he2024perfd}).
	
	On the other hand,
	\begin{align}\label{eq:prop:trace-tau-4}
		&|p^{-1}\mrm{Tr}_{K'/K''}(x)-x|=|p^{-1}\sum_{i=1}^{p}(\tau^{ip^{n-1}}-1)(x)|\\
		=\ &|p^{-1}\sum_{i=1}^{p}(\tau^{ip^{n-1}-1}+\cdots+\tau+1)(\tau-1)(x)|\nonumber\\
		\leq \ &|p^{-1}(\tau(x)-x)|\leq |p^{-(n+1)}\ak{m}_K^{-1}\ak{m}_{K'}\scr{D}_{K'/K}|\cdot|\tau(x)-x|,\nonumber
	\end{align}
	where the last inequality follows from \ref{cor:different-trace}. Combining with \eqref{eq:prop:trace-tau-2}, we see that 
	\begin{align}\label{eq:prop:trace-tau-5}
		|p^{-n}\mrm{Tr}_{K'/K}(x)-x|\leq |p^{-(n+1)}\ak{m}_K^{-1}\ak{m}_{K'}\scr{D}_{K'/K}|\cdot|\tau(x)-x|
	\end{align} 
	which completes the induction.
\end{proof}

\begin{mypara}\label{para:fitting-ideal}
	Let $K$ be a valuation field of height $1$. Recall that an $\ca{O}_K$-module $M$ is \emph{uniformly almost finitely generated} if there exists an integer $n\geq 0$ such that for any $\pi\in\widetilde{\ak{m}}_K$, there exists a submodule $M_\pi\subseteq M$ generated by $n$ elements and containing $\pi M$ (\cite[2.1.6, 2.3.9.(\luoma{1})]{gabber2003almost}). In this case, we say that $n$ is a \emph{uniform bound} for $M$. To the associated almost module $M^{\al}$, Gabber-Ramero \cite[2.3.24]{gabber2003almost} associate the \emph{Fitting ideals} $\{F_i(M^{\al})\}_{i\in\bb{N}}$ (which are ideals of the associated almost ring $\ca{O}_K^{\al}$) by a density argument. As $\ca{O}_K$ is a valuation ring, we actually have (\cite[6.3.6]{gabber2003almost}, cf. \cite[3.3]{he2024perfd})
	\begin{align}\label{eq:para:fitting-ideal-1}
		F_i(M^{\al})=(\prod_{j=i+1}^{\infty} \mrm{Ann}_{\ca{O}_K}(\wedge^j M))^{\al},
	\end{align}
	where $\mrm{Ann}_{\ca{O}_K}(\wedge^j M)$ is the annihilator of the $j$-th exterior power of the $\ca{O}_K$-module $M$. Notice that $\mrm{Ann}_{\ca{O}_K}(\wedge^j M)$ is equal to $\widetilde{\ak{m}}_K$ or $\ca{O}_K$ for every $j>n$ so that the infinite product makes sense by \eqref{eq:para:notation-almost-1}. Therefore, we put 
	\begin{align}\label{eq:para:fitting-ideal-2}
		F_i(M)=\prod_{j=i+1}^{\infty} \mrm{Ann}_{\ca{O}_K}(\wedge^j M)
	\end{align}
	called the $i$-th \emph{Fitting ideal} of the $\ca{O}_K$-module $M$.
	
	Moreover, we remark that for an exact sequence $0\to M_1\to M_2\to M_3\to 0$ of uniformly almost finitely generated $\ca{O}_K$-modules, we have an equality of the ideals of the almost ring $\ca{O}_K^{\al}$ (\cite[6.3.5.(\luoma{2})]{gabber2003almost})
	\begin{align}\label{eq:para:fitting-ideal-3}
		F_0(M_2^{\al})=F_0(M_1^{\al})\cdot F_0(M_3^{\al}).
	\end{align} 
	Then, we summarize Gabber-Ramero's computation of cotangent complexes of valuation ring extensions that will be used in the rest of this article.
\end{mypara}

\begin{mythm}[Gabber-Ramero]\label{thm:differential}
	Let $K\to F$ be an extension of valuation fields.
	\begin{enumerate}
		\renewcommand{\labelenumi}{{\rm(\theenumi)}}
		\item If $K$ is perfect or $F$ is algebraic separable over $K$, then the canonical augmentation of cotangent complex $\bb{L}_{\ca{O}_F/\ca{O}_K}\to \Omega^1_{\ca{O}_F/\ca{O}_K}$ is an isomorphism.\label{item:thm:differential-1}
		\item If $K$ is algebraically closed, then $\Omega^1_{\ca{O}_F/\ca{O}_K}$ is torsion-free.\label{item:thm:differential-2}
		\item If $K\to F$ is a finite separable extension of Henselian valuation fields of height $1$, then $\Omega^1_{\ca{O}_F/\ca{O}_K}$ is uniformly almost finitely generated over $\ca{O}_F$ and we have $\scr{D}_{F/K}^{\al}=F_0(\Omega^{1,\al}_{\ca{O}_F/\ca{O}_K})$.\label{item:thm:differential-3}
	\end{enumerate}
\end{mythm}
\begin{proof}
	(\ref{item:thm:differential-1}) is proved in \cite[6.5.12, 6.3.32]{gabber2003almost}; (\ref{item:thm:differential-2}) is proved in \cite[6.5.20]{gabber2003almost}; and (\ref{item:thm:differential-3}) is proved in \cite[6.3.8, 6.3.23]{gabber2003almost}.
\end{proof}

\begin{mycor}\label{cor:differential}
	Let $K\to F$ be an extension of valuation fields with $K$ pre-perfectoid {\rm(\ref{para:notation-perfd})}. Then, $\Omega^1_{\ca{O}_F/\ca{O}_K}$ is torsion-free.
\end{mycor}
\begin{proof}
	After replacing $K$ and $F$ by their Henselization (\cite[6.1.12.(\luoma{6})]{gabber2003almost}), we may assume that they are Henselian. Let $F^{\mrm{s}}$ be an algebraic separable closure of $F$ and let $K^{\mrm{s}}$ be the algebraic separable closure of $K$ contained in $F^{\mrm{s}}$. As $K$ is pre-perfectoid, it is perfect (\cite[5.24]{he2024coh}) so that $K^{\mrm{s}}$ is algebraically closed. Moreover, we have $\Omega^1_{\ca{O}_{K^{\mrm{s}}}/\ca{O}_K}=0$ by \cite[6.6.1, 6.6.6]{gabber2003almost}. As $H_1(\bb{L}_{\ca{O}_{F^{\mrm{s}}}/\ca{O}_F})=0$ by \ref{thm:differential}.(\ref{item:thm:differential-1}), we see that
	\begin{align}
		\ca{O}_{F^{\mrm{s}}}\otimes_{\ca{O}_F}\Omega^1_{\ca{O}_F/\ca{O}_K}\subseteq \Omega^1_{\ca{O}_{F^{\mrm{s}}}/\ca{O}_K}=\Omega^1_{\ca{O}_{F^{\mrm{s}}}/\ca{O}_{K^{\mrm{s}}}},
	\end{align}
	which is torsion-free by \ref{thm:differential}.(\ref{item:thm:differential-2}). Hence, so is $\Omega^1_{\ca{O}_F/\ca{O}_K}$.
\end{proof}

\begin{mylem}\label{lem:rhom}
	Let $K\to L\to F$ be extensions of valuation fields with $H_1(\bb{L}_{\ca{O}_F/\ca{O}_L})=0$ (e.g., $L$ is perfect or $F$ is algebraic separable over $L$ by {\rm\ref{thm:differential}.(\ref{item:thm:differential-1})}). Then, for any nonzero element $\pi\in\ca{O}_F$, there is a canonical exact sequence
	\begin{align}\label{eq:lem:rhom-1}
		0\to& \ca{O}_F\otimes_{\ca{O}_L}\Omega^1_{\ca{O}_L/\ca{O}_K}[\pi]\to\Omega^1_{\ca{O}_F/\ca{O}_K}[\pi]\to \Omega^1_{\ca{O}_F/\ca{O}_L}[\pi]\to\\ &\ca{O}_F\otimes_{\ca{O}_L}\Omega^1_{\ca{O}_L/\ca{O}_K}/\pi\to\Omega^1_{\ca{O}_F/\ca{O}_K}/\pi\to \Omega^1_{\ca{O}_F/\ca{O}_L}/\pi\to 0.\nonumber
	\end{align}
\end{mylem}
\begin{proof}
	Note that for any $\ca{O}_F$-module $M$, $\rr\ho_{\ca{O}_F}(\pi^{-1}\ca{O}_F/\ca{O}_F,M)$ is represented by $\ho_{\ca{O}_F}(\pi^{-1}\ca{O}_F,M)\to \ho_{\ca{O}_F}(\ca{O}_F,M)$ and thus by $M\stackrel{\cdot \pi}{\longrightarrow}M$, whose $0$-th cohomology is $M[\pi]$ and first cohomology is $M/\pi M$. Then, the conclusion follows from applying the derived functor $\rr\ho_{\ca{O}_F}(\pi^{-1}\ca{O}_F/\ca{O}_F,-)$ to the exact sequence $0\to \ca{O}_F\otimes_{\ca{O}_L}\Omega^1_{\ca{O}_L/\ca{O}_K}\to\Omega^1_{\ca{O}_F/\ca{O}_K}\to \Omega^1_{\ca{O}_F/\ca{O}_L}\to 0$.
\end{proof}

\begin{myprop}\label{prop:rhom}
	Let $K\to L\to F$ be extensions of valuation fields with $L$ pre-perfectoid. Then, for any nonzero element $\pi\in \ca{O}_L$, there are canonical isomorphisms
	\begin{align}
		\beta_{F/L/K}:\ca{O}_{F}\otimes_{\ca{O}_L}\Omega^1_{\ca{O}_L/\ca{O}_K}[\pi]&\iso \Omega^1_{\ca{O}_{F}/\ca{O}_K}[\pi],\label{eq:prop:rhom-1}\\
		\gamma_{F/L/K}:\Omega^1_{\ca{O}_F/\ca{O}_K}/\pi\Omega^1_{\ca{O}_F/\ca{O}_K}&\iso \Omega^1_{\ca{O}_{F}/\ca{O}_L}/\pi\Omega^1_{\ca{O}_F/\ca{O}_L}.\label{eq:prop:rhom-2}
	\end{align}
\end{myprop}
\begin{proof}
	It follows from \ref{lem:rhom} and the facts that $\Omega^1_{\ca{O}_{F}/\ca{O}_L}$ is torsion-free (\ref{cor:differential}) and $\Omega^1_{\ca{O}_L/\bb{Z}_p}$ is $\pi$-divisible (\cite[6.6.6]{gabber2003almost}) as $L$ is pre-perfectoid.
\end{proof}

\begin{mythm}\label{thm:ramified-bc}
	Let $K\to F$ be an extension of perfect Henselian valuation fields of height $1$, and let $K'$ be a finite field extension of $K$ such that $F'=K'\otimes_KF$ is a field. 
	\begin{align}\label{eq:thm:ramified-bc-1}
		\xymatrix{
			F'&F\ar[l]\\
			K'\ar[u]&K\ar[l]\ar[u]
		}
	\end{align}
	Assume that there exists $\pi\in\ca{O}_F$ such that any torsion element of $\Omega^1_{\ca{O}_F/\ca{O}_K}$ is killed by $\pi$. Then, for any $\varpi\in\ca{O}_{F'}$ such that $\mrm{Ann}_{\ca{O}_{F'}}(\ca{O}_{F'}\otimes_{\ca{O}_{K'}}\Omega^1_{\ca{O}_{K'}/\ca{O}_K})\subseteq \varpi\ca{O}_{F'}$, we have $\scr{D}_{F'/F}\subseteq \varpi\pi^{-1}\ca{O}_{F'}$. In particular,
	\begin{align}
		\pi\varpi^{-1}[K':K]\cdot \ak{m}_K\ak{m}_{K'}^{-1}\cdot \widetilde{\ak{m}}_{F'}\subseteq \ca{O}_{K'}\otimes_{\ca{O}_K}\ca{O}_F\subseteq \ca{O}_{F'}.
	\end{align}
\end{mythm}
\begin{proof}
	By \ref{thm:differential}.(\ref{item:thm:differential-1}), we see that $\ca{O}_{F'}\otimes_{\ca{O}_F}\Omega^1_{\ca{O}_F/\ca{O}_K}$ and $\ca{O}_{F'}\otimes_{\ca{O}_{K'}}\Omega^1_{\ca{O}_{K'}/\ca{O}_K}$ canonically identify with $\ca{O}_{F'}$-submodules of $\Omega^1_{\ca{O}_{F'}/\ca{O}_K}$. Let $D$ be their intersection. Since $\Omega^1_{\ca{O}_{K'}/\ca{O}_K}$ is torsion, we have $\pi D=0$ by assumption. Let $Q$ be the quotient of $\ca{O}_{F'}\otimes_{\ca{O}_{K'}}\Omega^1_{\ca{O}_{K'}/\ca{O}_K}$ by $D$. Then, there is a canonical injective morphism of exact sequences by \ref{thm:differential}.(\ref{item:thm:differential-1}),
	\begin{align}
		\xymatrix{
			0\ar[r]& D\ar[d]\ar[r]&\ca{O}_{F'}\otimes_{\ca{O}_{K'}}\Omega^1_{\ca{O}_{K'}/\ca{O}_K}\ar[d]\ar[r]&Q\ar[d]\ar[r]&0\\
			0\ar[r]&\ca{O}_{F'}\otimes_{\ca{O}_F}\Omega^1_{\ca{O}_F/\ca{O}_K}\ar[r]&\Omega^1_{\ca{O}_{F'}/\ca{O}_K}\ar[r]&\Omega^1_{\ca{O}_{F'}/\ca{O}_F}\ar[r]&0.
		}
	\end{align}
	As $\pi D=0$, we see that 
	\begin{align}
		\pi \mrm{Ann}_{\ca{O}_{F'}}(Q)\subseteq \mrm{Ann}_{\ca{O}_{F'}}(\ca{O}_{F'}\otimes_{\ca{O}_{K'}}\Omega^1_{\ca{O}_{K'}/\ca{O}_K})\subseteq \varpi\ca{O}_{F'},
	\end{align}
	which implies that $\mrm{Ann}_{\ca{O}_{F'}}(Q)\subseteq \varpi\pi^{-1}\ca{O}_{F'}$.
	
	On the other hand, notice that $Q$ is uniformly almost finitely generated over $\ca{O}_{F'}$ as a quotient of $\ca{O}_{F'}\otimes_{\ca{O}_{K'}}\Omega^1_{\ca{O}_{K'}/\ca{O}_K}$ by \ref{thm:differential}.(\ref{item:thm:differential-3}). Thus, we have
	\begin{align}
		\scr{D}^\al_{F'/F}=F_0(\Omega^{1,\al}_{\ca{O}_{F'}/\ca{O}_F})\subseteq F_0(Q^\al)\subseteq \mrm{Ann}_{\ca{O}_{F'}}(Q^\al)
	\end{align}
	by \ref{thm:differential}.(\ref{item:thm:differential-3}) and \eqref{eq:para:fitting-ideal-1}. In particular, $\widetilde{\ak{m}}_{F'}\scr{D}_{F'/F}\subseteq \mrm{Ann}_{\ca{O}_{F'}}(Q)\subseteq \varpi\pi^{-1}\ca{O}_{F'}$ and thus $\scr{D}_{F'/F}\subseteq \varpi\pi^{-1}\ca{O}_{F'}$ by \eqref{eq:para:notation-norm-3}. Therefore, we have
	\begin{align}
		\pi\varpi^{-1}[K':K]\cdot \ak{m}_K\ak{m}_{K'}^{-1}\cdot \scr{D}_{F'/F}\subseteq \pi\varpi^{-1} \cdot\scr{D}_{K'/K}\cdot\scr{D}_{F'/F}\subseteq \scr{D}_{K'/K} \ca{O}_{F'},
	\end{align}
	where the first inclusion follows from \ref{cor:different-trace}. Hence, the conclusion follows from \ref{prop:ramified-bc}.
\end{proof}

\section{Arithmetic and Geometric Valuation Fields}\label{sec:ari-geo}
In this section, we introduce two types of valuation fields in \ref{defn:ari-geo}: arithmetic and geometric. We discuss their basic examples and properties. In particular, we introduce differential Tate twists for non-arithmetic valuation fields in \ref{defn:int-tate-twist}. In the end, we prove that the valuation fields finitely generated over $\bb{Q}_p$ with transcendental degree $\leq 1$ are either arithmetic and geometric (see \ref{prop:trans-ari-geo}).

\begin{mypara}\label{para:notation-tate-mod}
	For any abelian group $M$, we put
	\begin{align}
		T_p(M)=&\ho_{\bb{Z}}(\bb{Z}[1/p]/\bb{Z},M)=\plim_{x\mapsto px}M[p^r],\label{eq:para:notation-Tate-mod-1}\\
		V_p(M)=&\ho_{\bb{Z}}(\bb{Z}[1/p],M)=\plim_{x\mapsto px}M.\label{eq:para:notation-Tate-mod-2}
	\end{align}
	Note that $T_p(M)$ is a $p$-adically complete $\bb{Z}_p$-module (\cite[\href{https://stacks.math.columbia.edu/tag/0G1Q}{0G1Q}]{stacks-project}). Moreover, if $M=M[p^\infty]$, then we have $V_p(M)=T_p(M)[1/p]$ by applying $\ho_{\bb{Z}}(-,M)[1/p]$ to the exact sequence $0\to \bb{Z}\to \bb{Z}[1/p]\to \bb{Z}[1/p]/\bb{Z}\to 0$. 
\end{mypara}

\begin{mypara}\label{para:notation-tate-twist}
	Let $\overline{\bb{Q}}_p$ be an algebraic closure of $\bb{Q}_p$. We put
	\begin{align}
		\bb{Z}_p(1)=T_p(\overline{\bb{Q}}_p^\times)=\lim_{r\in\bb{N}}\overline{\bb{Q}}_p^\times[p^r],
	\end{align}
	which is a free $\bb{Z}_p$-module of rank $1$. Indeed, any compatible system of primitive $p$-power roots of unity $(\zeta_{p^n})_{n\in \bb{N}}$ in $\overline{\bb{Q}}_p$ (i.e., $\zeta_{p^{n+1}}^p=\zeta_{p^n}$, $\zeta_1=1$, $\zeta_p\neq 1$) forms a basis of $\bb{Z}_p(1)$. For any $\bb{Z}_p$-module $M$ and $n\in \bb{Z}$, we set $M(n)=M\otimes_{\bb{Z}_p}\bb{Z}_p(1)^{\otimes n}$, the $n$-th \emph{Tate twist} of $M$.
\end{mypara}

\begin{mypara}\label{para:ari-geo}
	In the rest of this section, for any Henselian valuation field $F$ extension of $\bb{Q}_p$, we fix an algebraic closure $\overline{F}$ of $F$. Let $(\zeta_{p^n})_{n\in\bb{N}}$ be a compatible system of primitive $p$-power roots of unity contained in $\overline{F}$. For any $n\in\bb{N}$, we put
\begin{align}\label{eq:para:ari-geo-1}
	F_n=F(\zeta_{p^n}),\quad F_\infty&=\bigcup_{n\in\bb{N}}F_n,
\end{align}
which do not depend on the choice of $(\zeta_{p^n})_{n\in\bb{N}}$. Note that $F_0=F$. We put $G_{F_n}=\gal(\overline{F}/F_n)$ the absolute Galois group of $F_n$ for any $n\in\bb{N}\cup\{\infty\}$.
\end{mypara}

\begin{mylem}\label{lem:cyclotomic}
	Let $K$ be an absolutely unramified Henselian discrete valuation field extension of $\bb{Q}_p$, $n\in\bb{N}_{>0}$.
	\begin{enumerate}
		\renewcommand{\labelenumi}{{\rm(\theenumi)}}
		\item The field $K_n=K(\zeta_{p^n})$ is a totally ramified Galois extension of $K$ of degree $p^{n-1}(p-1)$ with uniformizer $\zeta_{p^n}-1$, and $\ca{O}_{K_n}=\ca{O}_K[T]/(\frac{T^{p^n}-1}{T^{p^{n-1}}-1})=\ca{O}_K[\zeta_{p^n}]$.\label{item:lem:cyclotomic-1}
		\item There is a canonical group homeomorphism, called the \emph{cyclotomic character},
		\begin{align}\label{eq:lem:cyclotomic-1}
			\chi:\gal(K_\infty/K)\iso\bb{Z}_p^\times,
		\end{align}
		characterized by $\sigma(\zeta_{p^n})=\zeta_{p^n}^{\chi(\sigma)}$ for any $\sigma\in \gal(K_\infty/K)$. Moreover, it induces an isomorphism of subgroups
		\begin{align}\label{eq:lem:cyclotomic-2}
			\chi:\gal(K_\infty/K_n)\iso 1+p^n\bb{Z}_p.
		\end{align}\label{item:lem:cyclotomic-2}
		\item There is an isomorphism of $\ca{O}_{K_n}$-modules
		\begin{align}\label{eq:lem:cyclotomic-3}
			(p^{-n}\ca{O}_{K_n}/(\zeta_p-1)^{-1}\ca{O}_{K_n})(1)\iso\Omega^1_{\ca{O}_{K_n}/\ca{O}_K},
		\end{align}
		sending $p^{-n}\otimes (\zeta_{p^k})_{k\in\bb{N}}$ to $\df\log(\zeta_{p^n})$.\label{item:lem:cyclotomic-3}
		\item The different ideal $\scr{D}_{K_n/K}$ is generated by $p^n(\zeta_p-1)^{-1}$.\label{item:lem:cyclotomic-4}
	\end{enumerate}
\end{mylem}
\begin{proof}
	(\ref{item:lem:cyclotomic-1}) and (\ref{item:lem:cyclotomic-2}) follow from \cite[\Luoma{1}.\textsection6, Proposition 17]{serre1979local} (cf. \cite[\Luoma{4}.\textsection4, Proposition 17]{serre1979local}). Then, (\ref{item:lem:cyclotomic-3}) follows directly from (\ref{item:lem:cyclotomic-1}), and (\ref{item:lem:cyclotomic-4}) follows from (\ref{item:lem:cyclotomic-3}) and \cite[\Luoma{3}.\textsection6, Corollary 2]{serre1979local} (or \eqref{eq:para:fitting-ideal-1} and \ref{thm:differential}.(\ref{item:thm:differential-3})).
\end{proof}

\begin{myprop}[{\cite[Th\'eor\`eme 1']{fontaine1982formes}}]\label{prop:cyclotomic-diff}
	Let $K$ be a Henselian discrete valuation field extension of $\bb{Q}_p$ with perfect residue field, $K'$ the unique absolutely unramified complete discrete valuation field subextension of $\widehat{K}/\bb{Q}_p$ such that $\widehat{K}$ is a totally ramified finite extension of $K'$ {\rm(\cite[\Luoma{2}.\textsection5, Theorem 4]{serre1979local})}, $\pi_K\in\ca{O}_K$ a generator of the different ideal $\scr{D}_{\widehat{K}/K'}\subseteq \ca{O}_{\widehat{K}}$. Then, there is an isomorphism of $\ca{O}_{K_\infty}$-modules
	\begin{align}\label{eq:prop:cyclotomic-diff-1}
		(K_\infty/\pi_K^{-1}(\zeta_p-1)^{-1}\ca{O}_{K_\infty})(1)\iso\Omega^1_{\ca{O}_{K_\infty}/\ca{O}_K},
	\end{align}  
	sending $p^{-n}\otimes (\zeta_{p^k})_{k\in\bb{N}}$ to $\df\log(\zeta_{p^n})$ for any $n\in\bb{N}$, where $K_\infty=\bigcup_{n\in\bb{N}}K(\zeta_{p^n})$. Moreover, $K_\infty$ is pre-perfectoid.
\end{myprop}
\begin{proof}
	We fix a $\bb{Z}_p$-basis $(\zeta_{p^k})_{k\in\bb{N}}$ of $\bb{Z}_p(1)$. Consider the homomorphism of $\ca{O}_{K_\infty}$-modules $K_\infty\to \Omega^1_{\ca{O}_{K_\infty}/\ca{O}_K}$ defined by sending $p^{-n}$ to $\df\log(\zeta_{p^n})$ for any $n\in\bb{N}$. It suffices to show that it is surjective with kernel generated by $\pi_K^{-1}(\zeta_p-1)^{-1}$. As $K$ is Henselian, $\ca{O}_{\widehat{K}_n}=\ca{O}_{K_n}\otimes_{\ca{O}_K}\ca{O}_{\widehat{K}}$ (\cite[\Luoma{2}.\textsection3, Proposition 4]{serre1979local}, see also \ref{lem:completion}). Thus, after replacing $K$ by $\widehat{K}$, we may assume that $K=\widehat{K}$. 
	
	As $K'_\infty=\bigcup_{n\in\bb{N}}K'(\zeta_{p^n})$ is pre-perfectoid (\ref{lem:cyclotomic}.(\ref{item:lem:cyclotomic-1})), the torsion module $\Omega^1_{\ca{O}_{K_\infty}/\ca{O}_{K'_\infty}}$ is torsion-free by \ref{cor:differential} and thus is zero. We also see that $K_\infty$ is pre-perfectoid by almost purity (\cite[6.6.16]{gabber2003almost}). Notice that $\Omega^1_{\ca{O}_{K_\infty}/\ca{O}_{K'}}=\ca{O}_{K_\infty}\otimes_{\ca{O}_{K'_\infty}}\Omega^1_{\ca{O}_{K'_\infty}/\ca{O}_{K'}}\cong K_\infty/(\zeta_p-1)^{-1}\ca{O}_{K_\infty}$ by \ref{thm:differential}.(\ref{item:thm:differential-1}) and \ref{lem:cyclotomic}.(\ref{item:lem:cyclotomic-3}). As $K$ is a totally ramified finite separable extension of $K'$, $\Omega^1_{\ca{O}_K/\ca{O}_{K'}}$ is generated by one element whose annihilator is the different ideal $\scr{D}_{K/K'}$ (\cite[\Luoma{3}.\textsection7, Proposition 14]{serre1979local}). Therefore, the submodule $\ca{O}_{K_\infty}\otimes_{\ca{O}_K}\Omega^1_{\ca{O}_K/\ca{O}_{K'}}\subseteq \Omega^1_{\ca{O}_{K_\infty}/\ca{O}_{K'}}$ is isomorphic to $\pi_K^{-1}(\zeta_p-1)^{-1}\ca{O}_{K_\infty}/(\zeta_p-1)^{-1}\ca{O}_{K_\infty}$, and thus its quotient $\Omega^1_{\ca{O}_{K_\infty}/\ca{O}_K}$ is isomorphic to $K_\infty/\pi_K^{-1}(\zeta_p-1)^{-1}\ca{O}_{K_\infty}$.
\end{proof}

\begin{myrem}\label{rem:cyclotomic-diff}
	Let $e_K$ be the absolute ramification index of $K$ (i.e., $p\ca{O}_K=\ak{m}_K^{e_K}$). Then, we have $[\widehat{K}:K']=e_K$ and thus $pe_K\ak{m}_{\widehat{K}}^{-1}\subseteq \scr{D}_{\widehat{K}/K'}\subseteq p\ak{m}_{\widehat{K}}^{-1}$ by \ref{cor:different-trace} and \cite[\Luoma{3}.\textsection6, Proposition 13]{serre1979local}. 
\end{myrem}

\begin{mycor}\label{cor:tate-cyclo}
	Let $K$ be a Henselian discrete valuation field extension of $\bb{Q}_p$ with perfect residue field, $F$ a valuation field extension of $K$ containing a compatible system of primitive $p$-power roots of unity $(\zeta_{p^n})_{n\in \bb{N}}$, $\ak{a}_{F/K}=\pi_K^{-1}(\zeta_p-1)^{-1}\ca{O}_F\subseteq F$ (where $\pi_K$ is defined in {\rm\ref{prop:cyclotomic-diff}}). Then, there is a canonical isomorphism of $\ca{O}_F$-modules
	\begin{align}\label{eq:cor:tate-cyclo-1}
		\alpha_{F/K}:(F/\ak{a}_{F/K})(1)\iso \Omega^1_{\ca{O}_F/\ca{O}_K}[p^\infty]
	\end{align}
	sending an element $p^{-r}\otimes (\zeta_{p^n})_{n\in \bb{N}}$ to $\df\log(\zeta_{p^r})$ for any $r\in\bb{N}$. Moreover, it induces a canonical isomorphism
	\begin{align}\label{eq:cor:tate-cyclo-2}
		\alpha_{F/K}:\ak{a}_{F/K}\ca{O}_{\widehat{F}}(1)\iso T_p(\Omega^1_{\ca{O}_F/\ca{O}_K})
	\end{align}
	sending $1\otimes(\zeta_{p^n})_{n\in\bb{N}}$ to $(\df\log(\zeta_{p^n}))_{n\in\bb{N}}$.
\end{mycor}
\begin{proof}
	Note that $\alpha_{K_\infty/K}:(K_\infty/\ak{a}_{K_\infty/K})(1)\to \Omega^1_{\ca{O}_{K_\infty}/\ca{O}_K}[p^\infty]=\Omega^1_{\ca{O}_{K_\infty}/\ca{O}_K}$ is an isomorphism and that $K_\infty$ is pre-perfectoid by \ref{prop:cyclotomic-diff}. Then, \eqref{eq:cor:tate-cyclo-1} follows from the isomorphism \eqref{eq:prop:rhom-1} $\beta_{F/K_\infty/K}:\ca{O}_F\otimes_{\ca{O}_{K_\infty}}\Omega^1_{\ca{O}_{K_\infty}/\ca{O}_K}[p^\infty]\iso \Omega^1_{\ca{O}_F/\ca{O}_K}[p^\infty]$. Finally, the ``in particular" part follows immediately from the identities (where we put $\ak{a}=\ak{a}_{F/K}$)
	\begin{align}\label{eq:cor:tate-cyclo-3}
		T_p(F/\ak{a})=\lim_{r\in\bb{N}}p^{-r}\ak{a}/\ak{a}=\lim_{r\in\bb{N}}\ak{a}/p^r\ak{a}=\ak{a}\ca{O}_{\widehat{F}},
	\end{align}
	where the transition morphisms are the multiplication by $p$ in the first limit and the canonical projection in the second limit, and where the last equality follows from \cite[\href{https://stacks.math.columbia.edu/tag/05GG}{05GG}]{stacks-project}.
\end{proof}

\begin{mycor}\label{cor:disc-diff-torsion}
	Let $K$ be a Henselian discrete valuation field extension of $\bb{Q}_p$. Then, $\Omega^1_{\ca{O}_K/\bb{Z}_p}[p^\infty]$ is killed by $pe_K\ak{m}_K^{-1}$, where $e_K$ is the absolute ramification index of $K$.
\end{mycor}
\begin{proof}
	Firstly, we claim that $\Omega^1_{\ca{O}_K/\bb{Z}_p}[p^\infty]$ is killed by $\pi_K$ if the residue field of $K$ is perfect. Indeed, consider the canonical exact sequence \eqref{eq:lem:rhom-1} $0\to \ca{O}_{K_\infty}\otimes_{\ca{O}_K}\Omega^1_{\ca{O}_K/\bb{Z}_p}[p^r]\to \Omega^1_{\ca{O}_{K_\infty}/\bb{Z}_p}[p^r]\to \Omega^1_{\ca{O}_{K_\infty}/\ca{O}_K}[p^r]\to \ca{O}_{K_\infty}\otimes_{\ca{O}_K}\Omega^1_{\ca{O}_K/\bb{Z}_p}/p^r\to \Omega^1_{\ca{O}_{K_\infty}/\bb{Z}_p}/p^r\to \Omega^1_{\ca{O}_{K_\infty}/\ca{O}_K}/p^r\to 0$. Taking filtered colimit over $r\in\bb{N}$, we obtain an exact sequence
	\begin{align}\label{eq:cor:disc-diff-torsion-1}
		0\longrightarrow \ca{O}_{K_\infty}\otimes_{\ca{O}_K}\Omega^1_{\ca{O}_K/\bb{Z}_p}[p^\infty]\longrightarrow \Omega^1_{\ca{O}_{K_\infty}/\bb{Z}_p}[p^\infty]\longrightarrow \Omega^1_{\ca{O}_{K_\infty}/\ca{O}_K}.
	\end{align}
	Note that $\Omega^1_{\ca{O}_{K_\infty}/\bb{Z}_p}[p^\infty]\cong K_\infty/(\zeta_p-1)^{-1}\ca{O}_{K_\infty}$ by \ref{cor:tate-cyclo} and $\Omega^1_{\ca{O}_{K_\infty}/\ca{O}_K}\cong K_\infty/\pi_K^{-1}(\zeta_p-1)^{-1}\ca{O}_{K_\infty}$ by \ref{prop:cyclotomic-diff}. We see that $\ca{O}_{K_\infty}\otimes_{\ca{O}_K}\Omega^1_{\ca{O}_K/\bb{Z}_p}[p^\infty]\cong \pi_K^{-1}(\zeta_p-1)^{-1}\ca{O}_{K_\infty}/(\zeta_p-1)^{-1}\ca{O}_{K_\infty}$ so that $\Omega^1_{\ca{O}_K/\bb{Z}_p}[p^\infty]$ is killed by $\pi_K$.
	
	In general, there is a complete discrete valuation field $L$ weakly unramified extension of $K$ with perfect residue field (cf. \cite[4.1]{he2021faltingsext}). Therefore, $\ca{O}_L\otimes_{\ca{O}_K}\Omega^1_{\ca{O}_K/\bb{Z}_p}[p^\infty]\subseteq \Omega^1_{\ca{O}_L/\bb{Z}_p}[p^\infty]$ is killed by $pe_L\ak{m}_L^{-1}=pe_K\ak{m}_K^{-1}\ca{O}_L$ by the claim above and \ref{rem:cyclotomic-diff}, and thus $\Omega^1_{\ca{O}_K/\bb{Z}_p}[p^\infty]$ is killed by $pe_K\ak{m}_K^{-1}$.
\end{proof}

\begin{mydefn}\label{defn:ari-geo}
	Let $F$ be a Henselian valuation field of height $1$ extension of $\bb{Q}_p$.
	\begin{enumerate}
		\renewcommand{\labelenumi}{{\rm(\theenumi)}}
		\item We say that $F$ is \emph{arithmetic} if the torsion submodule $\Omega^1_{\ca{O}_F/\bb{Z}_p}[p^\infty]$ is bounded (\ref{defn:sep-bound}.(\ref{item:defn:sep-bound-2})).\label{item:defn:ari-geo-1}
		\item We say that $F$ is \emph{geometric} if $\scr{D}_{F_\infty/F}\neq 0$, where we put $\scr{D}_{F_\infty/F}=\bigcap_{n\in\bb{N}}\scr{D}_{F_n/F}\ca{O}_{F_\infty}\subseteq \ca{O}_{F_\infty}$.\label{item:defn:ari-geo-2}
	\end{enumerate}
\end{mydefn}

\begin{mylem}\label{lem:ari-geo-basic}
	Let $F$ be a Henselian valuation field of height $1$ extension of $\bb{Q}_p$.
	\begin{enumerate}
		\renewcommand{\labelenumi}{{\rm(\theenumi)}}
		\item If the valuation on $F$ is discrete, then $F$ is arithmetic.\label{item:lem:ari-geo-basic-1}
		\item If there is a pre-perfectoid field $K$ valuation subextension of $F/\bb{Q}_p$, then $F$ is geometric.\label{item:lem:ari-geo-basic-2}
	\end{enumerate}
\end{mylem}
	\begin{proof}
	(\ref{item:lem:ari-geo-basic-1}) It follows directly from \ref{cor:disc-diff-torsion}.

	(\ref{item:lem:ari-geo-basic-2}) Note that $\ca{O}_{K_n}$ is almost finite \'etale over $\ca{O}_K$ for any $n\in\bb{N}$ (\cite[6.6.2]{gabber2003almost}). Taking base change along $\ca{O}_K\to\ca{O}_F$ and normalization in $F_n$, we see that $\ca{O}_{F_n}$ is also almost finite \'etale over $\ca{O}_F$ by \cite[\Luoma{5}.7.11]{abbes2016p}. Thus, $\widetilde{\ak{m}}_{F_n}\subseteq\scr{D}_{F_n/F}$ (\ref{para:different}) for any $n\in\bb{N}$ so that $F$ is geometric.
\end{proof}

\begin{mylem}\label{lem:ari-geo-ex}
	Let $F$ be a Henselian valuation field of height $1$ extension of $\bb{Q}_p$. If $F$ is arithmetic, then $F$ is not geometric (or equivalently, if $F$ is geometric, then $F$ is not arithmetic).
\end{mylem}
\begin{proof}
	Consider the canonical exact sequence \eqref{eq:lem:rhom-1} $0\to \ca{O}_{F_\infty}\otimes_{\ca{O}_F}\Omega^1_{\ca{O}_F/\bb{Z}_p}[p^r]\to \Omega^1_{\ca{O}_{F_\infty}/\bb{Z}_p}[p^r]\to \Omega^1_{\ca{O}_{F_\infty}/\ca{O}_F}[p^r]$. Taking filtered colimit over $r\in\bb{N}$, we obtain an exact sequence
	\begin{align}\label{eq:lem:ari-geo-ex-1}
		0\longrightarrow \ca{O}_{F_\infty}\otimes_{\ca{O}_F}\Omega^1_{\ca{O}_F/\bb{Z}_p}[p^\infty]\longrightarrow \Omega^1_{\ca{O}_{F_\infty}/\bb{Z}_p}[p^\infty]\longrightarrow \Omega^1_{\ca{O}_{F_\infty}/\ca{O}_F}.
	\end{align}
	Suppose that $F$ is arithmetic. As $\Omega^1_{\ca{O}_{F_\infty}/\bb{Z}_p}[p^\infty]\cong F_\infty/\ca{O}_{F_\infty}$ by \ref{cor:tate-cyclo}, the boundedness of $\Omega^1_{\ca{O}_F/\bb{Z}_p}[p^\infty]$ implies that $\ca{O}_{F_\infty}\otimes_{\ca{O}_F}\Omega^1_{\ca{O}_F/\bb{Z}_p}[p^\infty]$ is contained in $p^{-r}\ca{O}_{F_\infty}/\ca{O}_{F_\infty}$ for some $r\in\bb{N}$ under this isomorphism. In this case, the torsion module $\Omega^1_{\ca{O}_{F_\infty}/\ca{O}_F}$ is not bounded. Suppose moreover that $F$ is geometric. Then, $\Omega^1_{\ca{O}_{F_\infty}/\ca{O}_F}=\colim_{n\in\bb{N}}\Omega^1_{\ca{O}_{F_n}/\ca{O}_F}$ is annihilated by $\widetilde{\ak{m}}_F\scr{D}_{F_\infty/F}\neq 0$ (\ref{thm:differential}.(\ref{item:thm:differential-3})), which is a contradiction.
\end{proof}

\begin{mylem}\label{lem:ari-disc}
	Let $K$ be a Henselian discrete valuation field extension of $\bb{Q}_p$ with perfect residue field, $F$ a Henselian valuation field of height $1$ extension of $K$. Then, $F$ is arithmetic if and only if $\Omega^1_{\ca{O}_F/\ca{O}_K}[p^\infty]$ is bounded. Moreover, if $F$ is not arithmetic, then $\ca{O}_{F_\infty}\otimes_{\ca{O}_F}\Omega^1_{\ca{O}_F/\ca{O}_K}[p^\infty]=\Omega^1_{\ca{O}_{F_\infty}/\ca{O}_K}[p^\infty]\cong F_\infty/\ca{O}_{F_\infty}$ and $\Omega^1_{\ca{O}_{F_\infty}/\ca{O}_F}$ is a quotient of $\Omega^1_{\ca{O}_{F_\infty}/\ca{O}_{K_\infty}}$.
\end{mylem}
\begin{proof}
	Firstly, we claim that $\Omega^1_{\ca{O}_F/\ca{O}_K}[p^\infty]$ is unbounded if and only if $\ca{O}_{F_\infty}\otimes_{\ca{O}_F}\Omega^1_{\ca{O}_F/\ca{O}_K}[p^\infty]=\Omega^1_{\ca{O}_{F_\infty}/\ca{O}_K}[p^\infty]$. Indeed, consider the canonical morphisms
	\begin{align}\label{eq:lem:ari-disc-1}
		\xymatrix{
			\ca{O}_{F_\infty}\otimes_{\ca{O}_F}\Omega^1_{\ca{O}_F/\ca{O}_K}[p^\infty]\ar[r]& \Omega^1_{\ca{O}_{F_\infty}/\ca{O}_K}[p^\infty]& (F_\infty/\ak{a}_{F_\infty/K})(1)\ar[l]^-{\alpha_{F_\infty/K}}_-{\sim},
		}
	\end{align}
	where the first morphism is injective by \ref{thm:differential}.(\ref{item:thm:differential-1}) and the second isomorphism is \eqref{eq:cor:tate-cyclo-1}. Thus, we see that $\Omega^1_{\ca{O}_F/\ca{O}_K}[p^\infty]$ is unbounded if and only if the first morphism is an isomorphism.
	
	Then, we check that if $\ca{O}_{F_\infty}\otimes_{\ca{O}_F}\Omega^1_{\ca{O}_F/\ca{O}_K}[p^\infty]=\Omega^1_{\ca{O}_{F_\infty}/\ca{O}_K}[p^\infty]$, then $\Omega^1_{\ca{O}_{F_\infty}/\ca{O}_F}$ is a quotient of $\Omega^1_{\ca{O}_{F_\infty}/\ca{O}_{K_\infty}}$. Consider the canonical exact sequence (\ref{thm:differential}.(\ref{item:thm:differential-1}))
	\begin{align}
		0\longrightarrow \ca{O}_{F_\infty}\otimes_{\ca{O}_{K_\infty}}\Omega^1_{\ca{O}_{K_\infty}/\ca{O}_K}\longrightarrow \Omega^1_{\ca{O}_{F_\infty}/\ca{O}_K}\longrightarrow \Omega^1_{\ca{O}_{F_\infty}/\ca{O}_{K_\infty}}\longrightarrow 0.
	\end{align}
	As $K_\infty$ is pre-perfectoid (\ref{prop:cyclotomic-diff}), $\Omega^1_{\ca{O}_{F_\infty}/\ca{O}_{K_\infty}}$ is torsion-free by \ref{cor:differential}. Therefore, $\ca{O}_{F_\infty}\otimes_{\ca{O}_{K_\infty}}\Omega^1_{\ca{O}_{K_\infty}/\ca{O}_K}=\Omega^1_{\ca{O}_{F_\infty}/\ca{O}_K}[p^\infty]$ is contained in $\ca{O}_{F_\infty}\otimes_{\ca{O}_F}\Omega^1_{\ca{O}_F/\ca{O}_K}$. This implies that the torsion module $\Omega^1_{\ca{O}_{F_\infty}/\ca{O}_F}$ is a quotient of $\Omega^1_{\ca{O}_{F_\infty}/\ca{O}_{K_\infty}}$.
	
	Finally, it remains to show that $\Omega^1_{\ca{O}_F/\bb{Z}_p}[p^\infty]$ is bounded if and only if $\Omega^1_{\ca{O}_F/\ca{O}_K}[p^\infty]$ is bounded.
	For any $r\in\bb{N}$, consider the canonical exact sequence \eqref{eq:lem:rhom-1}
	\begin{align}
		0\to \ca{O}_F\otimes_{\ca{O}_K}\Omega^1_{\ca{O}_K/\bb{Z}_p}[p^r]\to \Omega^1_{\ca{O}_F/\bb{Z}_p}[p^r]\to \Omega^1_{\ca{O}_F/\ca{O}_K}[p^r]\to \ca{O}_F\otimes_{\ca{O}_K}\Omega^1_{\ca{O}_K/\bb{Z}_p}/p^r.
	\end{align}
	With the notation in \ref{prop:cyclotomic-diff}, note that $\Omega^1_{\ca{O}_K/\bb{Z}_p}[p^\infty]$ is killed by $\pi_K$ by (the proof of) \ref{cor:disc-diff-torsion} and that $\widehat{\Omega}^1_{\ca{O}_K/\bb{Z}_p}=\Omega^1_{\ca{O}_{\widehat{K}}/\ca{O}_{K'}}$ is killed by $\pi_K$  (\cite[3.3]{he2021faltingsext}). Thus, the kernel and cokernel of $\Omega^1_{\ca{O}_F/\bb{Z}_p}[p^\infty]\to \Omega^1_{\ca{O}_F/\ca{O}_K}[p^\infty]$ are both killed by $\pi_K$. Then, we see that the boundedness of $\Omega^1_{\ca{O}_F/\bb{Z}_p}[p^\infty]$ is equivalent to that of $\Omega^1_{\ca{O}_F/\ca{O}_K}[p^\infty]$.
\end{proof}

\begin{myprop}\label{prop:non-ari-flat}
	Let $K$ be a Henselian discrete valuation field extension of $\bb{Q}_p$ with perfect residue field, $F$ a Henselian valuation field of height $1$ extension of $K$, $\widehat{\Omega}^1_{\ca{O}_F/\ca{O}_K}$ the $p$-adic completion of $\Omega^1_{\ca{O}_F/\ca{O}_K}$. Then, $\widehat{\Omega}^1_{\ca{O}_F/\ca{O}_K}[p^\infty]$ is bounded. Moreover, it is zero if $F$ is not arithmetic.
\end{myprop}
\begin{proof}
	Consider the canonical exact sequence 
	\begin{align}
		0\longrightarrow \Omega^1_{\ca{O}_F/\ca{O}_K}[p^\infty]\to \Omega^1_{\ca{O}_F/\ca{O}_K}\longrightarrow \Omega^1_{\ca{O}_F/\ca{O}_K}/\Omega^1_{\ca{O}_F/\ca{O}_K}[p^\infty]\longrightarrow 0.
	\end{align}
	Taking $p$-adic completions, we obtain an exact sequence (\cite[\href{https://stacks.math.columbia.edu/tag/0315}{0315}]{stacks-project})
	\begin{align}
		0\longrightarrow (\Omega^1_{\ca{O}_F/\ca{O}_K}[p^\infty])^\wedge\to \widehat{\Omega}^1_{\ca{O}_F/\ca{O}_K}\longrightarrow (\Omega^1_{\ca{O}_F/\ca{O}_K}/\Omega^1_{\ca{O}_F/\ca{O}_K}[p^\infty])^\wedge\longrightarrow 0.
	\end{align}
	Note that $(\Omega^1_{\ca{O}_F/\ca{O}_K}/\Omega^1_{\ca{O}_F/\ca{O}_K}[p^\infty])^\wedge$ is torsion-free (\cite[5.7.(1)]{tsuji2018localsimpson}).
	
	If $F$ is arithmetic, then $\Omega^1_{\ca{O}_F/\ca{O}_K}[p^\infty]$ is bounded so that $\widehat{\Omega}^1_{\ca{O}_F/\ca{O}_K}[p^\infty]=(\Omega^1_{\ca{O}_F/\ca{O}_K}[p^\infty])^\wedge=\Omega^1_{\ca{O}_F/\ca{O}_K}[p^\infty]$ is also bounded.
	
	If $F$ is not arithmetic, $\ca{O}_{F_\infty}\otimes_{\ca{O}_F}\Omega^1_{\ca{O}_F/\ca{O}_K}[p^\infty]=\Omega^1_{\ca{O}_{F_\infty}/\ca{O}_K}[p^\infty]\cong F_\infty/\ca{O}_{F_\infty}$ by \ref{lem:ari-disc}. This shows that $\Omega^1_{\ca{O}_F/\ca{O}_K}[p^\infty]=p\cdot \Omega^1_{\ca{O}_F/\ca{O}_K}[p^\infty]$ by faithfully flat descent. Thus, $(\Omega^1_{\ca{O}_F/\ca{O}_K}[p^\infty])^\wedge=0$ so that $\widehat{\Omega}^1_{\ca{O}_F/\ca{O}_K}$ is torsion-free.
\end{proof}

\begin{myprop}\label{prop:tate-twist}
	Let $K$ be a Henselian discrete valuation field extension of $\bb{Q}_p$ with perfect residue field, $F$ a Henselian valuation field of height $1$ extension of $K$. Assume that $F$ is not arithmetic. Then, the $\ca{O}_{\widehat{F}}$-module $T_p(\Omega^1_{\ca{O}_F/\ca{O}_K})$	is finite free of rank $1$ with 
	\begin{align}\label{eq:prop:tate-twist-0}
		T_p(\Omega^1_{\ca{O}_F/\ca{O}_K})/p^nT_p(\Omega^1_{\ca{O}_F/\ca{O}_K})=\Omega^1_{\ca{O}_F/\ca{O}_K}[p^n]
	\end{align}
	for any $n\in\bb{N}$. Moreover, for any Henselian valuation field $F'$ of height $1$ extension of $F$, the canonical morphism
	\begin{align}\label{eq:prop:tate-twist-1}
		\ca{O}_{\widehat{F'}}\otimes_{\ca{O}_{\widehat{F}}}T_p(\Omega^1_{\ca{O}_F/\ca{O}_K})\longrightarrow T_p(\Omega^1_{\ca{O}_{F'}/\ca{O}_K})
	\end{align}
	is an isomorphism.
\end{myprop}
\begin{proof}
	If $F$ contains a compatible system of primitive $p$-power roots of unity $(\zeta_{p^n})_{n\in \bb{N}}$, then the conclusion follows directly from \ref{cor:tate-cyclo}. Therefore, it remains to check the case where $F'=F_\infty=\bigcup_{n\in\bb{N}}F(\zeta_{p^n})$. As $F$ is not arithmetic, we have
	\begin{align}
		\ca{O}_{F_\infty}\otimes_{\ca{O}_F}\Omega^1_{\ca{O}_F/\ca{O}_K}[p^n]= \Omega^1_{\ca{O}_{F_\infty}/\ca{O}_K}[p^n]
	\end{align}
	for any $n\in\bb{N}$ by \ref{lem:ari-disc}. As the latter is isomorphic to $(p^{-n}\ak{a}_{F_\infty/K}/\ak{a}_{F_\infty/K})(1)$ \eqref{eq:cor:tate-cyclo-1}, we see that each $\Omega^1_{\ca{O}_F/\ca{O}_K}[p^n]$ is finitely generated and $\Omega^1_{\ca{O}_F/\ca{O}_K}[p^{n+1}]/p^n\Omega^1_{\ca{O}_F/\ca{O}_K}[p^{n+1}]=\Omega^1_{\ca{O}_F/\ca{O}_K}[p^n]$ by faithfully flat descent. Therefore, $T_p(\Omega^1_{\ca{O}_F/\ca{O}_K})$ is $p$-adically complete and finitely generated over $\ca{O}_{\widehat{F}}$ with $T_p(\Omega^1_{\ca{O}_F/\ca{O}_K})/p^nT_p(\Omega^1_{\ca{O}_F/\ca{O}_K})=\Omega^1_{\ca{O}_F/\ca{O}_K}[p^n]$ (\cite[\href{https://stacks.math.columbia.edu/tag/09B8}{09B8}, \href{https://stacks.math.columbia.edu/tag/031D}{031D}]{stacks-project}). Moreover, by taking limits of the injective morphisms
	\begin{align}
		\Omega^1_{\ca{O}_F/\ca{O}_K}[p^n]\longrightarrow \Omega^1_{\ca{O}_{F_\infty}/\ca{O}_K}[p^n],
	\end{align}
	we see that $T_p(\Omega^1_{\ca{O}_F/\ca{O}_K})\to T_p(\Omega^1_{\ca{O}_{F_\infty}/\ca{O}_K})$ is also injective. As the latter is isomorphic to $\ak{a}_{F_\infty/K}\ca{O}_{\widehat{F_\infty}}(1)$ \eqref{eq:cor:tate-cyclo-2}, we see that $T_p(\Omega^1_{\ca{O}_F/\ca{O}_K})$ is also torsion-free.
	
	Therefore, $T_p(\Omega^1_{\ca{O}_F/\ca{O}_K})$ is a finitely generated torsion-free $\ca{O}_{\widehat{F}}$-module. It is actually finite free by \cite[\Luoma{6}.\textsection3.6, Lemme 1]{bourbaki2006commalg5-7}. Since the morphism of finite free $\ca{O}_{\widehat{F_\infty}}$-modules
	\begin{align}
		\ca{O}_{\widehat{F_\infty}}\otimes_{\ca{O}_{\widehat{F}}}T_p(\Omega^1_{\ca{O}_F/\ca{O}_K})\longrightarrow T_p(\Omega^1_{\ca{O}_{F_\infty}/\ca{O}_K})
	\end{align}
	induces an isomorphism modulo $p$, it is actually an isomorphism by d\'evissage. 
\end{proof}

\begin{mydefn}\label{defn:int-tate-twist}
	Let $F$ be a non-arithmetic Henselian valuation field of height $1$ extension of $\bb{Q}_p$. For any $\ca{O}_{\widehat{F}}$-module $M$ and $n\in\bb{Z}$, we put 
	\begin{align}\label{eq:defn:int-tate-twist-1}
		M\{n\}=M\otimes_{\ca{O}_{\widehat{F}}}T_p(\Omega^1_{\ca{O}_F/\bb{Z}_p})^{\otimes n},
	\end{align}
	called the $n$-th \emph{differential Tate twist} (or more commonly, \emph{Breuil-Kisin-Fargues twist}).
\end{mydefn}

	We note that for any Henselian valuation field $F'$ of height $1$ extension of $F$, by \ref{prop:tate-twist} there is a natural identification 
	\begin{align}\label{eq:defn:int-tate-twist-2}
		(\ca{O}_{\widehat{F'}}\otimes_{\ca{O}_{\widehat{F}}}M)\{n\}=\ca{O}_{\widehat{F'}}\otimes_{\ca{O}_{\widehat{F}}}M\{n\}.
	\end{align}
	Moreover, if $F$ contains a compatible system of primitive $p$-power roots of unity, then there is a canonical isomorphism induced by \eqref{eq:cor:tate-cyclo-2},
	\begin{align}\label{eq:defn:int-tate-twist-3}
		\alpha_{F/\bb{Q}_p}:M\otimes_{\ca{O}_F} \ak{a}_{F/\bb{Q}_p}^{\otimes n}(n)\iso M\{n\}.
	\end{align}

In the rest of this section, we investigate the extents of arithmetic and geometric valuation fields. These results will not be used in the sequel of the article.

\begin{mylem}\label{lem:completion}
	Let $K$ be a Henselian valuation field of height $1$, $\widehat{K}$ its completion. Then, for any finite field extension $K'$ of $K$, the canonical morphism $\ca{O}_{K'}\otimes_{\ca{O}_K}\ca{O}_{\widehat{K}}\to \ca{O}_{\widehat{K'}}$ is an isomorphism.
\end{mylem}
\begin{proof}
	Recall that $\ca{O}_{K'}$ is an almost finitely generated torsion-free module over $\ca{O}_K$ (\cite[6.3.8]{gabber2003almost}). Thus, $\ca{O}_{K'}\otimes_{\ca{O}_K}\ca{O}_{\widehat{K}}$ is complete by \ref{cor:sep-comp}, i.e., $\ca{O}_{K'}\otimes_{\ca{O}_K}\ca{O}_{\widehat{K}}= \ca{O}_{\widehat{K'}}$.
\end{proof}

\begin{myprop}\label{prop:ari-geo-basic}
	Let $F\to F'$ be an extension of Henselian valuation fields of height $1$ extension of $\bb{Q}_p$.
	\begin{enumerate}
		\renewcommand{\labelenumi}{{\rm(\theenumi)}}
		\item If $F$ is geometric, then so is $F'$. If $F'$ is arithmetic, then so is $F$.\label{item:prop:ari-geo-basic-1}
		\item Assume that $F'=\widehat{F}$, or that $F'$ is finite over $F$, or that $\ca{O}_F\to\ca{O}_{F'}$ is ind-\'etale. Then, $F'$ is arithmetic (resp. geometric) if and only if $F$ is so.\label{item:prop:ari-geo-basic-2}
	\end{enumerate}
\end{myprop}
\begin{proof}
	(\ref{item:prop:ari-geo-basic-1}) Note that $\widetilde{\ak{m}}_F\scr{D}_{F_n/F}\subseteq \widetilde{\ak{m}}_F\scr{D}_{F_n/F_n\cap F'} \subseteq \scr{D}_{F'_n/F'}$ by \ref{prop:ramified-bc}. If $F$ is geometric, then $0\neq \widetilde{\ak{m}}_F\scr{D}_{F_\infty/F}\subseteq \scr{D}_{F'_\infty/F'}$ so that $F'$ is geometric. Note that $\Omega^1_{\ca{O}_F/\bb{Z}_p}[p^\infty]\subseteq \Omega^1_{\ca{O}_{F'}/\bb{Z}_p}[p^\infty]$ by \ref{thm:differential}.(\ref{item:thm:differential-1}). If $F'$ is arithmetic, then $\Omega^1_{\ca{O}_F/\bb{Z}_p}[p^\infty]$ is also bounded, i.e., $F$ is arithmetic.
	
	(\ref{item:prop:ari-geo-basic-2}) If $F'=\widehat{F}$ or $\ca{O}_F\to \ca{O}_{F'}$ is ind-\'etale, then $\ca{O}_{F'_n}=\ca{O}_{F_n}\otimes_{\ca{O}_F}\ca{O}_{F'}$ for any $n\in\bb{N}$ by \ref{lem:completion} and \cite[\href{https://stacks.math.columbia.edu/tag/03GE}{03GE}]{stacks-project}. In this case, if $F'$ is geometric, then so is $F$ as $ \widetilde{\ak{m}}_{F'}\scr{D}_{F'_n/F'}=\widetilde{\ak{m}}_F\scr{D}_{F_n/F}\ca{O}_{F_n'}$ for any $n\in\bb{N}$ (\cite[4.1.24]{gabber2003almost}). On the other hand, if $F$ is arithmetic, then $\df\log(\zeta_{p^n})\neq 0$ in $\Omega^1_{\ca{O}_{F_\infty}/\ca{O}_F}$ for $n$ large enough and thus it is nonzero in $\Omega^1_{\ca{O}_{F'_\infty}/\ca{O}_{F'}}=\ca{O}_{F_\infty'}\otimes_{\ca{O}_{F_\infty}}\Omega^1_{\ca{O}_{F_\infty}/\ca{O}_F}$, which implies that $F'$ is also arithmetic by \ref{lem:ari-disc}.
	
	If $F'$ is finite over $F$, then we have $\widetilde{\ak{m}}_F\scr{D}_{F'_n/F'}\scr{D}_{F'/F}=\widetilde{\ak{m}}_F\scr{D}_{F'_n/F}\subseteq \widetilde{\ak{m}}_F\scr{D}_{F_n/F}\ca{O}_{F_n'}$ for any $n\in\bb{N}$ (\cite[4.1.25]{gabber2003almost}). This shows that $F$ is geometric if $F'$ is so. Consider the canonical exact sequence \eqref{eq:lem:rhom-1} $0\to \ca{O}_{F'}\otimes_{\ca{O}_F}\Omega^1_{\ca{O}_F/\bb{Z}_p}[p^\infty]\to \Omega^1_{\ca{O}_{F'}/\bb{Z}_p}[p^\infty]\to \Omega^1_{\ca{O}_{F'}/\ca{O}_F}$. As the torsion module $\Omega^1_{\ca{O}_{F'}/\ca{O}_F}$ is almost finitely generated (\ref{thm:differential}.(\ref{item:thm:differential-3})), we see that $\Omega^1_{\ca{O}_{F'}/\bb{Z}_p}[p^\infty]$ is bounded if $\Omega^1_{\ca{O}_F/\bb{Z}_p}[p^\infty]$ is so, i.e., $F'$ is arithmetic if $F$ is so.
	
	Combining with (\ref{item:prop:ari-geo-basic-1}), we obtain the conclusion.
\end{proof}

\begin{mylem}\label{lem:geo-uniform}
	Let $F$ be a Henselian valuation field of height $1$ extension of $\bb{Q}_p$. If $\Omega^1_{\ca{O}_{F_\infty}/\ca{O}_F}$ is uniformly almost finitely generated over $\ca{O}_{F_\infty}$, then $F$ is geometric.
\end{mylem}
\begin{proof}
	For any $n\in\bb{N}$, consider the canonical exact sequence of uniformly almost finitely generated $\ca{O}_{F_\infty}$-modules (\ref{thm:differential}),
	\begin{align}
		0\longrightarrow \ca{O}_{F_\infty}\otimes_{\ca{O}_{F_n}}\Omega^1_{\ca{O}_{F_n}/\ca{O}_F}\longrightarrow \Omega^1_{\ca{O}_{F_\infty}/\ca{O}_F}\longrightarrow \Omega^1_{\ca{O}_{F_\infty}/\ca{O}_{F_n}}\longrightarrow 0.
	\end{align}
	Thus, we have $F_0(\Omega^{1,\al}_{\ca{O}_{F_\infty}/\ca{O}_F})=F_0(\Omega^{1,\al}_{\ca{O}_{F_n}/\ca{O}_F})\cdot F_0(\Omega^{1,\al}_{\ca{O}_{F_\infty}/\ca{O}_{F_n}})$ by \eqref{eq:para:fitting-ideal-3}. Since $F_0(\Omega^{1,\al}_{\ca{O}_{F_n}/\ca{O}_F})=\scr{D}_{F_n/F}^{\al}$ by \ref{thm:differential}.(\ref{item:thm:differential-3}), we obtain that $\ak{m}_{F_\infty}\cdot F_0(\Omega^1_{\ca{O}_{F_\infty}/\ca{O}_F})\subseteq \ak{m}_{F_\infty}\cdot\scr{D}_{F_n/F}$. As $\Omega^1_{\ca{O}_{F_\infty}/\ca{O}_F}$ is torsion and uniformly almost finitely generated, we see that $F_0(\Omega^1_{\ca{O}_{F_\infty}/\ca{O}_F})\neq 0$ \eqref{eq:para:fitting-ideal-2} and thus $0\neq \ak{m}_{F_\infty}\cdot F_0(\Omega^1_{\ca{O}_{F_\infty}/\ca{O}_F})\subseteq \scr{D}_{F_\infty/F}$, which shows that $F$ is geometric.
\end{proof}

\begin{mylem}\label{lem:ari-geo-finite}
	Let $K$ be a valuation field of height $1$, $M$ a torsion-free $\ca{O}_K$-module of finite rank, $Q$ a quotient of $M$. If $Q$ is a bounded torsion module, then $Q$ is uniformly almost finitely generated.
\end{mylem}
\begin{proof}
	Let $L$ be a maximal valuation field of height $1$ extension of $K$ with value group isomorphic to $\bb{R}$ (which exists by {\rm\ref{lem:max-comp-exist}}). As $M\otimes_{\ca{O}_K}\ca{O}_L\cong L^{\oplus n_1}\oplus \ca{O}_L^{\oplus n_2}\oplus \ak{m}_L^{\oplus n_3}$ by \ref{thm:kaplansky}, the bounded torsion quotient $Q\otimes_{\ca{O}_K}\ca{O}_L$ is a quotient of $(\ca{O}_L/\pi \ca{O}_L)^{\oplus n_2}\oplus (\ak{m}_L/\pi\ak{m}_L)^{\oplus n_3}$ for some $\pi\in \ca{O}_L$ and thus uniformly almost finitely generated with uniform bound $c=n_2+n_3$. In particular, $\wedge^{c+1}_{\ca{O}_L}(Q\otimes_{\ca{O}_K}\ca{O}_L)$ is almost zero. By faithfully flat descent (\cite[3.2.26.(\luoma{2})]{gabber2003almost}), we see that $Q$ is almost finitely generated and that $\wedge^{c+1}_{\ca{O}_K}Q$ is almost zero. Therefore, $Q$ is uniformly almost finitely generated with uniform bound $c$ by \cite[6.3.6.(\luoma{2})]{gabber2003almost}.
\end{proof}

\begin{myprop}\label{prop:ari-geo-finite}
	Let $K$ be a Henselian discrete valuation field extension of $\bb{Q}_p$ with perfect residue field, $F$ a Henselian valuation field of height $1$ extension of $K$ of finite transcendental degree. Assume that $F$ is not arithmetic and that the torsion module $\Omega^1_{\ca{O}_{F_\infty}/\ca{O}_F}$ is bounded. Then, $F$ is geometric.
\end{myprop}
\begin{proof}
	As $F$ is not arithmetic and of finite transcendental degree over $K$, the torsion module $\Omega^1_{\ca{O}_{F_\infty}/\ca{O}_F}$ is a quotient of the torsion-free $\ca{O}_{F_\infty}$-module $\Omega^1_{\ca{O}_{F_\infty}/\ca{O}_{K_\infty}}$ of finite rank by \ref{lem:ari-disc}. Therefore, $\Omega^1_{\ca{O}_{F_\infty}/\ca{O}_F}$ is uniformly almost finitely generated by \ref{lem:ari-geo-finite} so that $F$ is geometric by \ref{lem:geo-uniform}.
\end{proof}

\begin{myrem}\label{rem:ari-geo-finite}
	The proof of \ref{prop:ari-geo-finite} also shows that if the torsion-free module $\Omega^1_{\ca{O}_{F_\infty}/\ca{O}_{K_\infty}}$ is bounded, then $F$ is either arithmetic or geometric.
\end{myrem}
	
\begin{mycor}\label{cor:sen-perfd}
	Let $K$ be a Henselian discrete valuation field extension of $\bb{Q}_p$ with perfect residue field, $F$ an algebraic extension of $K$. Then, the following three properties on $F$ are equivalent to each other: non-arithmetic, geometric, and pre-perfectoid. 
\end{mycor}
\begin{proof}
	The implications ``pre-perfectoid" $\Rightarrow$ ``geometric" $\Rightarrow$ ``non-arithmetic" are proved in \ref{lem:ari-geo-basic}.(\ref{item:lem:ari-geo-basic-2}) and \ref{lem:ari-geo-ex} respectively. Suppose that $F$ is not arithmetic. Note that $\Omega^1_{\ca{O}_{F_\infty}/\ca{O}_{K_\infty}}=0$ as $K_\infty$ is pre-perfectoid (\ref{prop:cyclotomic-diff}). Thus, $\Omega^1_{\ca{O}_{F_\infty}/\ca{O}_F}$ also vanishes as a quotient of $\Omega^1_{\ca{O}_{F_\infty}/\ca{O}_{K_\infty}}$ by \ref{lem:ari-disc}. Therefore, $F$ is pre-perfectoid as $F_\infty$ is so (\cite[6.6.2]{gabber2003almost}).
\end{proof}

\begin{myrem}\label{rem:sen-perfd}
	In \ref{cor:sen-perfd}, if $F$ is finitely wildly ramified, i.e., the $p$-exponent of the ramification index of any finite subextension of $F/K$ has a uniform upper bound, then $F$ is arithmetic by \ref{cor:disc-diff-torsion}. Conversely, if we assume moreover that $F$ is a Galois extension of $K$ whose Galois group is a $p$-adic analytic group, then by a theorem of Sen on its ramification subgroups (\cite[\textsection4]{sen1972ramification}), if $F$ is infinitely (wildly) ramified, then $F$ is not arithmetic (and thus pre-perfectoid) (see also \cite[2.13]{coatesgreenberg1996kummer}).
\end{myrem}

\begin{mylem}[{\cite[6.1.9]{gabber2003almost}}]\label{lem:gauss}
	Let $K\to F$ be an extension of valuation fields, $\Gamma_F$ the opposite of the (totally ordered) value group $F^\times/\ca{O}_F^\times$ of $F$, $|\cdot|:F\to \Gamma_F\cup\{0\}$ the canonical absolute value map, $x\in F$, $a,b_1,\dots,b_n\in K$ finitely many elements. Assume that $|x-a|\leq |x-b_i|$ for any $1\leq i\leq n$. Then, for $f(X)=(X-b_1)\cdots(X-b_n)=\sum_{k=0}^{n}c_k(X-a)^k\in K[X]$, we have
	\begin{align}\label{eq:lem:gauss}
		|f(x)|=\sup_{0\leq k\leq n}|c_k(x-a)^k|.
	\end{align}
\end{mylem}
\begin{proof}
	Since both sides of \eqref{eq:lem:gauss} is multiplicative with respect to $f\in K[X]$ by \cite[\Luoma{6}.\textsection10.1, Lemme 1]{bourbaki2006commalg5-7}, it suffices to prove the case $f(X)=X-b_i$, i.e., $|x-b_i|=\max(|x-a|,|a-b_i|)$, which follows directly from $|x-a|\leq |x-b_i|$.
\end{proof}

\begin{myprop}[{\cite[6.5.9]{gabber2003almost}}]\label{prop:trans}
	Let $K$ be an algebraically closed valuation field, $F$ a valuation field extension of $K$ purely transcendental of degree $\mrm{trdeg}_K(F)=1$, $\Gamma_F$ the opposite of the (totally ordered) value group $F^\times/\ca{O}_F^\times$ of $F$, $|\cdot|:F\to \Gamma_F\cup\{0\}$ the canonical absolute value map, $\widehat{\Gamma}_F$ the completion of the totally ordered group $\Gamma_F\cup\{0\}$, $X\in F\setminus K$ a generator of $F$, $s=\inf_{b\in K}|X-b|\in \widehat{\Gamma}_F$.
	\begin{enumerate}
		\renewcommand{\labelenumi}{{\rm(\theenumi)}}
		\item Assume that $|X-b|>s$ for any $b\in K$. Then, $\Gamma_F=\Gamma_K$. Moreover, for any collection $\{a_i\}_{i\in I}$ of elements of $K$ indexed by a directed set $I$ with $\lim_{i\in I}|X-a_i|=s$ and for any $b_i\in K$ with $|X-a_i|=|b_i|$, if we put $A_i=\ca{O}_K[\frac{X-a_i}{b_i}]\subseteq \ca{O}_F$ and $\ak{p}_i=A_i\cap \ak{m}_F$, then there is a canonical filtered union of $\ca{O}_K$-subalgebras $\ca{O}_F=\colim_{i\in I}A_{i,\ak{p}_i}$.\label{item:prop:trans-1}
		\item Assume that there exists $a\in K$ with $|X-a|=s\in \Gamma_K\subseteq \Gamma_F$. Then, $\Gamma_F=\Gamma_K$. Moreover, if we take $b\in K^\times$ with $|b|=s$, $A=\ca{O}_K[\frac{X-a}{b}]\subseteq \ca{O}_F$ and $\ak{p}=\ak{m}_K A$, then $\ca{O}_F=A_{\ak{p}}$.\label{item:prop:trans-2}
		\item Assume that there exists $a\in K$ with $|X-a|=s$ and that $s<\Gamma_K$ (resp. $s>\Gamma_K$). If we put $A_b=\ca{O}_K[\frac{X-a}{b}]\subseteq \ca{O}_F$ (resp. $\ca{O}_K[\frac{b}{X-a}]$) for any $b\in K^\times$ and $\ak{p}_b\subseteq A_b$ the prime ideal generated by $\ak{m}_K$ and $\frac{X-a}{b}$ (resp. $\frac{b}{X-a}$), then there is a canonical filtered union of $\ca{O}_K$-subalgebras $\ca{O}_F=\colim_{b\in K^\times}A_{b,\ak{p}_b}$.\label{item:prop:trans-3}
		\item Assume that there exists $a\in K$ with $|X-a|=s\notin \Gamma_K$ and that $|b_0'|<s<|b_0|$ for some $b_0,b_0'\in K$. Then, for any collections $\{b_i\}_{i\in I}$ and $\{b_i'\}_{i\in I}$ of elements of $K$ indexed by a directed set $I$ such that $|b_i'|<s<|b_i|$ and $\lim_{i\in I}|b_i'|=s=\lim_{i\in I}|b_i|$, if we put $A_i=\ca{O}_K[\frac{b_i'}{X-a},\frac{X-a}{b_i}]\subseteq \ca{O}_F$ and $\ak{p}_i\subseteq A_i$ the prime ideal generated by $\ak{m}_K,\frac{b_i'}{X-a}$ and $\frac{X-a}{b_i}$, then there is a canonical filtered union of $\ca{O}_K$-algebras $\ca{O}_F=\colim_{i\in I}A_{i,\ak{p}_i}$.\label{item:prop:trans-4}
	\end{enumerate}
\end{myprop}
\begin{proof}
	This is proved in the proof of \cite[6.5.9]{gabber2003almost} which we sketch in the following. 
	
	(\ref{item:prop:trans-1}) For any collection $\{a_i\}_{i\in I}$ of elements of $K$ with $I$ directed and $\lim_{i\in I}|X-a_i|=s$ and for any $b\in K$, we have $|X-b|=\max(|X-a_i|,|a_i-b|)=|a_i-b|\in \Gamma_K$ for $i$ large enough. As $K$ is algebraically closed, any polynomial over $K$ splits. In particular, this verifies $\Gamma_F=\Gamma_K$. We take $b_i\in K$ with $|X-a_i|=|b_i|$ and for any element $f(X)/g(X)\in \ca{O}_F$, we write $f(X)=\sum_{k=0}^{n}c_k(X-a_i)^k$ and $g(X)=\sum_{k=0}^{n}c_k'(X-a_i)^k$ for some $c_k,c_k'\in K$ and $n\in\bb{N}$. After dividing $f$ and $g$ by an element $c\in K$ with $|c|=|g(X)|$, we may assume that $|g(X)|=1$. Taking $i$ large enough such that $|X-a_i|\leq |X-b|$ for any root $b$ of $fg$, we see that $|f(X)|=\sup_{0\leq k\leq n}|c_k(X-a_i)^k|\leq |g(X)|=\sup_{0\leq k\leq n}|c_k'(X-a_i)^k|=1$ by \ref{lem:gauss}. Therefore, we have $\sup_{0\leq k\leq n}|c_kb_i^k|\leq \sup_{0\leq k\leq n}|c_k'b_i^k|= 1$. This shows that $f(X)\in A_i$ and $g(X)\in A_i\setminus \ak{p}_i$ so that $f(X)/g(X)\in A_{i,\ak{p}_i}$. This also shows that $(A_{i,\ak{p}_i})_{i\in I}$ forms a directed system of $\ca{O}_K$-subalgebras of $\ca{O}_F$ with colimit $\ca{O}_F$.
	
	(\ref{item:prop:trans-2}) Note that for any  $b\in K$, we have $|X-b|=\max(|X-a|,|a-b|)\in \Gamma_K$. This shows that $\Gamma_F=\Gamma_K$ as $K$ is algebraically closed. We take $b\in K$ with $|X-a|=|b|$ and for any element $f(X)/g(X)\in \ca{O}_F$, we write $f(X)=\sum_{k=0}^{n}c_k(X-a)^k$ and $g(X)=\sum_{k=0}^{n}c_k'(X-a)^k$ for some $c_k,c_k'\in K$ and $n\in\bb{N}$. After dividing $f$ and $g$ by an element $c\in K$ with $|c|=|g(X)|$, we may assume that $|g(X)|=1$. Therefore, we see that $|f(X)|=\sup_{0\leq k\leq n}|c_k(X-a)^k|\leq |g(X)|=\sup_{0\leq k\leq n}|c_k'(X-a)^k|=1$ by \ref{lem:gauss}, i.e., $\sup_{0\leq k\leq n}|c_kb^k|\leq \sup_{0\leq k\leq n}|c_k'b^k|= 1$. This shows that $f(X)\in A$ and $g(X)\in A\setminus \ak{p}$ so that $f(X)/g(X)\in A_{\ak{p}}$ (and thus $\ca{O}_F=A_{\ak{p}}$). 
	
	(\ref{item:prop:trans-3}) Note that for any  $b\in K$, we have $|X-b|=\max(|X-a|,|a-b|)$. For any element $f(X)/g(X)\in \ca{O}_F$, we write $f(X)=\sum_{k=-m}^{n}c_k(X-a)^k$ and $g(X)=\sum_{k=-m}^{n}c_k'(X-a)^k$ for some $c_k,c_k'\in K$ and $m,n\in\bb{N}$. We have $|f(X)|=\sup_{-m\leq k\leq n}|c_k(X-a)^k|\leq |g(X)|=\sup_{-m\leq k\leq n}|c_k'(X-a)^k|\neq 0$ by \ref{lem:gauss}. As $|X-a|<\Gamma_K$ (resp. $|X-a|>\Gamma_K$), we have $|g(X)|=|c_{k_0}'(X-a)^{k_0}|$ where $k_0$ is the smallest (resp. largest) index such that $c_{k_0}'\neq 0$. After dividing $f$ and $g$ by $c_{k_0}'(X-a)^{k_0}$, we may assume that $m=0$ (resp. $n=0$) and $|g(X)|=|c_0'|=1$. Taking $b\in K^\times$ with $|b|$ sufficiently small (resp. large) such that $|c_kb^k|\leq 1$ and $|c_k'b^k|\leq 1$ for any $k$, we see that $f(X)\in A_b$ and $g(X)\in A_b\setminus \ak{p}_b$. This also shows that $(A_{b,\ak{p}_b})_{b\in K^\times}$ forms a directed system of $\ca{O}_K$-subalgebras of $\ca{O}_F$ with colimit $\ca{O}_F$.
	
	(\ref{item:prop:trans-4}) Note that for any  $b\in K$, we have $|X-b|=\max(|X-a|,|a-b|)$. For any element $f(X)/g(X)\in \ca{O}_F$, we write $f(X)=\sum_{k=-m}^{n}c_k(X-a)^k$ and $g(X)=\sum_{k=-m}^{n}c_k'(X-a)^k$ for some $c_k,c_k'\in K$ and $m,n\in\bb{N}$. We have $|f(X)|=\sup_{-m\leq k\leq n}|c_k(X-a)^k|\leq |g(X)|=\sup_{-m\leq k\leq n}|c_k'(X-a)^k|\neq 0$ by \ref{lem:gauss}. We take $k_0$ such that $|g(X)|=|c_{k_0}'(X-a)^{k_0}|$. After dividing $f$ and $g$ by $c_{k_0}'(X-a)^{k_0}$, we may assume that $|g(X)|=|c_0'|=1$. Taking $i$ large enough such that $|c_kb_i^k|\leq 1$ and $|c_k'b_i^k|\leq 1$ for any $k>0$ and $|c_kb_i'^k|\leq 1$ and $|c_k'b_i'^k|\leq 1$ for any $k<0$, we see that $f(X)\in A_i$ and $g(X)\in A_i\setminus \ak{p}_i$. This also shows that $(A_{i,\ak{p}_i})_{i\in I}$ forms a directed system of $\ca{O}_K$-subalgebras of $\ca{O}_F$ with colimit $\ca{O}_F$.
\end{proof}

\begin{mycor}\label{cor:trans}
	With the notation in {\rm\ref{prop:trans}}, assume that $F$ is of height $1$ (so that $\widehat{\Gamma}_F$ identifies with a subgroup of $\bb{R}_{\geq 0}$). Consider the nonzero $\ca{O}_F$-submodules $\ak{s}=\{b\in F\ |\ s\cdot|b|< 1\}$ and $\overline{\ak{s}}=\{b\in F\ |\ s\cdot|b|\leq 1\}$. 
	\begin{enumerate}
		\renewcommand{\labelenumi}{{\rm(\theenumi)}}
		\item Assume that $|X-b|>s$ for any $b\in K$. Then, there is an isomorphism of $\ca{O}_F$-modules $\ak{s}\iso \Omega^1_{\ca{O}_F/\ca{O}_K}$, which sends $1\in \ak{s}[1/p]=F$ to $\df X\in \Omega^1_{\ca{O}_F/\ca{O}_K}[1/p]=\Omega^1_{F/K}$ after inverting $p$.\label{item:cor:trans-1}
		\item Assume that there exists $a\in K$ with $|X-a|=s$. Then, there is an isomorphism of $\ca{O}_F$-modules $\overline{\ak{s}}\iso \Omega^1_{\ca{O}_F/\ca{O}_K}$, which sends $1\in \overline{\ak{s}}[1/p]=F$ to $\df X\in \Omega^1_{\ca{O}_F/\ca{O}_K}[1/p]=\Omega^1_{F/K}$ after inverting $p$.\label{item:cor:trans-2}
	\end{enumerate}
\end{mycor}
\begin{proof}
	Firstly, notice that $\Omega^1_{\ca{O}_F/\ca{O}_K}$ is torsion-free as $K$ is algebraically closed (\ref{thm:differential}.(\ref{item:thm:differential-2})). Thus, it identifies with an $\ca{O}_F$-submodule of $\Omega^1_{F/K}=F\cdot \df X$. We consider each case of \ref{prop:trans} separately.
	
	In the case of \ref{prop:trans}.(\ref{item:prop:trans-1}), $\Omega^1_{\ca{O}_F/\ca{O}_K}$ is generated by $\df \frac{X-a_i}{b_i}=b_i^{-1}\df X$, and $\ak{s}$ is also generated by $b_i^{-1}$ as $\lim_{i\in I}|b_i|=s<|b_i|$.
	
	In the case of \ref{prop:trans}.(\ref{item:prop:trans-2}), $\Omega^1_{\ca{O}_F/\ca{O}_K}$ is generated by $\df \frac{X-a}{b}=b^{-1}\df X$, and $\overline{\ak{s}}$ is also generated by $b^{-1}$ as $s=|b|$.
	
	The case of \ref{prop:trans}.(\ref{item:prop:trans-3}) does not hold here because $F$ is of height $1$.
	
	In the case of \ref{prop:trans}.(\ref{item:prop:trans-4}), $\Omega^1_{\ca{O}_F/\ca{O}_K}$ is generated by $\df \frac{b_i'}{X-a}=\frac{-b_i'}{(X-a)^2}\df X$ and $\df \frac{X-a}{b_i}=b_i^{-1}\df X$, and $\overline{\ak{s}}=\ak{s}$ is also generated by $\frac{-b_i'}{(X-a)^2}$ and $b_i^{-1}$ as $|b_i'|<s<|b_i|$ and $\lim_{i\in I}|b_i'|=s=\lim_{i\in I}|b_i|\notin \Gamma_K$.
\end{proof}

\begin{mylem}\label{lem:trans}
	With the notation in {\rm\ref{prop:trans}}, assume that $F$ is of height $1$ (so that $\widehat{\Gamma}_F$ identifies with a subgroup of $\bb{R}_{\geq 0}$). Then, the following conditions are equivalent:
	\begin{enumerate}
		\renewcommand{\labelenumi}{{\rm(\theenumi)}}
		\item The torsion-free $\ca{O}_F$-module $\Omega^1_{\ca{O}_F/\ca{O}_K}$ is not bounded {\rm(\ref{defn:sep-bound}.(\ref{item:defn:sep-bound-1}))}. \label{item:lem:trans-1}
		\item $s=\inf_{b\in K}|X-b|=0$.\label{item:lem:trans-2}
		\item $F$ is a valuation field subextension of the completion $K\to \widehat{K}$.\label{item:lem:trans-3}
	\end{enumerate}
\end{mylem}
\begin{proof}
	Firstly, note that (\ref{item:lem:trans-1}) and (\ref{item:lem:trans-2}) are equivalent by \ref{cor:trans}, and that (\ref{item:lem:trans-3}) implies (\ref{item:lem:trans-2}) directly. Then, we suppose that (\ref{item:lem:trans-2}) holds. In this case, $|X-b|>s=0$ for any $b\in K$ as $X\notin K$. In particular, there is a collection $\{a_i\}_{i\in I}$ of elements of $K$ with $I$ directed and $\lim_{i\in I}|X-a_i|=0$. Thus, $\{a_i\}_{i\in I}$ is a Cauchy net on $K$ with respect to the valuation topology and let $a\in\widehat{K}$ be its limit. Notice that $a\notin K$ (otherwise $|X-a|=0$ for $a\in K$ which is a contradiction). As $K$ is algebraically closed, $a$ is transcendental over $K$. Hence, the $K$-algebra homomorphism $K[X]\to \widehat{K}$ sending $X$ to $a$ is injective and thus induces an injection of fraction fields $\iota:F=K(X)\to \widehat{K}$. To prove (\ref{item:lem:trans-3}), it suffices to check that $\ca{O}_{\widehat{K}}$ and $\ca{O}_F$ induce the same valuation on $K[X]$. Since any $K$-polynomial splits, it suffices to consider the valuation of $X-b$ where $b\in K$. In $F$, we have $|X-b|_F=|a_i-b|_K$ for $i$ large enough as $\lim_{i\in I}|X-a_i|_F=0$; and in $\widehat{K}$, we also have $|\iota(X-b)|_{\widehat{K}}=|a-b|_{\widehat{K}}=|a_i-b|_K$ for $i$ large enough as $\lim_{i\in I}|a-a_i|_{\widehat{K}}=0$, which completes the proof.
\end{proof}

\begin{mylem}\label{lem:trans-completion}
	Let $K$ be a Henselian valuation field of height $1$ with algebraic closure $\overline{K}$, $F$ a Henselian valuation field subextension of $\widehat{\overline{K}}/K$, $L=\overline{K}\cap F$. Then, $F\subseteq \widehat{L}$.
\end{mylem}
\begin{proof}
	Let $M$ be the composite of $\overline{K}$ and $F$ in $\widehat{\overline{K}}$. Note that $\widehat{\overline{K}}=\widehat{M}$. By Galois theory, there is a canonical isomorphism of Galois groups $\Sigma=\gal(M/F)\iso \gal(\overline{K}/L)$ given by restriction. Recall that their actions on $M$ and $\overline{K}$ extends continuously and uniquely to $\widehat{M}=\widehat{\overline{K}}$.
	\begin{align}
		\xymatrix{
			\widehat{M}&M\ar[l]&F\ar[l]\\
			\widehat{\overline{K}}\ar@{=}[u]&\overline{K}\ar[l]\ar[u]&L\ar[l]\ar[u]
		}
	\end{align}
	Therefore, $F\subseteq \widehat{M}^\Sigma=(\widehat{\overline{K}})^\Sigma=\widehat{L} $ by Ax-Sen-Tate's theorem \cite[page 417]{ax1970ax}.
\end{proof}

\begin{myprop}\label{prop:trans-ari-geo}
	Let $K$ be a Henselian discrete valuation field extension of $\bb{Q}_p$ with perfect residue field, $L$ an algebraic extension of $K$, $F$ a valuation field of height $1$ finitely generated extension of $L$ of transcendental degree $\mrm{trdeg}_L(F)\leq 1$. Then, the Henselization $F^{\mrm{h}}$ of $F$ is either arithmetic or geometric.
\end{myprop}
\begin{proof}
	We may assume that $F$ is purely transcendental over $L$ by \ref{prop:ari-geo-basic}.(\ref{item:prop:ari-geo-basic-2}). If $\mrm{trdeg}_L(F)=0$ (i.e., $F=L$), then the conclusion follows from \ref{cor:sen-perfd}. It remains to consider the case $\mrm{trdeg}_L(F)=1$. Let $F^{\mrm{h}}_{\overline{K}}$ be the composite of $\overline{K}$ and $F^{\mrm{h}}$ endowed with the unique valuation extending that of $F^{\mrm{h}}$.
	
	If the torsion-free module $\Omega^1_{\ca{O}_{F_\infty}/\ca{O}_{K_\infty}}$ is bounded (where $\ca{O}_{F_\infty}=F_\infty\cap \ca{O}_{F^{\mrm{h}}_{\overline{K}}}$), then the conclusion follows from \ref{prop:ari-geo-finite} and \ref{rem:ari-geo-finite} as $\mrm{trdeg}_K(F)=1$.
	
	Then, we assume that $\Omega^1_{\ca{O}_{F_\infty}/\ca{O}_{K_\infty}}$ is unbounded. Let $X$ be a generator of $F$ over $L$. As $\Omega^1_{F_\infty/K_\infty}= F_\infty\cdot\df X$, the inclusion $\Omega^1_{\ca{O}_{F_\infty}/\ca{O}_{K_\infty}}\subseteq \Omega^1_{F_\infty/K_\infty}$ must be an equality.  Let $M=\overline{K}(X)$ be the composite of $\overline{K}$ and $F$ in $F^{\mrm{h}}_{\overline{K}}$. Then, we see that the inclusion $\Omega^1_{\ca{O}_M/\ca{O}_{\overline{K}}}\subseteq \Omega^1_{M/\overline{K}}=M\cdot \df X$ is also an equality (where $\ca{O}_M=M\cap \ca{O}_{F^{\mrm{h}}_{\overline{K}}}$). Therefore, $M$ is a valuation field subextension of $\widehat{\overline{K}}$ by \ref{lem:trans} and thus so are $F$ and $F^{\mrm{h}}$. Then, we have $\overline{K}\cap F^{\mrm{h}}\subseteq F^{\mrm{h}}\subseteq  (\overline{K}\cap F^{\mrm{h}})^\wedge$ by \ref{lem:trans-completion}. As $\overline{K}\cap F^{\mrm{h}}$ is either arithmetic or geometric by \ref{cor:sen-perfd}, we see that $F^{\mrm{h}}$ is either arithmetic or geometric by \ref{prop:ari-geo-basic}.
\end{proof}

\begin{myrem}\label{rem:trans-ari-geo}
	If all the Henselian valuation fields of height $1$ extension of $K$ of transcendental degree $\leq 1$ (i.e., algebraic extensions of such $F$ in \ref{prop:trans-ari-geo}) were either arithmetic or geometric, then we could extend \ref{prop:trans-ari-geo} to the case where $\mrm{trdge}_L(F)\leq 2$ by the same arguments.
\end{myrem}

\section{Ramification of Kummer Extensions over Cyclotomic Field}\label{sec:kummer}
Following \cite[\textsection4]{he2024perfd}, we compute the Galois cohomology of a Kummer extension defined by adjoining $p$-power roots of a single element in a valuation ring (see \ref{prop:tate-sen} and \ref{cor:galois-coh}).

\begin{mypara}\label{para:tower}
	In this section, we fix a pre-perfectoid field $K$ extension of $\bb{Q}_p$ containing a compatible system of primitive $p$-power roots of unity $(\zeta_{p^n})_{n\in\bb{N}}$, a Henselian valuation field $F$ of height $1$ extension of $K$, and an algebraic closure $\overline{F}$ of $F$. Note that the integral closure $\ca{O}_{\overline{F}}$ of $\ca{O}_F$ in $\overline{F}$ is still a Henselian valuation ring of height $1$ (see \cite[\Luoma{6}.\textsection8.6, Proposition 6]{bourbaki2006commalg5-7} and \cite[\href{https://stacks.math.columbia.edu/tag/04GH}{04GH}]{stacks-project}). Let $v_p:\overline{F}\to \bb{R}\cup\{\infty\}$ be a valuation map with $v_p(p)=1$ and $v_p(0)=\infty$, and let
	\begin{align}\label{eq:para:tower-1}
		|\cdot|=p^{-v_p(\cdot)}:\overline{F}\longrightarrow \bb{R}_{\geq 0}
	\end{align}
	be the associated ultrametric absolute value. We put $G_F=\gal(\overline{F}/F)$ the absolute Galois group of $F$ which naturally acts on $\ca{O}_{\overline{F}}$.
	
	We fix an element $\pi\in \ak{m}_K\setminus p\ca{O}_K$, an element $t\in\ca{O}_F$ such that 
	\begin{align}\label{eq:para:tower-t}
		\df t \notin \pi\ak{m}_K\cdot \Omega^1_{\ca{O}_F/\ca{O}_K},
	\end{align}
	and a compatible system of $p$-power roots $(t_{p^n})_{n\in \bb{N}}$ of $t$ contained in $\overline{F}$ (where $t_1=t$) . For any $n\in\bb{N}$, we put $F_n=F(t_{p^n})$ the finite extension of $F$ generated by $t_{p^n}$, which is a finite Galois extension of $F$ independent of the choice of the $p^n$-th root $t_{p^n}$ of $t$. We put $F_\infty=\bigcup_{n\in\bb{N}}F_n$. Then, there is a continuous group homomorphism
	\begin{align}\label{eq:para:tower-2}
		\xi_t:G_F\longrightarrow \bb{Z}_p
	\end{align}
	characterized by $\tau(t_{p^n})=\zeta_{p^n}^{\xi_t(\tau)}t_{p^n}$ for any $\tau\in G_F$ and $n\in\bb{N}$. Since $t_p\notin F$ (as $\df t \notin p\cdot \Omega^1_{\ca{O}_F/\ca{O}_K}$), $\xi_t$ induces an isomorphism for any $n\in\bb{N}$ (see \cite[4.8]{he2024perfd})
	\begin{align}\label{eq:para:tower-3}
		\gal(F_\infty/F_n)\iso p^n\bb{Z}_p.
	\end{align}
\end{mypara}

\begin{mylem}[{\cite[4.10]{he2024perfd}}]\label{lem:ann-dt}
	For any $n\in\bb{N}$, the annihilator of the element $\df t_{p^n}$ of the $\ca{O}_{F_n}$-module $\Omega^1_{\ca{O}_{F_n}/\ca{O}_F}$ satisfies the following relations
	\begin{align}
		p^nt_{p^n}^{p^n-1}\ca{O}_{F_n}\subseteq \mrm{Ann}_{\ca{O}_{F_n}}(\df t_{p^n}) \subseteq \pi^{-1}p^nt_{p^n}^{p^n-1}\ca{O}_{F_n}.
	\end{align}
\end{mylem}
\begin{proof}
	We follow the proof of \cite[4.10]{he2024perfd}. We may assume that $n>0$. Consider the canonical exact sequence (\ref{thm:differential}.(\ref{item:thm:differential-1}))
	\begin{align}
		0\longrightarrow \ca{O}_{F_n}\otimes_{\ca{O}_F}\Omega^1_{\ca{O}_F/\ca{O}_K}\stackrel{\alpha_n}{\longrightarrow} \Omega^1_{\ca{O}_{F_n}/\ca{O}_K}\stackrel{\beta_n}{\longrightarrow} \Omega^1_{\ca{O}_{F_n}/\ca{O}_F}\longrightarrow 0.
	\end{align}
	Firstly, in $\Omega^1_{\ca{O}_{F_n}/\ca{O}_K}$, we have $p^nt_{p^n}^{p^n-1}\df t_{p^n}=\df t\in \mrm{Im}(\alpha_n)$. Thus, $p^nt_{p^n}^{p^n-1}\ca{O}_{F_n}\subseteq \mrm{Ann}_{\ca{O}_{F_n}}(\df t_{p^n})$ for $\df t_{p^n}\in \Omega^1_{\ca{O}_{F_n}/\ca{O}_F}$.
	
	Suppose that $\mrm{Ann}_{\ca{O}_{F_n}}(\df t_{p^n})$ is not contained in $\pi^{-1}p^nt_{p^n}^{p^n-1}\ca{O}_{F_n}$. Then, there exists $\epsilon\in\ak{m}_K$ and $\omega\in \ca{O}_{F_n}\otimes_{\ca{O}_F}\Omega^1_{\ca{O}_F/\ca{O}_K}$ such that $\alpha_n(\omega)=(\pi\epsilon)^{-1}p^nt_{p^n}^{p^n-1}\df t_{p^n}$. Thus, $\alpha_n(\pi\epsilon \omega)=\alpha_n(\df t)$. The injectivity of $\alpha_n$ implies that $\df t=\pi\epsilon\omega$ in $\ca{O}_{F_n}\otimes_{\ca{O}_F}\Omega^1_{\ca{O}_F/\ca{O}_K}$, i.e., $\df t$ is zero in $\ca{O}_{F_n}/\pi\epsilon\ca{O}_{F_n}\otimes_{\ca{O}_F}\Omega^1_{\ca{O}_F/\ca{O}_K}$. Since $\Omega^1_{\ca{O}_F/\ca{O}_K}$ is a flat $\ca{O}_F$-module (\ref{cor:differential}), the canonical homomorphism $\Omega^1_{\ca{O}_F/\ca{O}_K}/\pi\epsilon\Omega^1_{\ca{O}_F/\ca{O}_K}\to \ca{O}_{F_n}/\pi\epsilon\ca{O}_{F_n}\otimes_{\ca{O}_F}\Omega^1_{\ca{O}_F/\ca{O}_K}$ is injective as $\ca{O}_F/\pi\epsilon\ca{O}_F\to \ca{O}_{F_n}/\pi\epsilon\ca{O}_{F_n}$ is so. Hence, $\df t$ is also zero in $\Omega^1_{\ca{O}_F/\ca{O}_K}/\pi\epsilon\Omega^1_{\ca{O}_F/\ca{O}_K}$, which contradicts with our assumption \eqref{eq:para:tower-t} on $t$. Therefore, $\mrm{Ann}_{\ca{O}_{F_n}}(\df t_{p^n}) \subseteq \pi^{-1}p^nt_{p^n}^{p^n-1}\ca{O}_{F_n}$ for $\df t_{p^n}\in \Omega^1_{\ca{O}_{F_n}/\ca{O}_F}$.
\end{proof}

\begin{myprop}[{\cite[4.11]{he2024perfd}}]\label{prop:different-range}
	For any $n\in\bb{N}$, we have
	\begin{align}
		p^n\ak{m}_{F_n}\subseteq \scr{D}_{F_n/F}\subseteq \pi^{-1}p^n\ca{O}_{F_n}.
	\end{align}
\end{myprop}
\begin{proof}
	Since $\scr{D}_{F_n/F}^{\al}=F_0(\Omega^{1,\al}_{\ca{O}_{F_n}/\ca{O}_F})=\prod_{i=1}^\infty (\mrm{Ann}_{\ca{O}_{F_n}}(\Omega^i_{\ca{O}_{F_n}/\ca{O}_F}))^{\al}$ (\ref{thm:differential}.(\ref{item:thm:differential-3}) and \eqref{eq:para:fitting-ideal-1}), we see that $\ak{m}_{F_n}\scr{D}_{F_n/F}$ annihilates $\Omega^1_{\ca{O}_{F_n}/\ca{O}_F}$. Hence, we have $\ak{m}_{F_n}\scr{D}_{F_n/F}\subseteq \pi^{-1}p^n\ca{O}_{F_n}$ by \ref{lem:ann-dt} and thus $\scr{D}_{F_n/F}\subseteq \pi^{-1}p^n\ca{O}_{F_n}$ by \eqref{eq:para:notation-norm-3}. On the other hand, we have $p^n\ak{m}_{F_n}=[F_n:F]\cdot \ak{m}_F\ak{m}_{F_n}^{-1}\subseteq \scr{D}_{F_n/F}$ by \eqref{eq:para:tower-3}, \eqref{eq:para:notation-almost-2} and \ref{cor:different-trace}. 
\end{proof}

\begin{mycor}\label{cor:different-dt}
	For any $n\in\bb{N}$, the quotient by the $\ca{O}_{F_n}$-submodule of $\Omega^1_{\ca{O}_{F_n}/\ca{O}_F}$ generated by $\df t_{p^n}$ is killed by $\pi\ak{m}_K$. Moreover, we have
	\begin{align}\label{eq:cor:different-dt-1}
		\df t_{p^n}\notin \pi\ak{m}_K\cdot \Omega^1_{\ca{O}_{F_n}/\ca{O}_K}.
	\end{align}
\end{mycor}
\begin{proof}
	Let $Q$ be the quotient of $\ca{O}_{F_n}\cdot\df t_{p^n}\subseteq \Omega^1_{\ca{O}_{F_n}/\ca{O}_F}$. Then, $0\to \ca{O}_{F_n}\cdot\df t_{p^n}\to \Omega^1_{\ca{O}_{F_n}/\ca{O}_F}\to Q\to 0$ is an exact sequence of uniformly almost finitely generated $\ca{O}_{F_n}$-modules (\ref{thm:differential}.(\ref{item:thm:differential-3})). In particular, we have $F_0(\Omega^{1,\al}_{\ca{O}_{F_n}/\ca{O}_F})=F_0(\ca{O}_{F_n}^{\al}\cdot\df t_{p^n})\cdot F_0(Q^{\al})$ by \eqref{eq:para:fitting-ideal-3}. Recall that $p^n\ca{O}_{F_n}^{\al}\subseteq \scr{D}_{F_n/F}^{\al}=F_0(\Omega^{1,\al}_{\ca{O}_{F_n}/\ca{O}_F})$ (\ref{prop:different-range} and \ref{thm:differential}.(\ref{item:thm:differential-3})) and $F_0(\ca{O}_{F_n}^{\al}\cdot\df t_{p^n})=(\mrm{Ann}_{\ca{O}_{F_n}}(\df t_{p^n}))^{\al}\subseteq \pi^{-1}p^n\ca{O}_{F_n}^{\al}$ (\eqref{eq:para:fitting-ideal-1} and \ref{lem:ann-dt}). Thus, we have $\pi\ca{O}_{F_n}^{\al}\subseteq F_0(Q^{\al})$, which implies that $\pi\ak{m}_K\cdot Q=0$ by \eqref{eq:para:fitting-ideal-1}.

	Moreover, suppose that $\df t_{p^n}=\pi\epsilon \omega$ for some $\epsilon\in\ak{m}_K$ and $\omega\in \Omega^1_{\ca{O}_{F_n}/\ca{O}_F}$. We write $\epsilon=\epsilon_1\cdot \epsilon_2$ for some $\epsilon_1,\epsilon_2\in\ak{m}_K$. Then, as $\pi\epsilon_1\cdot Q=0$, we have $\pi\epsilon\omega=a\epsilon_2\cdot \df t_{p^n}$ for some $a\in\ca{O}_{F_n}$. Thus, we have $\df t_{p^n}=a\epsilon_2\df t_{p^n}$ in $\Omega^1_{\ca{O}_{F_n}/\ca{O}_F}$, which implies that $\df t_{p^n}=0$ as $1-a\epsilon_2\in\ca{O}_{F_n}^\times$. This contradicts with \ref{lem:ann-dt}. Therefore, we have $\df t_{p^n}\notin \pi\ak{m}_K\cdot \Omega^1_{\ca{O}_{F_n}/\ca{O}_F}$ and thus $\df t_{p^n}\notin \pi\ak{m}_K\cdot \Omega^1_{\ca{O}_{F_n}/\ca{O}_K}$.
\end{proof}

\begin{myrem}\label{rem:different-dt}
	If we put $A_n=\ca{O}_F[T]/(T^{p^n}-t)$ regarded as an $\ca{O}_F$-subalgebra of $\ca{O}_{F_n}$ by identifying $T$ with $t_{p^n}\in\ca{O}_{F_n}$, then one can moreover check that $\pi\ak{m}_K\cdot H_0(\bb{L}_{\ca{O}_{F_n}/A_n})=0$, $\pi\cdot H_1(\bb{L}_{\ca{O}_{F_n}/A_n})=0$ and $H_i(\bb{L}_{\ca{O}_{F_n}/A_n})=0$ for any $i\in\bb{Z}\setminus\{0,1\}$.
\end{myrem}

\begin{mycor}[{\cite[4.12]{he2024perfd}}]\label{cor:trace-range}
	For any $n\in \bb{N}$ and any $x\in F_n$, we have
	\begin{align}
		|p^{-n}\mrm{Tr}_{F_n/F}(x)|\leq |\pi|^{-1}\cdot|x|.
	\end{align}
\end{mycor}
\begin{proof}
	It follows directly from \ref{cor:different-trace} and \ref{prop:different-range} (note that $|\ak{m}_{F_n}|=1$ for any $n\in\bb{N}$).
\end{proof}

\begin{mypara}\label{para:tate-trace}
	For any $n\in\bb{N}$, the normalized trace map $p^{-n}\mrm{Tr}_{F_n/F}:F_n\to F$ is a $G_F$-equivariant $F$-linear retraction of the inclusion $F\to F_n$. We obtain a system of maps $(p^{-n}\mrm{Tr}_{F_n/F}:F_n\to F)_{n\in\bb{N}}$ compatible with the inclusions $F_n\to F_{n+1}$. Taking filtered union, we obtain a $G_{F}$-equivariant $F$-linear retraction
	\begin{align}\label{eq:para:tate-trace-1}
		\ca{T}:F_\infty\longrightarrow F
	\end{align}
	of the inclusion $F\to F_\infty$, called \emph{Tate's normalized trace map} of the extension $F_\infty/F$. In particular, there is a canonical decomposition of $F$-modules for any $n\in\bb{N}\cup\{\infty\}$,
	\begin{align}\label{eq:para:tate-trace-2}
		F_n=F\oplus \ke(\ca{T}|_{F_n}),
	\end{align}
	where $\ca{T}|_{F_n}$ denotes the restriction of $\ca{T}$ on $F_n$. Note that $\{t_{p^n}^i\}_{0< i< p^n}$ forms an $F$-basis of $\ke(\ca{T}|_{F_n})$.
\end{mypara}

\begin{mylem}\label{lem:trace-tau}
	Let $\tau\in \gal(F_\infty/F)$ be a topological generator \eqref{eq:para:tower-3}. Then, for any $x\in F_\infty$, we have
	\begin{align}\label{eq:cor:trace-tau-1}
		|\ca{T}(x)-x|\leq |p\pi|^{-1}\cdot |\tau(x)-x|,
	\end{align}
	where $\ca{T}$ is Tate's normalized trace map \eqref{eq:para:tate-trace-1}.
\end{mylem}
\begin{proof}
	It follows directly from \ref{prop:trace-tau} and \ref{prop:different-range} (note that $|\ak{m}_{F_n}|=1$ for any $n\in\bb{N}$).
\end{proof}

\begin{myprop}\label{prop:tate-sen}
	Let $D=\ca{O}_{F_\infty}\cap \ke(\ca{T})$, where $\ca{T}$ is Tate's normalized trace map \eqref{eq:para:tate-trace-1}.
	\begin{enumerate}
		\renewcommand{\labelenumi}{{\rm(\theenumi)}}
		\item The canonical $G_F$-equivariant morphism of $\ca{O}_F$-modules
		\begin{align}\label{eq:prop:tate-sen-1}
			\ca{O}_F\oplus D\longrightarrow \ca{O}_{F_\infty}
		\end{align}
		is injective with cokernel killed by $\pi$. In particular, the decomposition $F_\infty=F\oplus \ke(\ca{T})$ \eqref{eq:para:tate-trace-2} is obtained by inverting $p$ on \eqref{eq:prop:tate-sen-1}. \label{item:prop:tate-sen-1}
		\item Let $\tau\in \gal(F_\infty/F)$ be a topological generator \eqref{eq:para:tower-3}. The $\ca{O}_F$-linear endomorphism $\tau-1$ on $\ca{O}_{F_\infty}$ annihilates $\ca{O}_F$ and stabilizes $D$. Moreover, it induces an injection
		\begin{align}
			\tau-1:D\longrightarrow D
		\end{align}
		with cokernel killed by $p\pi$.\label{item:prop:tate-sen-2}
	\end{enumerate}
\end{myprop}
\begin{proof}
	(\ref{item:prop:tate-sen-1}) Since $\ca{T}=\colim_{n\in\bb{N}}p^{-n}(1+\tau+\cdots+\tau^{p^n-1}):F_\infty\to F$ commutes with the action of $G_F$ as $\gal(F_\infty/F)$ is commutative, $D=\ca{O}_{F_\infty}\cap\ke(\ca{T})$ is a $G_F$-stable $\ca{O}_F$-submodule of $\ca{O}_{F_\infty}$. It is clear that $D[1/p]=\ke(\ca{T})\subseteq F_\infty$. Thus, we see that $\ca{O}_F\oplus D\to \ca{O}_{F_\infty}$ is injective and induces the decomposition $F\oplus\ke(\ca{T})=F_\infty$ \eqref{eq:para:tate-trace-2} after inverting $p$. 
	
	For any $x\in F_\infty$, this decomposition gives $x=\ca{T}(x)+(1-\ca{T})(x)$. As $|\ca{T}(x)|\leq |\pi^{-1} x|$ by \ref{cor:trace-range}, we have $\ca{T}(x)\in\ca{O}_F$ (and thus $(1-\ca{T})(x)\in D$) if $x\in \pi\ca{O}_{F_\infty}$. This shows that the quotient of $\ca{O}_F\oplus D\subseteq \ca{O}_{F_\infty}$ is killed by $\pi$.
	
	(\ref{item:prop:tate-sen-2}) The morphism $\tau-1:D\to D$ is injective because the kernel of $\tau-1:F_\infty\to F_\infty$ is $F$. It remains to show that $p\pi D\subseteq (\tau-1)(D)$. Recall that $\tau-1$ induces an isomorphism $D[1/p]\iso D[1/p]$. Indeed, since $\tau-1$ is an $F$-linear injective endomorphism of $D[1/p]=\colim_{n\in\bb{N}}\ke(\ca{T}|_{\ca{F}_n})$, we see that it is an isomorphism by a dimension argument on each $G_F$-stable finite-dimensional subspace $\ke(\ca{T}|_{\ca{F}_n})$. Hence, for any $x\in D$, there exists $y\in D[1/p]$ with $(\tau-1)(y)=p\pi x$. Since $|y|\leq |p\pi|^{-1}\cdot |\tau(y)-y|=|x|\leq 1$ by \ref{lem:trace-tau}, we have $y\in\ca{O}_{F_\infty}$. Thus, $y\in \ca{O}_{F_\infty}\cap D[1/p]=D$ which completes the proof.
\end{proof}

\begin{mycor}\label{cor:galois-coh}
	With the notation in {\rm\ref{prop:tate-sen}}, for any $r\in\bb{N}$, the Galois cohomology group 
	\begin{align}
		H^q(\gal(F_\infty/F),D/p^rD)
	\end{align}
	is killed by $p\pi$ for any $q\in\{0,1\}$ and is zero for any $q\in \bb{Z}\setminus \{0,1\}$
\end{mycor}
\begin{proof}
	As $\gal(F_\infty/F)=\bb{Z}_p\tau$, $\rr\Gamma(\gal(F_\infty/F),D/p^rD)$ is represented by $D/p^rD\stackrel{\tau-1}{\longrightarrow} D/p^rD$ (\cite[\Luoma{2}.3.23]{abbes2016p}). Since $\tau-1:D\to D$ is injective whose cokernel is killed by $p\pi$ (\ref{prop:tate-sen}.(\ref{item:prop:tate-sen-2})), the kernel and cokernel of $\tau-1:D/p^rD\to D/p^rD$ are also killed by $p\pi$ by snake lemma (or \ref{lem:pi-isom}), which completes the proof.
\end{proof}

\section{Perfectoid Towers and Galois Cohomology over Cyclotomic Field}\label{sec:perfd}
By our study on the structure of completed differentials of valuation rings (see \ref{lem:differential}), we construct a perfectoid tower by iteratively adjoining $p$-power roots. Then, we compute the Galois cohomology by iteratively applying the results in Section \ref{sec:kummer} (see \ref{prop:perfd-tower-coh}).

\begin{mylem}\label{lem:differential}
	Let $K\to F$ be an extension of valuation fields of height $1$ with $K$ pre-perfectoid and with finite transcendental degree $\mrm{trdeg}_K(F)<\infty$. Then, the canonical morphism of $\ca{O}_{\widehat{F}}$-modules
	\begin{align}\label{eq:lem:differential-1}
		\ca{O}_{\widehat{F}}\otimes_{\ca{O}_F}\Omega^1_{\ca{O}_F/\ca{O}_K}\longrightarrow \widehat{\Omega}^1_{\ca{O}_F/\ca{O}_K}
	\end{align}
	is surjective, where $\widehat{\Omega}^1_{\ca{O}_F/\ca{O}_K}$ denotes the $\pi$-adic completion of $\Omega^1_{\ca{O}_F/\ca{O}_K}$ (which does not depend on the choice of $\pi\in\ak{m}_F\setminus 0$). Moreover, the $\ca{O}_{\widehat{F}}$-module $\widehat{\Omega}^1_{\ca{O}_F/\ca{O}_K}$ is torsion-free, almost finitely generated, and admits a $\pi$-basis {\rm(\ref{defn:pi-basis})} $\df t_1,\dots,\df t_d$ with $t_1,\dots,t_d\in \ca{O}_F^\times$ and $d=\mrm{rank}_{\ca{O}_{\widehat{F}}}(\widehat{\Omega}^1_{\ca{O}_F/\ca{O}_K})\leq\mrm{rank}_{\ca{O}_F}(\Omega^1_{\ca{O}_F/\ca{O}_K})$ {\rm(\ref{defn:val-rank})} for any $\pi\in\ak{m}_F\setminus 0$. 
\end{mylem}
\begin{proof}
	As $\Omega^1_{F/K}$ is a finite free $F$-module of rank $\mrm{trdeg}_K(F)$, we see that $\Omega^1_{\ca{O}_F/\ca{O}_K}$ is a torsion-free $\ca{O}_F$-module of finite rank (\ref{cor:differential}). Thus, the surjectivity of \eqref{eq:lem:differential-1} follows from \ref{thm:completion}. Moreover, $\widehat{\Omega}^1_{\ca{O}_F/\ca{O}_K}$ is torsion-free and almost finitely generated over $\ca{O}_{\widehat{F}}$. 
	
	Recall that the $\ca{O}_F$-module $\Omega^1_{\ca{O}_F/\ca{O}_K}$ is generated by $\{\df t\}_{t\in\ca{O}_F}$ by definition. Since for $t\in\ak{m}_F$, we have $\df t=\df(1+t)$ and $1+t\in\ca{O}_F^{\times}=\ca{O}_F\setminus \ak{m}_F$. Hence, $\{\df t\}_{t\in\ca{O}_F^\times}$ generates $\Omega^1_{\ca{O}_F/\ca{O}_K}$ and thus also $\widehat{\Omega}^1_{\ca{O}_F/\ca{O}_K}$ by the surjectivity of \eqref{eq:lem:differential-1}. Therefore, for any $\pi\in\ak{m}_F$, there are finitely many elements $t_1,\dots,t_d\in\ca{O}_F^\times$ such that $\df t_1,\dots,\df t_d$ form a $\pi$-basis of $\widehat{\Omega}^1_{\ca{O}_F/\ca{O}_K}$ by \ref{prop:generator-basis}, where $d=\mrm{rank}_{\ca{O}_{\widehat{F}}}(\widehat{\Omega}^1_{\ca{O}_F/\ca{O}_K})\leq \mrm{rank}_{\ca{O}_F}(\Omega^1_{\ca{O}_F/\ca{O}_K})$.
\end{proof}

\begin{mypara}\label{para:notation-perfd-tower}
	In this section, we fix a pre-perfectoid field $K$ extension of $\bb{Q}_p$ containing a compatible system of primitive $p$-power roots of unity $(\zeta_{p^n})_{n\in\bb{N}}$, a Henselian valuation field $F$ of height $1$ extension of $K$ of finite transcendental degree $\mrm{trdeg}_K(F)<\infty$, and an algebraic closure $\overline{F}$ of $F$. We put $G_F=\gal(\overline{F}/F)$ the absolute Galois group of $F$ which naturally acts on the Henselian valuation ring $\ca{O}_{\overline{F}}$.
	
	Let $d=\mrm{rank}_{\ca{O}_{\widehat{F}}}(\widehat{\Omega}^1_{\ca{O}_F/\ca{O}_K})$. We fix an element $\pi\in \ak{m}_K$ such that $\pi^{2^d}\notin p\ca{O}_K$. We also fix finitely many elements $t_1,\dots,t_d\in\ca{O}_F^\times$ such that $\df t_1,\dots,\df t_d$ form a $\pi\ak{m}_K$-basis of $\widehat{\Omega}^1_{\ca{O}_F/\ca{O}_K}$ (which exist by \ref{lem:differential}), namely
	\begin{align}\label{eq:notation-perfd-tower-1}
		\pi\ak{m}_K\cdot \widehat{\Omega}^1_{\ca{O}_F/\ca{O}_K}\subseteq \ca{O}_{\widehat{F}}\cdot\df t_1\oplus \cdots\oplus \ca{O}_{\widehat{F}}\cdot\df t_d\subseteq \widehat{\Omega}^1_{\ca{O}_F/\ca{O}_K}.
	\end{align}
	For any $1\leq i\leq d$, we fix a compatible system of $p$-power roots $(t_{i,p^n})_{n\in \bb{N}}$ of $t_i$ contained in $\overline{F}$ (where $t_{i,1}=t_i$). For any $\underline{n}=(n_1,\dots,n_d)\in\bb{N}^d$, we denote the finite extension of $F$ generated by $t_{1,p^{n_1}},\dots,t_{d,p^{n_d}}$ by
	\begin{align}\label{eq:notation-perfd-tower-2}
		 F_{\underline{n}}=F(t_{1,p^{n_1}},\dots,t_{d,p^{n_d}}).
	\end{align}
	It is clear that these fields form a system of finite Galois extension of $F$ over the directed partially ordered set $\bb{N}^d$ (\ref{para:product}). We extend this notation for one of the components of $\underline{n}$ being $\infty$ by taking the corresponding filtered union, and we omit the index $\underline{n}$ if $\underline{n}=\underline{0}$.
\end{mypara}

\begin{myprop}\label{prop:perfd-tower-basis}
	For any $\underline{n}=(n_1,\dots,n_d)\in (\bb{N}\cup\{\infty\})^d$, if $n_{i_1},\dots,n_{i_k}$ are the finite components of $\underline{n}$, then $\df t_{i_1,p^{n_{i_1}}},\dots,\df t_{i_k,p^{n_{i_k}}}$ form a $\pi^{2^i}\ak{m}_K$-basis of $\widehat{\Omega}^1_{\ca{O}_{F_{\underline{n}}}/\ca{O}_K}$, where $i$ is the number of nonzero components of $\underline{n}$.
\end{myprop}
\begin{proof}
	We prove by induction on $0\leq i\leq d$ that the differentials corresponding to the finite components of $\underline{n}^{(i)}=(n_1,\dots,n_i,0,\dots,0)\in (\bb{N}\cup\{\infty\})^d$ form a $\pi^{2^i}\ak{m}_K$-basis of $\widehat{\Omega}^1_{\ca{O}_{F_{\underline{n}^{(i)}}}/\ca{O}_K}$. The case $i=0$ holds by the assumption \eqref{eq:notation-perfd-tower-1}. In general, for $i>0$, suppose the claim holds for $i-1$ and by symmetry we may assume that $\infty=n_1=\cdots=n_{r-1}>n_r\geq \cdots \geq n_{i-1}$ for some $1\leq r\leq i$. For simplicity, we put $E_n=F_{(n_1,\dots,n_{i-1},n,0,\dots,0)}$ for any $n\in\bb{N}\cup\{\infty\}$ and $E=E_0$. Then, by induction hypothesis, $\pi^{2^{i-1}}\ak{m}_K\cdot\widehat{\Omega}^1_{\ca{O}_E/\ca{O}_K}$ is contained in the following $\ca{O}_{\widehat{E}}$-submodule of $\widehat{\Omega}^1_{\ca{O}_E/\ca{O}_K}$,
	\begin{align}\label{eq:prop:perfd-tower-basis-1}
		\ca{O}_{\widehat{E}}\cdot \df t_{r,p^{n_r}}\oplus \cdots \oplus \ca{O}_{\widehat{E}}\cdot \df t_{i-1,p^{n_{i-1}}}\oplus \ca{O}_{\widehat{E}}\cdot \df t_i\oplus \ca{O}_{\widehat{E}}\cdot \df t_{i+1}\oplus\cdots\oplus \ca{O}_{\widehat{E}}\cdot \df t_d.
	\end{align}
	Note that $\df t_i\notin\pi^{2^{i-1}}\ak{m}_K\cdot\Omega^1_{\ca{O}_E/\ca{O}_K}$ by \ref{lem:pi-basis} (where we used $\ak{m}_{\widehat{E}}=\ak{m}_K\ca{O}_{\widehat{E}}$ \eqref{eq:para:notation-almost-2}) and $\pi^{2^{i-1}}\notin p\ca{O}_K$ by our choice of $\pi$ in \ref{para:notation-perfd-tower}. The tower of Kummer extensions $(E_n)_{n\in\bb{N}}$ lies in the situation \ref{para:tower} and thus we can apply all the results in Section \ref{sec:kummer}.
	
	As $\Omega^1_{\ca{O}_E/\ca{O}_K}$ is torsion-free (\ref{cor:differential}), we put $\ak{a}=E\cdot \df t_i \cap \Omega^1_{\ca{O}_E/\ca{O}_K}\subseteq \Omega^1_{E/K}$. We have 
	\begin{align}\label{eq:prop:perfd-tower-basis-2}
		\ca{O}_E\subseteq \ak{a}\subseteq \pi^{-2^{i-1}}\ca{O}_E \subseteq p^{-1}\ca{O}_E
	\end{align}
	as $\df t_i\notin\pi^{2^{i-1}}\ak{m}_K\cdot\Omega^1_{\ca{O}_E/\ca{O}_K}$. By flat base change, we see that $\ak{a}\ca{O}_{E_n}=E_n\cdot \df t_i \cap (\ca{O}_{E_n}\otimes_{\ca{O}_E}\Omega^1_{\ca{O}_E/\ca{O}_K})\subseteq \Omega^1_{E_n/K}$. Let $M_n$ be the $\ca{O}_{E_n}$-submodule of $\Omega^1_{\ca{O}_{E_n}/\ca{O}_K}$ generated by $\df t_{i,p^n}$ and $\Omega^1_{\ca{O}_E/\ca{O}_K}$. Note that $\ca{O}_{E_n}\cdot\df t_{i,p^n}=p^{-n}\ca{O}_{E_n}\cdot\df t_i$ as $t_i\in\ca{O}_E^\times$. Then, for any integer $n\geq 1$, there is an injective morphism of exact sequences by \ref{thm:differential}.(\ref{item:thm:differential-1}),
	\begin{align}\label{eq:prop:perfd-tower-basis-3}
		\xymatrix{
			0\ar[r]&\ca{O}_{E_n}\otimes_{\ca{O}_E}\Omega^1_{\ca{O}_E/\ca{O}_K}\ar[r]\ar@{=}[d]&M_n\ar[d]\ar[r]&p^{-n}\ca{O}_{E_n}/\ak{a}\ca{O}_{E_n}\ar[r]\ar[d]^-{\cdot \df t_i}&0\\
			0\ar[r]& \ca{O}_{E_n}\otimes_{\ca{O}_E}\Omega^1_{\ca{O}_E/\ca{O}_K}\ar[r]& \Omega^1_{\ca{O}_{E_n}/\ca{O}_K}\ar[r]&\Omega^1_{\ca{O}_{E_n}/\ca{O}_E}\ar[r]& 0
		}
	\end{align}
	Applying \ref{cor:different-dt} to $K\to E$, we see that the cokernel $Q_n$ of the right vertical arrow is killed by $\pi^{2^{i-1}}\ak{m}_K$. Taking $p$-adic completion, we still obtain an exact sequence (\cite[\href{https://stacks.math.columbia.edu/tag/0BNG}{0BNG}]{stacks-project})
	\begin{align}\label{eq:prop:perfd-tower-basis-4}
		0\longrightarrow \widehat{M_n}\longrightarrow \widehat{\Omega}^1_{\ca{O}_{E_n}/\ca{O}_K}\longrightarrow Q_n\longrightarrow 0.
	\end{align}
	Similarly, as $p^{-n}\ca{O}_{E_n}/\ak{a}\ca{O}_{E_n}$ is killed by $p^n$, there is an exact sequence
	\begin{align}\label{eq:prop:perfd-tower-basis-5}
		0\longrightarrow \ca{O}_{\widehat{E_n}}\otimes_{\ca{O}_{\widehat{E}}}\widehat{\Omega}^1_{\ca{O}_E/\ca{O}_K}\longrightarrow \widehat{M_n}\longrightarrow p^{-n}\ca{O}_{E_n}/\ak{a}\ca{O}_{E_n}\longrightarrow 0,
	\end{align}
	where we used the fact that $\ca{O}_{E_n}\widehat{\otimes}_{\ca{O}_E}\Omega^1_{\ca{O}_E/\ca{O}_K}=\ca{O}_{\widehat{E_n}}\otimes_{\ca{O}_{\widehat{E}}}\widehat{\Omega}^1_{\ca{O}_E/\ca{O}_K}$ by \ref{cor:sep-comp} (note that $\widehat{\Omega}^1_{\ca{O}_E/\ca{O}_K}$ is almost finitely generated by \ref{lem:differential}). Consider the canonical morphism of exact sequences induced by \eqref{eq:prop:perfd-tower-basis-1},
	\begin{align}\label{eq:prop:perfd-tower-basis-6}
		\xymatrix{
			0\ar[r]&\ca{O}_{\widehat{E_n}}^{\oplus d-r}\oplus \ca{O}_{\widehat{E_n}}\cdot \df t_i\ar[r]\ar[d]&\ca{O}_{\widehat{E_n}}^{\oplus d-r}\oplus p^{-n}\ca{O}_{\widehat{E_n}}\cdot \df t_i\ar[r]\ar[d]&p^{-n}\ca{O}_{\widehat{E_n}}/\ca{O}_{\widehat{E_n}}\ar[r]\ar[d]&0\\
			0\ar[r]& \ca{O}_{\widehat{E_n}}\otimes_{\ca{O}_{\widehat{E}}}\widehat{\Omega}^1_{\ca{O}_E/\ca{O}_K}\ar[r]& \widehat{M_n}\ar[r]& p^{-n}\ca{O}_{\widehat{E_n}}/\ak{a}\ca{O}_{\widehat{E_n}}\ar[r]& 0
		}
	\end{align} 
	where the left vertical arrow is injective whose cokernel is killed by $\pi^{2^{i-1}}\ak{m}_K$. As $\widehat{M_n}$ is torsion-free by \eqref{eq:prop:perfd-tower-basis-4}, we see that the middle vertical arrow is also injective by inverting $p$. Therefore, the cokernel of the middle vertical arrow is also killed by $\pi^{2^{i-1}}\ak{m}_K$. This implies that $\pi^{2^i}\ak{m}_K\cdot\widehat{\Omega}^1_{\ca{O}_{E_n}/\ca{O}_K}$ is contained in the following $\ca{O}_{\widehat{E_n}}$-submodule of $\widehat{\Omega}^1_{\ca{O}_{E_n}/\ca{O}_K}$ by \eqref{eq:prop:perfd-tower-basis-4},
	\begin{align}\label{eq:prop:perfd-tower-basis-7}
		\ca{O}_{\widehat{E_n}}\cdot \df t_{r,p^{n_r}}\oplus \cdots \oplus \ca{O}_{\widehat{E_n}}\cdot \df t_{i-1,p^{n_{i-1}}}\oplus \ca{O}_{\widehat{E_n}}\cdot \df t_{i,p^n}\oplus \ca{O}_{\widehat{E_n}}\cdot \df t_{i+1}\oplus\cdots\oplus \ca{O}_{\widehat{E_n}}\cdot \df t_d.
	\end{align}
	In other words, $\df t_{r,p^{n_r}},\dots, \df t_{i-1,p^{n_{i-1}}}, \df t_{i,p^n},\df t_{i+1},\dots,\df t_d$ form a $\pi^{2^i}\ak{m}_K$-basis of $\widehat{\Omega}^1_{\ca{O}_{E_n}/\ca{O}_K}$. This completes the induction when $n_i$ is finite.
	
	For $n_i$ infinite, taking filtered colimit of \eqref{eq:prop:perfd-tower-basis-3} over $n\in\bb{N}$, we obtain an injective morphism of exact sequences
	\begin{align}\label{eq:prop:perfd-tower-basis-8}
		\xymatrix{
			0\ar[r]&\ca{O}_{E_\infty}\otimes_{\ca{O}_E}\Omega^1_{\ca{O}_E/\ca{O}_K}\ar[r]\ar@{=}[d]&M_\infty\ar[d]\ar[r]&E_\infty/\ak{a}\ca{O}_{E_\infty}\ar[r]\ar[d]^-{\cdot \df t_i}&0\\
			0\ar[r]& \ca{O}_{E_\infty}\otimes_{\ca{O}_E}\Omega^1_{\ca{O}_E/\ca{O}_K}\ar[r]& \Omega^1_{\ca{O}_{E_\infty}/\ca{O}_K}\ar[r]&\Omega^1_{\ca{O}_{E_\infty}/\ca{O}_E}\ar[r]& 0
		}
	\end{align}
	The cokernel $Q_\infty$ of the right vertical arrow is still killed by $\pi^{2^{i-1}}\ak{m}_K$. Taking $p$-adic completion, we still obtain an exact sequence (\cite[\href{https://stacks.math.columbia.edu/tag/0BNG}{0BNG}]{stacks-project})
	\begin{align}\label{eq:prop:perfd-tower-basis-9}
		0\longrightarrow \widehat{M_\infty}\longrightarrow \widehat{\Omega}^1_{\ca{O}_{E_\infty}/\ca{O}_K}\longrightarrow Q_\infty\longrightarrow 0.
	\end{align}
	Applying the derived functor $\rr\ho_{\bb{Z}}(p^{-r}\bb{Z}/\bb{Z},-)$ to the first row of \eqref{eq:prop:perfd-tower-basis-8}, we obtain an exact sequence (see the proof of \ref{lem:rhom})
	\begin{align}\label{eq:prop:perfd-tower-basis-10}
		0=M_\infty[p^r]\longrightarrow p^{-r}\ak{a}\ca{O}_{E_\infty}/\ak{a}\ca{O}_{E_\infty}\longrightarrow (\ca{O}_{E_\infty}\otimes \Omega^1_{\ca{O}_E/\ca{O}_K})/p^r\longrightarrow M_\infty/p^r\longrightarrow(E_\infty/\ak{a}\ca{O}_{E_\infty})/p^r=0.
	\end{align}
	As the inverse system $(p^{-r}\ak{a}\ca{O}_{E_\infty}/\ak{a}\ca{O}_{E_\infty})_{r\in\bb{N}}$ satisfies the Mittag-Leffler condition, taking limit over $r\in\bb{N}$, we obtain an exact sequence
	\begin{align}\label{eq:prop:perfd-tower-basis-11}
		\xymatrix{
			0\ar[r]& \ak{a}\ca{O}_{\widehat{E_\infty}}\ar[r]^-{\cdot \df t_i}& \ca{O}_{\widehat{E_\infty}}\otimes_{\ca{O}_{\widehat{E}}}\widehat{\Omega}^1_{\ca{O}_E/\ca{O}_K}\ar[r]&\widehat{M_\infty}\ar[r]& 0,
		}
	\end{align}
	where we used the fact that $\ca{O}_{E_\infty}\widehat{\otimes}_{\ca{O}_E}\Omega^1_{\ca{O}_E/\ca{O}_K}=\ca{O}_{\widehat{E_\infty}}\otimes_{\ca{O}_{\widehat{E}}}\widehat{\Omega}^1_{\ca{O}_E/\ca{O}_K}$ by \ref{cor:sep-comp}. Consider the canonical morphism of exact sequences induced by \eqref{eq:prop:perfd-tower-basis-1},
	\begin{align}\label{eq:prop:perfd-tower-basis-12}
		\xymatrix{
			0\ar[r]&\ca{O}_{\widehat{E_\infty}}\cdot \df t_i\ar[r]\ar[d]&\ca{O}_{\widehat{E_\infty}}^{\oplus d-r}\oplus \ca{O}_{\widehat{E_\infty}}\cdot \df t_i\ar[r]\ar[d]&\ca{O}_{\widehat{E_\infty}}^{\oplus d-r}\ar[r]\ar[d]&0\\
			0\ar[r]& \ak{a}\ca{O}_{\widehat{E_\infty}}\cdot \df t_i\ar[r]& \ca{O}_{\widehat{E_\infty}}\otimes_{\ca{O}_{\widehat{E}}}\widehat{\Omega}^1_{\ca{O}_E/\ca{O}_K}\ar[r]&\widehat{M_\infty}\ar[r]& 0
		}
	\end{align} 
	where the middle vertical arrow is injective whose cokernel is killed by $\pi^{2^{i-1}}\ak{m}_K$. As $\widehat{M_\infty}$ is torsion-free by \eqref{eq:prop:perfd-tower-basis-9}, we see that the right vertical arrow is also injective by inverting $p$. On the other hand, the cokernel of the right vertical arrow is also killed by $\pi^{2^{i-1}}\ak{m}_K$. This implies that $\pi^{2^i}\ak{m}_K\cdot\widehat{\Omega}^1_{\ca{O}_{E_\infty}/\ca{O}_K}$ is contained in the following $\ca{O}_{\widehat{E_\infty}}$-submodule of $\widehat{\Omega}^1_{\ca{O}_{E_\infty}/\ca{O}_K}$ by \eqref{eq:prop:perfd-tower-basis-9},
	\begin{align}\label{eq:prop:perfd-tower-basis-13}
		\ca{O}_{\widehat{E_\infty}}\cdot \df t_{r,p^{n_r}}\oplus \cdots \oplus \ca{O}_{\widehat{E_\infty}}\cdot \df t_{i-1,p^{n_{i-1}}}\oplus \ca{O}_{\widehat{E_\infty}}\cdot \df t_{i+1}\oplus\cdots\oplus \ca{O}_{\widehat{E_\infty}}\cdot \df t_d.
	\end{align}
	In other words, $\df t_{r,p^{n_r}},\dots, \df t_{i-1,p^{n_{i-1}}}, \df t_{i+1},\dots,\df t_d$ form a $\pi^{2^i}\ak{m}_K$-basis of $\widehat{\Omega}^1_{\ca{O}_{E_\infty}/\ca{O}_K}$. This completes the induction when $n_i$ is infinite.
\end{proof}

\begin{mycor}\label{cor:perfd-tower-perfd}
	The valuation field $F_{\underline{\infty}}$ is pre-perfectoid {\rm(\ref{para:notation-perfd})}.
\end{mycor}
\begin{proof}
	By \ref{prop:perfd-tower-basis}, we see that $\widehat{\Omega}^1_{\ca{O}_{F_{\underline{\infty}}}/\ca{O}_K}=0$, i.e., $\Omega^1_{\ca{O}_{F_{\underline{\infty}}}/\ca{O}_K}$ is $p$-divisible. This implies that $F_{\underline{\infty}}$ is pre-perfectoid by \cite[6.6.6]{gabber2003almost} (see \cite[10.17]{he2024purity}).
\end{proof}

\begin{mycor}\label{cor:perfd-tower-gal}
	The continuous group homomorphism
	\begin{align}\label{eq:cor:perfd-tower-gal-1}
		\xi=(\xi_1,\dots,\xi_d):G_F\longrightarrow \bb{Z}_p^d,
	\end{align}
	characterized by $\tau(t_{i,p^n})=\zeta_{p^n}^{\xi_i(\tau)}t_{i,p^n}$ for any $\tau\in G_F$, $n\in\bb{N}$ and $1\leq i\leq d$, induces an isomorphism for any $\underline{n}=(n_1,\dots,n_d)\in\bb{N}^d$,
	\begin{align}\label{eq:cor:perfd-tower-gal-2}
		\xi:\gal(F_{\underline{\infty}}/F_{\underline{n}})\iso p^{n_1}\bb{Z}_p\times \cdots\times p^{n_d}\bb{Z}_p.
	\end{align}
\end{mycor}
\begin{proof}
	It suffices to show that for any $1\leq i\leq d$, the homomorphism $\xi$ induces an isomorphism
	\begin{align}
		\gal(F_{(\infty,\dots,\infty,n_{i+1},\dots,n_d)}/F_{(\infty,\dots,n_i,n_{i+1},\dots,n_d)})\iso p^{n_i}\bb{Z}_p.
	\end{align}
	We put $\underline{m}=(\infty,\dots,\infty,n_i,n_{i+1},\dots,n_d)\in(\bb{N}\cup\{\infty\})^d$. Since $\df t_{i,p^{n_i}},\dots,\df t_{d,p^{n_d}}$ form a $\pi^{2^d}\ak{m}_K$-basis of $\widehat{\Omega}^1_{\ca{O}_{F_{\underline{m}}}/\ca{O}_K}$ by \ref{prop:perfd-tower-basis}, they are not in $\pi^{2^d}\ak{m}_K\cdot \Omega^1_{\ca{O}_{F_{\underline{m}}}/\ca{O}_K}$ by \ref{lem:pi-basis}. As $\pi^{2^d}\notin p\ca{O}_K$, we can apply \ref{para:tower} to $K\to F_{\underline{m}}$ so that the conclusion follows from \eqref{eq:para:tower-3}.
\end{proof}

\begin{mypara}\label{para:notation-J}
	We take $\tau_1,\dots,\tau_d\in \gal(F_{\underline{\infty}}/F)$ corresponding to the standard $\bb{Z}_p$-basis of $\bb{Z}_p^d$ via $\xi$ \eqref{eq:cor:perfd-tower-gal-1}. We put $I=\{1,2,\dots,d\}$ and for any subset $J\subseteq I$, we put 
	\begin{align}\label{para:notation-J-1}
		\underline{J}=(n_1,\dots,n_d),\quad n_j=\left\{\begin{array}{ll}
			\infty&\trm{if }j\in J,\\
			0&\trm{if }j\notin J.
		\end{array}\right.
	\end{align}
	In particular, $F_{\underline{\emptyset}}=F$ and $F_{\underline{I}}=F_{\underline{\infty}}$. Consider Tate's normalized trace map for any $i\in I$ defined in \ref{para:tate-trace},
	\begin{align}\label{para:notation-J-2}
		\ca{T}_i=\colim_{n\in\bb{N}}p^{-n}(1+\tau_i+\dots+\tau_i^{p^n-1}):F_{\underline{I}}\longrightarrow F_{\underline{I\setminus\{i\}}}.
	\end{align}
	Since $\tau_1,\dots,\tau_d$ commute with each other, regarding these $\ca{T}_i$ as $F$-linear idempotent operators of $F_{\underline{I}}$, we have
	\begin{align}\label{para:notation-J-3}
		\ca{T}_i\ca{T}_j=\ca{T}_j\ca{T}_i,\quad \ca{T}_i\tau_j=\tau_j\ca{T}_i,\quad \forall\ 1\leq i,j\leq d.
	\end{align}
	In particular, there is a decomposition of $F_{\underline{I}}$ into $G_F$-stable $F$-submodules,
	\begin{align}\label{para:notation-J-4}
		F_{\underline{I}}=\bigoplus_{J\subseteq I} \im(\prod_{j\in J}(1-\ca{T}_j)\prod_{j\notin J}\ca{T}_j)
	\end{align}
	by decomposing each $x\in F_{\underline{\infty}}$ into
	\begin{align}\label{para:notation-J-5}
		x=\sum_{J\subseteq I} (\prod_{j\in J}(1-\ca{T}_j)\prod_{j\notin J}\ca{T}_j)(x).
	\end{align}
	Note that $\im(\prod_{j\in J}(1-\ca{T}_j)\prod_{j\notin J}\ca{T}_j)= \bigcap_{j\in J}\ke(\ca{T}_j)\cap \bigcap_{j\notin J}\ke(1-\ca{T}_j)\subseteq F_{\underline{J}}$.
	
	We put
	\begin{align}\label{para:notation-J-7}
		D_J=\ca{O}_{F_{\underline{I}}}\cap \im(\prod_{j\in J}(1-\ca{T}_j)\prod_{j\notin J}\ca{T}_j),
	\end{align}
	which is a $G_F$-stable $\ca{O}_F$-submodule of $\ca{O}_{F_{\underline{J}}}$. In particular, the decomposition \eqref{para:notation-J-4} identifies with $F_{\underline{I}}=\bigoplus_{J\subseteq I}D_J[1/p]$ and thus $D_J=\ca{O}_{F_{\underline{I}}}\cap D_J[1/p]$. Moreover, for any subsets $J'\subseteq J\subseteq I$, we put
	\begin{align}\label{para:notation-J-9}
		D_{(J,J')}=\bigoplus_{\substack{J''\subseteq J\\J''\nsubseteq J'}} D_{J''}.
	\end{align}
\end{mypara}

\begin{mylem}\label{lem:D_J}
	With the notation in {\rm\ref{para:notation-J}}, let $J$ be a subset of $I$.
	\begin{enumerate}
		\renewcommand{\labelenumi}{{\rm(\theenumi)}}
		\item The $F$-module $D_J[1/p]$ admits a basis 
		\begin{align}\label{eq:lem:D_J-1}
			\{\prod_{j\in J}t_{j,p^{n_j}}^{k_j}\ |\ n_j\geq 1,\ 1\leq k_j<p^{n_j},\ p\nmid k_j\}.
		\end{align}\label{item:lem:D_J-1}
		\item The canonical morphism $\bigoplus_{J'\subseteq J} D_{J'}\to \ca{O}_{F_{\underline{J}}}$ is injective and induces a decomposition by inverting $p$,
		\begin{align}\label{eq:lem:D_J-2}
			F_{\underline{J}}=\im(\prod_{j\notin J}\ca{T}_j)=\bigoplus_{J'\subseteq J} D_{J'}[1/p].
		\end{align}
		\label{item:lem:D_J-2}
		\item For any subset $J'\subseteq J$, the canonical morphism $\ca{O}_{F_{\underline{J'}}}\oplus D_{(J,J')}\to \ca{O}_{F_{\underline{J}}}$ is injective and induces a decomposition by inverting $p$,
		\begin{align}\label{eq:lem:D_J-3}
			F_{\underline{J}}=F_{\underline{J'}}\oplus F_{\underline{J}}\cap \ke(\prod_{j\in J\setminus J'}\ca{T}_j)=F_{\underline{J'}}\oplus D_{(J,J')}[1/p].
		\end{align}\label{item:lem:D_J-3}
	\end{enumerate}
\end{mylem}
\begin{proof}
	(\ref{item:lem:D_J-1}) For any $\lambda\in \bb{Z}[1/p]/\bb{Z}$, it is uniquely expressed as $\lambda=k/p^n$ where $k,n\in\bb{N}$ such that $0\leq k< p^n$ and $\gcd(k,p^n)=1$. Then, we put $t_i^\lambda=t_{i,p^n}^k$ for any $i\in I$. Notice that $\ak{B}=\{\prod_{i\in I}t_i^{\lambda_i}\ |\ \lambda_i\in \bb{Z}[1/p]/\bb{Z}\}$ forms an $F$-basis of $F_{\underline{I}}$. Consider the decomposition  
	\begin{align}\label{eq:lem:D_J-4}
		\ak{B}=\{\prod_{i\in I}t_i^{\lambda_i}\ |\ \lambda_i\in \bb{Z}[1/p]/\bb{Z}\}=\coprod_{J\subseteq I}\{\prod_{j\in J}t_j^{\lambda_j}\ |\ \lambda_j\in \bb{Z}[1/p]/\bb{Z}\trm{ and }\lambda_j\neq 0\}=\coprod_{J\subseteq I}\ak{B}_J.
	\end{align}
	Notice that for any $j\in I$, we have
	\begin{align}\label{eq:lem:D_J-5}
		\ca{T}_j(t_i^\lambda)=\left\{\begin{array}{ll}
			t_i^\lambda&\trm{if }j\neq i \trm{ or }\lambda=0,\\
			0&\trm{if }j=i\trm{ and }\lambda\neq 0.
		\end{array}\right.
	\end{align}
	This shows that $\ak{B}_J\subseteq \im(\prod_{j\in J}(1-\ca{T}_j)\prod_{j\notin J}\ca{T}_j)=D_J[1/p]$. Then, we deduce from the decompositions $\ak{B}=\coprod_{J\subseteq I}\ak{B}_J$ \eqref{eq:lem:D_J-4} and $F_{\underline{I}}=\bigoplus_{J\subseteq I} D_J[1/p]$ \eqref{para:notation-J-4} that $\ak{B}_J$ forms an $F$-basis of $D_J[1/p]$.
	
	(\ref{item:lem:D_J-2}) It follows directly from (\ref{item:lem:D_J-1}) and the fact that $\coprod_{J'\subseteq J}\ak{B}_{J'}$ forms an $F$-basis of $F_{\underline{J}}$.
	
	(\ref{item:lem:D_J-3}) It follows directly from (\ref{item:lem:D_J-2}).
\end{proof}

\begin{mylem}\label{lem:perfd-tower-norm}
	With the notation in {\rm\ref{para:notation-J}}, for any subset $J\subseteq I$, the quotient of the inclusion $\bigoplus_{J'\subseteq J} D_{J'}\subseteq\ca{O}_{F_{\underline{J}}}$ is killed by $\pi^{2^{|J|}-1}$, where $|J|$ is the cardinality of $J$.
\end{mylem}
\begin{proof}
	Firstly, we claim that for any $j\in J$ and $x\in F_{\underline{J}}$, we have
	\begin{align}\label{eq:lem:perfd-tower-norm-1}
		|\ca{T}_j(x)|\leq |\pi^{-2^{|J|-1}}x|.
	\end{align}
	Indeed, notice that $\{\df t_k\}_{k\notin J'}$ (where $J'=J\setminus\{j\}$) forms a $\pi^{2^{|J'|}}\ak{m}_K$-basis of $\widehat{\Omega}^1_{\ca{O}_{F_{\underline{J'}}}/\ca{O}_K}$ by \ref{prop:perfd-tower-basis}. In particular, we have $\df t_j\notin \pi^{2^{|J'|}}\ak{m}_K\cdot \Omega^1_{\ca{O}_{F_{\underline{J'}}}/\ca{O}_K}$ by \ref{lem:pi-basis}. Thus, the claim follows from applying \ref{cor:trace-range} to the extensions $K\to F_{\underline{J'}}\to F_{\underline{J}}=\bigcup_{n\in\bb{N}}F_{\underline{J'}}(t_{j,p^n})$.
	
	By repeatedly applying the claim above, we see that for any $J\subseteq I$ and $x\in F_{\underline{J}}$,
	\begin{align}\label{eq:lem:perfd-tower-norm-2}
		|(\prod_{j\in J}\ca{T}_j)(x)|\leq |\pi^{-2^{|J|-1}} \pi^{-2^{|J|-2}} \cdots \pi^{-2^0}x|=|\pi^{1-2^{|J|}}x|.
	\end{align}
	On the other hand, we have
	\begin{align}\label{eq:lem:perfd-tower-norm-3}
		x=\sum_{J'\subseteq I} (\prod_{j\in J'}(1-\ca{T}_j)\prod_{j\notin J'}\ca{T}_j)(x)=\sum_{J'\subseteq J} (\prod_{j\in J'}(1-\ca{T}_j)\prod_{j\in J\setminus J'}\ca{T}_j)(x),
	\end{align}
	where the first decomposition is \eqref{para:notation-J-5} and the second equality follows from the fact that $(1-\ca{T}_j)(x)=0$ for any $j\notin J$. Since each $(\prod_{j\in J'}(1-\ca{T}_j)\prod_{j\in J\setminus J'}\ca{T}_j)(x)$ is a linear combination of $\{(\prod_{j\in J''}\ca{T}_j)(x)\}_{J''\subseteq J}$, we deduce from \eqref{eq:lem:perfd-tower-norm-2} that
	\begin{align}
		|(\prod_{j\in J'}(1-\ca{T}_j)\prod_{j\in J\setminus J'}\ca{T}_j)(x)|\leq |\pi^{1-2^{|J|}}x|.
	\end{align}
	Therefore, for any $x\in \pi^{2^{|J|}-1}\ca{O}_{F_{\underline{J}}}$, the decomposition \eqref{eq:lem:perfd-tower-norm-3} lies in $\oplus_{J'\subseteq J}D_{J'}$, which completes the proof.
\end{proof}

\begin{mylem}\label{lem:perfd-tower-coh}
	With the notation in {\rm\ref{para:notation-J}}, for any subsets $J'\subseteq J\subseteq I$ and $r\in\bb{N}$, the Galois cohomology group
	\begin{align}
		H^q(\gal(F_{\underline{J}}/F_{\underline{J'}}),D_{(J,J')}/p^rD_{(J,J')})
	\end{align}
	is killed by $p^2\pi^{2^{|J|+2}}$ if $0\leq q\leq |J\setminus J'|$ and is zero otherwise.
\end{mylem}
\begin{proof}
	We follow the proof of \cite[2-1]{hyodo1986hodge}. We take induction on $|J\setminus J'|$. For simplicity, we put $G_{(J,J')}=\gal(F_{\underline{J}}/F_{\underline{J'}})$ and $E_{(J,J')}=D_{(J,J')}/p^rD_{(J,J')}$. As the $p$-cohomological dimension of $G_{(J,J')}\cong \bb{Z}_p^{J\setminus J'}$ is no more than $|J\setminus J'|$ (\cite[\Luoma{2}.3.24]{abbes2016p}), we only need to verify the statement for $0\leq q\leq |J\setminus J'|$. 

	Suppose that the statement holds for $J''\subseteq J'\subseteq I$. We need to verify the statement for $J''\subseteq J=J'\cup\{j\}$ for any $j\in I\setminus J'$. Note that $D_{(J,J'')}=D_{(J,J')}\oplus D_{(J',J'')}$ \eqref{para:notation-J-9}. We also have $E_{(J,J'')}=E_{(J,J')}\oplus E_{(J',J'')}$. Thus, it suffices to show that $H^q(G_{(J,J'')},E_{(J,J')})$ and $H^q(G_{(J,J'')},E_{(J',J'')})$ are killed by $p^2\pi^{2^{|J|+2}}$.

	Consider the convergent spectral sequence (\cite[\Luoma{5}.11.7]{abbes2016p})
	\begin{align}
		E_2^{a,b}=H^a(G_{(J',J'')},H^b(G_{(J,J')},E_{(J,J')}))\Rightarrow H^{a+b}(G_{(J,J'')},E_{(J,J')}).
	\end{align}
	Note that $H^b(G_{(J,J')},E_{(J,J')})$ is zero for any $b\notin\{0,1\}$. We claim that $H^b(G_{(J,J')},E_{(J,J')})$ is killed by $p\pi^{2^{|J|+1}}$ for any $b\in\{0,1\}$. Indeed, notice that $\{\df t_k\}_{k\notin J'}$ forms a $\pi^{2^{|J'|}}\ak{m}_K$-basis of $\widehat{\Omega}^1_{\ca{O}_{F_{\underline{J'}}}/\ca{O}_K}$ by \ref{prop:perfd-tower-basis}. In particular, we have $\df t_j\notin \pi^{2^{|J'|}}\ak{m}_K\cdot \Omega^1_{\ca{O}_{F_{\underline{J'}}}/\ca{O}_K}$ by \ref{lem:pi-basis}. Thus, we can apply \ref{cor:galois-coh} to the extensions $K\to F_{\underline{J'}}\to F_{\underline{J}}=\bigcup_{n\in\bb{N}}F_{\underline{J'}}(t_{j,p^n})$ so that $H^b(G_{(J,J')},D/p^rD)$ is killed by $p\pi^{2^{|J'|}}$ where (cf. \eqref{eq:lem:D_J-3})
	\begin{align}
		D=\ca{O}_{F_{\underline{J}}}\cap \ke(\ca{T}_j)=\ca{O}_{F_{\underline{J}}}\cap D_{(J,J')}[1/p].
	\end{align}
	Since the quotient of the inclusion $D_{(J,J')}\subseteq D$ is killed by $\pi^{2^{|J|}}$ by \ref{lem:perfd-tower-norm}, the kernel and cokernel of $H^b(G_{(J,J')},E_{(J,J')})\to H^b(G_{(J,J')},D/p^rD)$ are killed by $\pi^{2^{|J|}}$ by \ref{lem:pi-isom}. Therefore, $H^b(G_{(J,J')},E_{(J,J')})$ is killed by $p\pi^{2^{|J|+1}}$, which verifies the claim.
	
	Then, the spectral sequence degenerates on page $3$ and induces a long exact sequence
	\begin{align}
		\cdots\stackrel{\df}{\longrightarrow}& H^q(G_{(J',J'')},H^0(G_{(J,J')},E_{(J,J')}))\stackrel{\mrm{inf}}{\longrightarrow} H^q(G_{(J,J'')},E_{(J,J')})\stackrel{\mrm{res}}{\longrightarrow} H^{q-1}(G_{(J',J'')},H^1(G_{(J,J')},E_{(J,J')}))\\
		\stackrel{\df}{\longrightarrow}& H^{q+1}(G_{(J',J'')},H^0(G_{(J,J')},E_{(J,J')}))\stackrel{\mrm{inf}}{\longrightarrow}\cdots\nonumber
	\end{align}
	from which we see that $H^q(G_{(J,J'')},E_{(J,J')})$ is killed by $p^2\pi^{2^{|J|+2}}$. 

	Consider the convergent spectral sequence (\cite[\Luoma{5}.11.7]{abbes2016p})
	\begin{align}
		E_2^{a,b}=H^a(G_{(J',J'')},H^b(G_{(J,J')},E_{(J',J'')}))\Rightarrow H^{a+b}(G_{(J,J'')},E_{(J',J'')}).
	\end{align}
	As $G_{(J,J')}$ acts trivially on $E_{(J',J'')}$, the spectral sequence degenerates on page $3$ and induces a long exact sequence
	\begin{align}
		\cdots\stackrel{\df}{\longrightarrow}& H^q(G_{(J',J'')},E_{(J',J'')})\stackrel{\mrm{inf}}{\longrightarrow} H^q(G_{(J,J'')},E_{(J',J'')})\stackrel{\mrm{res}}{\longrightarrow} H^{q-1}(G_{(J',J'')},E_{(J',J'')})\\
		\stackrel{\df}{\longrightarrow}& H^{q+1}(G_{(J',J'')},E_{(J',J'')})\stackrel{\mrm{inf}}{\longrightarrow}\cdots.\nonumber
	\end{align}
	Notice that the inflation map $H^q(G_{(J',J'')},E_{(J',J'')})\stackrel{\mrm{inf}}{\longrightarrow} H^q(G_{(J,J'')},E_{(J',J'')})$ admits a left inverse given by the composition of
	\begin{align}
		H^q(G_{(J,J'')},E_{(J',J'')})\stackrel{\mrm{res}}{\longrightarrow} H^q(G_{(J,J''\cup\{j\})},E_{(J',J'')})\stackrel{\sim}{\longleftarrow}H^q(G_{(J',J'')},E_{(J',J'')})
	\end{align}
	where the second isomorphism is induced by the canonical isomorphism $G_{(J,J''\cup\{j\})}\iso G_{(J',J'')}$. Therefore, there is a decomposition
	\begin{align}
		H^q(G_{(J,J'')},E_{(J',J'')})=H^q(G_{(J',J'')},E_{(J',J'')})\oplus H^{q-1}(G_{(J',J'')},E_{(J',J'')}),
	\end{align}
	which is killed by $p^2\pi^{2^{|J'|+2}}$ by induction hypothesis. This completes the proof.
\end{proof}

\begin{myprop}\label{prop:perfd-tower-coh}
	For any $q,r\in\bb{N}$, the canonical homomorphism
	\begin{align}\label{eq:prop:perfd-tower-coh-1}
		\wedge^q_{\ca{O}_F}\ho_{\bb{Z}_p}(\gal(F_{\underline{\infty}}/F),\ca{O}_F/p^r\ca{O}_F)\longrightarrow H^q(G_F,\ca{O}_{\overline{F}}/p^r\ca{O}_{\overline{F}})
	\end{align}
	has kernel killed by $\pi^{2^d}$ and cokernel killed by $p^2\pi^{2^{d+3}}$.
\end{myprop}
\begin{proof}
	On the one hand, as $G_F$ is isomorphic to $\bb{Z}_p^d$ (\ref{cor:perfd-tower-gal}), the canonical morphism induced by the cup product
	\begin{align}\label{eq:prop:perfd-tower-coh-2}
		\wedge^q_{\ca{O}_F}\ho_{\bb{Z}_p}(\gal(F_{\underline{\infty}}/F),\ca{O}_F/p^r\ca{O}_F)\longrightarrow H^q(\gal(F_{\underline{\infty}}/F),\ca{O}_F/p^r\ca{O}_F)
	\end{align}
	is an isomorphism (\cite[\Luoma{2}.3.30, \Luoma{2}.6.12]{abbes2016p}). On the other hand, as $F_{\underline{\infty}}$ is pre-perfectoid by \ref{cor:perfd-tower-perfd}, the canonical morphism
	\begin{align}\label{eq:prop:perfd-tower-coh-3}
		H^q(\gal(F_{\underline{\infty}}/F),\ca{O}_{F_{\underline{\infty}}}/p^r\ca{O}_{F_{\underline{\infty}}})\longrightarrow H^q(G_F,\ca{O}_{\overline{F}}/p^r\ca{O}_{\overline{F}})
	\end{align}
	is an almost isomorphism by almost purity and almost Galois descent (cf. \cite[\Luoma{2}.6.24]{abbes2016p}). Notice that the canonical morphism $\ca{O}_F\oplus D_{(I,\emptyset)}\to \ca{O}_{F_{\underline{\infty}}}$ \eqref{eq:lem:D_J-3} is injective with cokernel killed by $\pi^{2^d-1}$ by \ref{lem:perfd-tower-norm}. Thus, the canonical morphism
	\begin{align}\label{eq:prop:perfd-tower-coh-4}
		H^q(\gal(F_{\underline{\infty}}/F),\ca{O}_F/p^r)\oplus H^q(\gal(F_{\underline{\infty}}/F),D_{(I,\emptyset)}/p^r)\longrightarrow H^q(\gal(F_{\underline{\infty}}/F),\ca{O}_{F_{\underline{\infty}}}/p^r)
	\end{align}
	has kernel and cokernel both killed by $\pi^{2^d-1}$ by \ref{lem:pi-isom}.
	Thus, the conclusion follows from the fact that $H^q(\gal(F_{\underline{\infty}}/F),D_{(I,\emptyset)}/p^rD_{(I,\emptyset)})$ is killed by $p^2\pi^{2^{d+2}}$ by \ref{lem:perfd-tower-coh}.
\end{proof}

\begin{mylem}[{\cite[2.6.3]{abbes2020suite}}]\label{lem:pi-isom}
	Let $A$ be a ring, $\ca{A}$ an abelian category endowed with a linear action of $A$ (i.e., a ring homomorphism $A\to \mrm{End}(\id_{\ca{A}})$), $f:M\to N$ a morphism in $\ca{A}$. Assume that the kernel of $f$ is annihilated by $\pi\in A$ and the cokernel of $f$ is annihilated by $\pi'\in A$. Then, there exists a canonical morphism $g:N\to M$ such that $f\circ g=\pi\pi'|_N$ and $g\circ f=\pi\pi'|_M$. 
	
	In particular, for any functor of abelian categories $F:\ca{A}\to \ca{B}$ compatible with the linear action of $A$, the kernel and cokernel of $F(f):F(M)\to F(N)$ are annihilated by $\pi\pi'$.
\end{mylem}
\begin{proof}
	Since $\cok(f)$ is killed by $\pi'$, $\pi'|_N:N\to N$ factors through $\varphi:N\to \im(f)$. Since $\ke(f)$ is killed by $\pi$, $\pi|_M:M\to M$ factors through $\psi:\im(f)\to M$. Then, the conclusion follows by taking $g=\psi\circ\varphi:N\to M$.
\end{proof}

\section{Canonical Comparison between Differentials and Galois Cohomology}\label{sec:comp}
We construct a canonical comparison morphism from the module of differentials to the Galois cohomology (see \ref{prop:comp}). This allows us to vary the bound $\pi$ in the Section \ref{sec:perfd} and take filtered colimit over finitely generated subfields (see \ref{thm:coh}). In the end, we deduce from this computation of Galois cohomology of valuation fields an analogue of Hodge-Tate spectral sequence for (non-smooth) proper varieties (see \ref{rem:htss}).

\begin{mypara}\label{para:notation-comp}
	In this section, we fix a Henselian valuation field $F$ of height $1$ extension of $\bb{Q}_p$, and an algebraic closure $\overline{F}$ of $F$. We put $G_F=\gal(\overline{F}/F)$ the absolute Galois group of $F$.
	
	For any $r\in\bb{N}$, there is a canonical morphism in the derived category of $\ca{O}_F$-modules
	\begin{align}\label{eq:para:notation-comp-1}
		\rr\ho_{\bb{Z}}(p^{-r}\bb{Z}/\bb{Z},\Omega^1_{\ca{O}_F/\bb{Z}_p})\longrightarrow \rr\Gamma(G_F,(\ca{O}_{\overline{F}}/p^r\ca{O}_{\overline{F}})\{1\}),
	\end{align}
	where $\{1\}$ is the first differential Tate twist over $\ca{O}_{\widehat{\overline{F}}}$ (\ref{defn:int-tate-twist}), which fits into the following commutative diagram
	\begin{align}\label{eq:para:notation-comp-2}
		\xymatrix{
			\rr\ho_{\bb{Z}}(p^{-r}\bb{Z}/\bb{Z},\Omega^1_{\ca{O}_F/\bb{Z}_p})\ar[r]\ar[d]& \rr\Gamma(G_F,(\ca{O}_{\overline{F}}/p^r\ca{O}_{\overline{F}})\{1\})\ar[d]^-{\wr}\\
			\rr\Gamma(G_F,\rr\ho_{\bb{Z}}(p^{-r}\bb{Z}/\bb{Z},\Omega^1_{\ca{O}_{\overline{F}}/\bb{Z}_p}))&\rr\Gamma(G_F,\Omega^1_{\ca{O}_{\overline{F}}/\bb{Z}_p}[p^r])\ar[l]_-{\sim}
		}
	\end{align}
	where the left vertical arrow is induced by the canonical $G_F$-equivariant morphism $\rr\ho_{\bb{Z}}(p^{-r}\bb{Z}/\bb{Z},\Omega^1_{\ca{O}_F/\bb{Z}_p})\to \rr\ho_{\bb{Z}}(p^{-r}\bb{Z}/\bb{Z},\Omega^1_{\ca{O}_{\overline{F}}/\bb{Z}_p})$, the bottom arrow is induced by the canonical isomorphism $\Omega^1_{\ca{O}_{\overline{F}}/\bb{Z}_p}[p^r]\iso \rr\ho_{\bb{Z}}(p^{-r}\bb{Z}/\bb{Z},\Omega^1_{\ca{O}_{\overline{F}}/\bb{Z}_p})$ (as $\overline{F}$ is pre-perfectoid, see the proof of \ref{prop:rhom}), and the right vertical arrow is induced by the canonical isomorphism 
	$(\ca{O}_{\overline{F}}/p^r\ca{O}_{\overline{F}})\{1\}=T_p(\Omega^1_{\ca{O}_{\overline{F}}/\bb{Z}_p})/p^rT_p(\Omega^1_{\ca{O}_{\overline{F}}/\bb{Z}_p})\iso\Omega^1_{\ca{O}_{\overline{F}}/\bb{Z}_p}[p^r]$ \eqref{eq:prop:tate-twist-0}.
\end{mypara}

\begin{mylem}[{cf. \cite[8.15]{bhattmorrowscholze2018integral}}]\label{lem:comp}
	Let $(\zeta_{p^n})_{n\in\bb{N}}$ be a compatible system of primitive $p$-power roots of unity contained in $\overline{F}$, $(t_{p^n})_{n\in\bb{N}}$ a compatible system of $p$-power roots of an element $t\in\ca{O}_F^\times$ contained in $\overline{F}$. Then, for any $r\in\bb{N}$ the canonical morphism induced by taking the first cohomology group of \eqref{eq:para:notation-comp-1},
	\begin{align}\label{eq:lem:comp-1}
		\Omega^1_{\ca{O}_F/\bb{Z}_p}/p^r\Omega^1_{\ca{O}_F/\bb{Z}_p}\longrightarrow H^1(G_F,(\ca{O}_{\overline{F}}/p^r\ca{O}_{\overline{F}})\{1\}),
	\end{align}
	sends $\df\log(t)$ to the class of the continuous $1$-cocycle  $-\xi_t\otimes  (\df\log(\zeta_{p^n}))_{n\in\bb{N}}:G_F\to (\ca{O}_{\overline{F}}/p^r\ca{O}_{\overline{F}})\{1\}$, where $\xi_t:G_F\to \bb{Z}_p$ is the continuous map characterized by $\tau(t_{p^n})=\zeta_{p^n}^{\xi_t(\tau)}t_{p^n}$ for any $\tau\in G_F$ and $n\in\bb{N}$.
\end{mylem}
\begin{proof}
	We follow closely the proof of \cite[8.15]{bhattmorrowscholze2018integral}. Recall that $\rr\ho_{\bb{Z}}(p^{-r}\bb{Z}/\bb{Z},\Omega^1_{\ca{O}_{\overline{F}}/\bb{Z}_p})$ is represented by $\Omega^1_{\ca{O}_{\overline{F}}/\bb{Z}_p}\stackrel{\cdot p^r}{\longrightarrow}\Omega^1_{\ca{O}_{\overline{F}}/\bb{Z}_p}$. Let $C^\bullet_{\mrm{cont}}(G_F,\Omega^1_{\ca{O}_{\overline{F}}/\bb{Z}_p})$ be the complex of non-homogeneous continuous cochains of $G_F$ with value in the (discrete) $\ca{O}_{\overline{F}}$-module $\Omega^1_{\ca{O}_{\overline{F}}/\bb{Z}_p}$ (\cite[\Luoma{2}.3.8]{abbes2016p}). Then, $\rr\Gamma(G_F,\rr\ho_{\bb{Z}}(p^{-r}\bb{Z}/\bb{Z},\Omega^1_{\ca{O}_{\overline{F}}/\bb{Z}_p}))$ is represented by the total complex of the double complex
	\begin{align}\label{eq:lem:comp-2}
		\xymatrix{
			C^0_{\mrm{cont}}(G_F,\Omega^1_{\ca{O}_{\overline{F}}/\bb{Z}_p})\ar[r]^-{\df}&C^1_{\mrm{cont}}(G_F,\Omega^1_{\ca{O}_{\overline{F}}/\bb{Z}_p})\ar[r]^-{\df}&\cdots\\
			C^0_{\mrm{cont}}(G_F,\Omega^1_{\ca{O}_{\overline{F}}/\bb{Z}_p})\ar[r]^-{\df}\ar[u]^-{\cdot p^r}&C^1_{\mrm{cont}}(G_F,\Omega^1_{\ca{O}_{\overline{F}}/\bb{Z}_p})\ar[r]^-{\df}\ar[u]^-{\cdot p^r}&\cdots.
		}
	\end{align}
	For any $t\in\ca{O}_F^\times$, we see that the image $\omega$ of $\df\log(t)\in \Omega^1_{\ca{O}_F/\bb{Z}_p}/p^r\Omega^1_{\ca{O}_F/\bb{Z}_p}=\mrm{Ext}^1_{\bb{Z}}(p^{-r}\bb{Z}/\bb{Z},\Omega^1_{\ca{O}_F/\bb{Z}_p})$ under the map induced by the left vertical arrow of \eqref{eq:para:notation-comp-2} is represented by the element $\df\log(t)$ of the upper-left corner $C^0_{\mrm{cont}}(G_F,\Omega^1_{\ca{O}_{\overline{F}}/\bb{Z}_p})$ of the diagram \eqref{eq:lem:comp-2}. Since $\df\log(t)=p^r\df\log(t_{p^r})$, we see that $\omega$ is also represented by the element $-\df(\df\log(t_{p^r})):G_F\to \Omega^1_{\ca{O}_{\overline{F}}/\bb{Z}_p}$ sending $\tau\in G_F$ to $-(\tau-1)(\df\log(t_{p^r}))=-\xi_t(\tau)\df\log(\zeta_{p^r})$. Then, the conclusion follows from unwinding the definition of \eqref{eq:para:notation-comp-2}.
\end{proof}

\begin{myprop}\label{prop:comp}
	There is a canonical $\ca{O}_{\widehat{F}}$-linear homomorphism
	\begin{align}\label{eq:cor:comp}
		\widehat{\Omega}^1_{\ca{O}_F/\bb{Z}_p}\longrightarrow H^1(G_F,\ca{O}_{\widehat{\overline{F}}}\{1\})
	\end{align}
	sending $\df\log(t)$ to $-\xi_t\otimes  (\df\log(\zeta_{p^n}))_{n\in\bb{N}}:G_F\to \ca{O}_{\widehat{\overline{F}}}\{1\}$ for any $t\in\ca{O}_F^\times$ {\rm(\ref{lem:comp})}.
\end{myprop}
\begin{proof}
	The canonical morphism \eqref{eq:para:notation-comp-1} induces a canonical morphism 
	\begin{align}
		\rr\lim_{r\in\bb{N}}\rr\ho_{\bb{Z}}(p^{-r}\bb{Z}/\bb{Z},\Omega^1_{\ca{O}_F/\bb{Z}_p})\longrightarrow \rr\lim_{r\in\bb{N}}\rr\Gamma(G_F,(\ca{O}_{\overline{F}}/p^r\ca{O}_{\overline{F}})\{1\}).
	\end{align}
	Moreover, it induces a canonical morphism of exact sequences (\cite[\href{https://stacks.math.columbia.edu/tag/07KY}{07KY}]{stacks-project})
	\begin{align}
		\xymatrix{
			0\to\rr^1\lim \ho(p^{-r}\bb{Z}/\bb{Z},\Omega^1)\ar[r]\ar[d]& H^1(\rr\lim\rr\ho(p^{-r}\bb{Z}/\bb{Z},\Omega^1))\ar[d]^-{\beta}\ar[r]^-{\alpha}&\lim\mrm{Ext}^1(p^{-r}\bb{Z}/\bb{Z},\Omega^1)\ar[d]\to 0\\ 
			0\to\rr^1\lim H^0(G_F,(\ca{O}_{\overline{F}}/p^r)\{1\})\ar[r]& H^1(\rr\lim\rr\Gamma(G_F,(\ca{O}_{\overline{F}}/p^r)\{1\}))\ar[r]&\lim H^1(G_F,(\ca{O}_{\overline{F}}/p^r)\{1\})\to 0
		}
	\end{align}
	where $\Omega^1=\Omega^1_{\ca{O}_F/\bb{Z}_p}$. If $\Omega^1[p^\infty]$ is bounded (i.e., $F$ is arithmetic), then we see that the inverse system $(\ho(p^{-r}\bb{Z}/\bb{Z},\Omega^1))_{r\in\bb{N}}=(\Omega^1[p^r])_{r\in\bb{N}}$ is essentially zero. If $\Omega^1[p^\infty]$ is unbounded (i.e., $F$ is non-arithmetic), we have $(\ho(p^{-r}\bb{Z}/\bb{Z},\Omega^1))_{r\in\bb{N}}\cong (\ca{O}_F/p^r)_{r\in\bb{N}}$ by \ref{prop:tate-twist}. In either case, $(\ho(p^{-r}\bb{Z}/\bb{Z},\Omega^1))_{r\in\bb{N}}$ satisfies the Mittag-Leffler condition so that $\rr^1\lim \ho(p^{-r}\bb{Z}/\bb{Z},\Omega^1)=0$ (\cite[\href{https://stacks.math.columbia.edu/tag/07KW}{07KW}]{stacks-project}). On the other hand, as $((\ca{O}_{\overline{F}}/p^r)\{1\})_{r\in\bb{N}}$ satisfies the Mittag-Leffler condition with (topological) limit $\ca{O}_{\widehat{\overline{F}}}\{1\}$, $\rr\lim\rr\Gamma(G_F,(\ca{O}_{\overline{F}}/p^r)\{1\})$ is represented by $C^{\bullet}_{\mrm{cont}}(G_F,\ca{O}_{\widehat{\overline{F}}}\{1\})$ (see the proof of \cite[2.2]{jannsen1988cont}). Thus, we obtain a canonical morphism
	\begin{align}
		\beta\circ\alpha^{-1}:\widehat{\Omega}^1_{\ca{O}_F/\bb{Z}_p}=\lim\mrm{Ext}^1(p^{-r}\bb{Z}/\bb{Z},\Omega^1)\longrightarrow H^1(G_F,\ca{O}_{\widehat{\overline{F}}}\{1\}).
	\end{align}
	It follows from (the proof of) \ref{lem:comp} that this morphism sends $\df\log(t)$ to $-\xi_t\otimes  (\df\log(\zeta_{p^n}))_{n\in\bb{N}}$ for any $t\in\ca{O}_F^\times$ (cf. \cite[\Luoma{2}.3.9, \Luoma{2}.3.10]{abbes2016p}).
\end{proof}

\begin{myrem}\label{rem:prop:comp}
	For any $q\in\bb{N}$, there are canonical morphisms of $\ca{O}_F$-modules,
	\begin{align}\label{eq:rem:prop:comp-1}
		\Omega^q_{\ca{O}_F/\bb{Z}_p}=\wedge^q\Omega^1_{\ca{O}_F/\bb{Z}_p}\longrightarrow\wedge^q\widehat{\Omega}^1_{\ca{O}_F/\bb{Z}_p}\longrightarrow H^q(G_F,\ca{O}_{\widehat{\overline{F}}}\{q\}),
	\end{align}
	where the second is induced by the cup product of \eqref{eq:cor:comp}. Moreover, if $H^q(G_F,\ca{O}_{\widehat{\overline{F}}}\{q\})$ is $p$-adically complete, then by taking $p$-adic completion, the composition of \eqref{eq:rem:prop:comp-1} induces a canonical morphism of complete $\ca{O}_{\widehat{F}}$-modules
	\begin{align}\label{eq:rem:prop:comp-2}
		\widehat{\Omega}^q_{\ca{O}_F/\bb{Z}_p}\longrightarrow H^q(G_F,\ca{O}_{\widehat{\overline{F}}}\{q\}).
	\end{align}
\end{myrem}

\begin{mypara}\label{para:notation-fal-ext}
	Recall that the Faltings extension of $\ca{O}_{\overline{F}}$ over $\bb{Z}_p$ (\cite[3.11]{he2024perfd}) is a canonical exact sequence of $\widehat{\overline{F}}$-modules endowed with $G_F$-action,
	\begin{align}\label{eq:para:notation-fal-ext-1}
		0\longrightarrow \widehat{\overline{F}}(1)\stackrel{\iota}{\longrightarrow} \scr{E}_{\ca{O}_{\overline{F}}}\stackrel{\jmath}{\longrightarrow} \widehat{\overline{F}}\otimes_F\Omega^1_{F/\bb{Q}_p}\longrightarrow 0,
	\end{align}
	such that there is a canonical $G_F$-equivariant group homomorphism
	\begin{align}\label{eq:para:notation-fal-ext-2}
		V_p(\overline{F}^\times)=\lim_{x\mapsto x^p}\overline{F}^\times\longrightarrow \scr{E}_{\ca{O}_{\overline{F}}},\ (t_{p^n})_{n\in\bb{N}}\mapsto (\df\log(t_{p^n}))_{n\in\bb{N}},
	\end{align} 
	and that $\iota((\zeta_{p^n})_{n\in\bb{N}})=(\df\log(\zeta_{p^n}))_{n\in\bb{N}}$ and $\jmath((\df\log(t_{p^n}))_{n\in\bb{N}})=\df\log(t)$, where $t=t_1\in \overline{F}^\times$. We remind the readers that we don't put any topology on $\scr{E}_{\ca{O}_{\overline{F}}}$ unless the transcendental degree of $F$ over $\bb{Q}_p$ is finite (where we put the canonical topology as a finite free $\widehat{\overline{F}}$-module). Taking the (discrete) group cohomology of \eqref{eq:para:notation-fal-ext-1}, we obtain a coboundary map
	\begin{align}\label{eq:para:notation-fal-ext-3}
		\delta:\widehat{F}\otimes_F\Omega^1_{F/\bb{Q}_p}\longrightarrow H^1_{\mrm{disc}}(G_F,\widehat{\overline{F}}(1)),
	\end{align}
	where we used the fact that $(\widehat{\overline{F}})^{G_F}=\widehat{F}$ by Ax-Sen-Tate's theorem \cite[page 417]{ax1970ax}. As the $G_F$-action on $\widehat{\overline{F}}(1)$ is continuous, the continuous group cohomology $H^1(G_F,\widehat{\overline{F}}(1))$ coincides with the submodule of $H^1_{\mrm{disc}}(G_F,\widehat{\overline{F}}(1))$ consisting of classes represented by continuous $1$-cocycles.
\end{mypara}

\begin{mylem}\label{lem:fal-ext-delta}
	The coboundary map $\delta$ \eqref{eq:para:notation-fal-ext-3} is the composition of the canonical morphisms
	\begin{align}
		\xymatrix{
			\widehat{F}\otimes_F\Omega^1_{F/\bb{Q}_p}\ar[r]^-{\cdot(-1)}&	\widehat{\Omega}^1_{\ca{O}_F/\bb{Z}_p}[1/p]\ar[r]^-{\eqref{eq:cor:comp}}& H^1(G_F,\widehat{\overline{F}}(1))\ar@{^{(}->}[r]&H^1_{\mrm{disc}}(G_F,\widehat{\overline{F}}(1)).
		}
	\end{align}
\end{mylem}
\begin{proof}
	Firstly, we show that $\delta$ factors through $H^1(G_F,\widehat{\overline{F}}(1))$. As $\{\df\log(t)\}_{t\in\ca{O}_F^\times}$ generates the $F$-module $\Omega^1_{F/\bb{Q}_p}$, we only need to prove that $\delta(\df\log(t))$ is represented by a continuous $1$-cocycle. Let $(t_{p^n})_{n\in\bb{N}}$ be a compatible system of $p$-power roots of $t$ in $\overline{F}$ and let $\xi_t:G_F\to \bb{Z}_p$ be the associated continuous map (\ref{lem:comp}). By \eqref{eq:para:notation-fal-ext-1}, we see that $\delta(\df\log(t))$ is represented by the $1$-cocycle $G_F\to \widehat{\overline{F}}(1)$ sending $\tau\in G_F$ to $(\tau-1)(\df\log(t_{p^n}))_{n\in\bb{N}}=\xi_t(\tau)(\zeta_{p^n})_{n\in\bb{N}}$, which is clearly continuous. Thus, $\delta$ factors through $H^1(G_F,\widehat{\overline{F}}(1))$. Moreover, we see that $\delta$ is compatible with \eqref{eq:cor:comp} by \ref{prop:comp} and the isomorphism $(\zeta_p-1)^{-1}\ca{O}_{\widehat{\overline{F}}}(1)\iso \ca{O}_{\widehat{\overline{F}}}\{1\}$ \eqref{eq:cor:tate-cyclo-2}.
\end{proof}

\begin{mythm}\label{thm:coh}
	Assume that $F$ contains a compatible system of primitive $p$-power roots of unity $(\zeta_{p^n})_{n\in\bb{N}}$. Then, for any $q,r\in\bb{N}$, the canonical morphism induced by the cup product of \eqref{eq:lem:comp-1},
	\begin{align}\label{eq:thm:coh-1}
		\Omega^q_{\ca{O}_F/\bb{Z}_p}/p^r\Omega^q_{\ca{O}_F/\bb{Z}_p}\longrightarrow H^q(G_F,\ca{O}_{\overline{F}}/p^r\ca{O}_{\overline{F}}\{q\}),
	\end{align}
	has kernel killed by $(\zeta_p-1)^q\ak{m}_F$ and cokernel killed by $(\zeta_p-1)^qp^2\ak{m}_F$.
\end{mythm}
\begin{proof}
	We put $K=\bigcup_{n\in\bb{N}}\bb{Q}_p(\zeta_{p^n})\subseteq F$ and $\ak{a}=(\zeta_p-1)^{-1}\ca{O}_{\widehat{K}}$. We identify $\ca{O}_{\widehat{K}}\{1\}$ with $\ak{a}(1)$ by \eqref{eq:cor:tate-cyclo-2}. Firstly, assume that the transcendental degree of $F$ over $K$ is finite. Let $d$ be the rank of $\widehat{\Omega}^1_{\ca{O}_F/\ca{O}_K}$ and let $\pi\in\ak{m}_K$ such that $\pi^{2^d}\notin p\ca{O}_K$. We take $t_1,\dots,t_d\in\ca{O}_F^\times$ such that $\df t_1,\dots,\df t_d$ form a $\pi\ak{m}_K$-basis of $\widehat{\Omega}^1_{\ca{O}_F/\ca{O}_K}$ by \ref{lem:differential}. We take again the notation in \ref{para:notation-perfd-tower} and \ref{para:notation-J}. Consider the homomorphism of $\ca{O}_F$-modules
	\begin{align}\label{eq:thm:coh-2}
		\varphi:\ho_{\bb{Z}_p}(\gal(F_{\underline{\infty}}/F),\ca{O}_F/p^r\ca{O}_F)\longrightarrow (\Omega^1_{\ca{O}_F/\bb{Z}_p}/p^r\Omega^1_{\ca{O}_F/\bb{Z}_p})\{-1\}
	\end{align}
	sending an element $f:\gal(F_{\underline{\infty}}/F)\to \ca{O}_F/p^r\ca{O}_F$ to $\varphi(f)=\sum_{i=1}^d f(\tau_i)\df\log(t_i)\otimes (\zeta_p-1)(\zeta_{p^n})_{n\in\bb{N}}^{-1}$, where $(\zeta_p-1)(\zeta_{p^n})_{n\in\bb{N}}^{-1}$ is the basis of $\ca{O}_{\widehat{K}}\{-1\}=\ak{a}^{-1}(-1)$. As $\gal(F_{\underline{\infty}}/F)$ is a free $\bb{Z}_p$-module with basis $\tau_1,\dots,\tau_d$, we see that $\varphi$ is injective with cokernel killed by $\pi\ak{m}_K$ by \eqref{eq:notation-perfd-tower-1} and the equality $\Omega^1_{\ca{O}_F/\bb{Z}_p}/p^r=\Omega^1_{\ca{O}_F/\ca{O}_K}/p^r$ \eqref{eq:prop:rhom-2}. Then, it follows from \ref{lem:comp} that the following diagram is commutative:
	\begin{align}\label{eq:thm:coh-3}
		\xymatrix{
			(\Omega^q_{\ca{O}_F/\bb{Z}_p}/p^r\Omega^q_{\ca{O}_F/\bb{Z}_p})\{-q\}\ar[r]^-{\eqref{eq:thm:coh-1}}& H^q(G_F,\ca{O}_{\overline{F}}/p^r\ca{O}_{\overline{F}})\\
			\wedge^q_{\ca{O}_F}\ho_{\bb{Z}_p}(\gal(F_{\underline{\infty}}/F),\ca{O}_F/p^r\ca{O}_F)\ar[r]^-{\eqref{eq:prop:perfd-tower-coh-1}}\ar[u]^-{\wedge^q\varphi}&H^q(G_F,\ca{O}_{\overline{F}}/p^r\ca{O}_{\overline{F}})\ar[u]_-{h^q}
		}
	\end{align}
	where $h^q$ is induced by the multiplication by $(1-\zeta_p)^q$ on $\ca{O}_{\overline{F}}/p^r\ca{O}_{\overline{F}}$. Note that the kernel and cokernel of $h^q$ are killed by $(\zeta_p-1)^q$ and that the kernel and cokernel of $\wedge^q\varphi$ are killed by $\pi^q\ak{m}_K$ (\ref{lem:pi-isom}). Since \eqref{eq:prop:perfd-tower-coh-1} has kernel killed by $\pi^{2^d}$ and cokernel killed by $p^2\pi^{2^{d+3}}$, by chasing the diagram we see that the kernel of \eqref{eq:thm:coh-1} is killed by $(\zeta_p-1)^q\pi^{q+2^d}\ak{m}_K$ and its cokernel is killed by $(\zeta_p-1)^qp^2\pi^{2^{d+3}}$. Varying $\pi$ in $\ak{m}_K\setminus p\ca{O}_K$, we conclude that the kernel of \eqref{eq:thm:coh-1} is killed by $(\zeta_p-1)^q\ak{m}_K$ and its cokernel is killed by $(\zeta_p-1)^qp^2\ak{m}_K$.
	
	In general, we write $F$ as a filtered union $F=\bigcup_{\lambda\in\Lambda}F_\lambda$ of its subfields finitely generated over $K$.  After replacing $F_\lambda$ by the fraction field of the Henselization of the valuation ring $\ca{O}_{F_\lambda}=F_\lambda\cap \ca{O}_F$, we may assume that each $F_\lambda$ is a Henselian valuation field of finite transcendental degree over $K$ (\cite[6.1.12.(\luoma{6})]{gabber2003almost}). Let $\overline{F}_\lambda$ be the algebraic closure of $F_\lambda$ in $\overline{F}$ and we put $G_{F_\lambda}=\gal(\overline{F}_\lambda/F_\lambda)$. Then, the canonical morphism \eqref{eq:thm:coh-1} identifies with (\cite[\Luoma{1}.\textsection2, Proposition 8]{serre2002galois})
	\begin{align}\label{eq:thm:coh-4}
		\colim_{\lambda\in\Lambda}\Omega^q_{\ca{O}_{F_\lambda}/\bb{Z}_p}/p^r\Omega^q_{\ca{O}_{F_\lambda}/\bb{Z}_p}\longrightarrow \colim_{\lambda\in\Lambda}H^q(G_{F_\lambda},(\ca{O}_{\overline{F}_\lambda}/p^r\ca{O}_{\overline{F}_\lambda})\{q\}).
	\end{align}
	Thus, the conclusion follows from applying the previous special case to each $F_\lambda$.
\end{proof}

\begin{mylem}\label{lem:complete}
	Let $A$ be a ring, $\pi$ an element of $A$, $0\to N\to M\to Q\to 0$ an exact sequence of $A$-modules such that $Q[\pi^\infty]=Q[\pi^c]$ for some $c\in\bb{N}$. Then, the sequence of the $\pi$-adic completions 
	\begin{align}
		0\longrightarrow \widehat{N}\longrightarrow \widehat{M}\longrightarrow \widehat{Q}\longrightarrow 0
	\end{align}
	is exact. In particular, $M$ is $\pi$-adically complete if $N$ and $Q$ are so.
\end{mylem}
\begin{proof}
	We claim that $N\cap \pi^n M\subseteq \pi^{n-c}N$ for any $n\in\bb{N}_{\geq c}$. Indeed, for any $x\in N\cap \pi^n M$, we write $x=\pi^ny$ for some $y\in M$. Thus, the image of $y$ in $Q$ lies in $Q[\pi^n]=Q[\pi^c]$. This implies that $\pi^cy\in N$ and thus $x\in \pi^{n-c}N$. Combining with $\pi^nN\subseteq N\cap \pi^n M$, the claim implies that $\widehat{N}=\lim_{n\in\bb{N}}N/\pi^nN=\lim_{n\in\bb{N}}N/N\cap \pi^n M$. As the inverse system $(N/N\cap \pi^n M)_{n\in\bb{N}}$ satisfies the Mittag-Leffler condition, the conclusion follows from taking the inverse limit of the exact sequences $0\to N/N\cap \pi^n M\to M/\pi^n M\to Q/\pi^nQ\to 0$.
\end{proof}

\begin{mythm}\label{thm:coh-int}
	Assume that $F$ contains a compatible system of primitive $p$-power roots of unity $(\zeta_{p^n})_{n\in\bb{N}}$. Then, for any $q\in\bb{N}$, the $q$-th Galois cohomology group $H^q(G_F,\ca{O}_{\widehat{\overline{F}}}\{q\})$ is $p$-adically complete. Moreover, the canonical morphism \eqref{eq:rem:prop:comp-2},
	\begin{align}\label{eq:thm:coh-1-int-1}
		\widehat{\Omega}^q_{\ca{O}_F/\bb{Z}_p}\longrightarrow H^q(G_F,\ca{O}_{\widehat{\overline{F}}}\{q\}),
	\end{align}
	has kernel killed by $(\zeta_p-1)^q\ak{m}_F$ and cokernel killed by $(\zeta_p-1)^{4q}p^4\ak{m}_F$. In particular, it induces a canonical isomorphism by inverting $p$,
	\begin{align}\label{eq:thm:coh-1-int-2}
		\widehat{\Omega}^q_{\ca{O}_F/\bb{Z}_p}[1/p]\iso H^q(G_F,\widehat{\overline{F}}(q)).
	\end{align}
\end{mythm}
\begin{proof}
	There is a canonical exact sequence (\cite[\Luoma{2}.3.10]{abbes2016p})
	\begin{align}\label{eq:thm:coh-1-int-3}
		0\to  \rr^1\lim_{r\in\bb{N}}H^{q-1}(G_F,\ca{O}_{\overline{F}}/p^r\ca{O}_{\overline{F}})\to H^q(G_F,\ca{O}_{\widehat{\overline{F}}})\to \lim_{r\in\bb{N}} H^q(G_F,\ca{O}_{\overline{F}}/p^r\ca{O}_{\overline{F}})\to 0.
	\end{align}
	As the inverse system $((\Omega^q_{\ca{O}_F/\bb{Z}_p}/p^r\Omega^q_{\ca{O}_F/\bb{Z}_p})\{-q\})_{r\in\bb{N}}$ satisfies the Mittag-Leffler condition, we see that $\rr^1\lim(\Omega^q_{\ca{O}_F/\bb{Z}_p}/p^r\Omega^q_{\ca{O}_F/\bb{Z}_p})\{-q\}=0$. Thus, $\rr^1\lim H^q(G_F,\ca{O}_{\overline{F}}/p^r\ca{O}_{\overline{F}})$ is killed by $(\zeta_p-1)^{2q}p^2\ak{m}_F$ by \ref{thm:coh} and \ref{lem:pi-isom}. In particular, it is $p$-adically complete.
	
	On the other hand, we claim that $\lim_{r\in\bb{N}}H^q(G_F,\ca{O}_{\overline{F}}/p^r\ca{O}_{\overline{F}})$ is $p$-adically complete with bounded $p$-power torsion. Indeed, it is $p$-adically complete as each $H^q(G_F,\ca{O}_{\overline{F}}/p^r\ca{O}_{\overline{F}})$ is killed by $p^r$ (\cite[\href{https://stacks.math.columbia.edu/tag/0G1Q}{0G1Q}]{stacks-project}). Notice that the canonical morphism induced by \eqref{eq:thm:coh-1},
	\begin{align}\label{eq:thm:coh-1-int-4}
		\widehat{\Omega}^q_{\ca{O}_F/\bb{Z}_p}\{-q\}=\lim_{r\in\bb{N}}(\Omega^q_{\ca{O}_F/\bb{Z}_p}/p^r\Omega^q_{\ca{O}_F/\bb{Z}_p})\{-q\} \longrightarrow   \lim_{r\in\bb{N}} H^q(G_F,\ca{O}_{\overline{F}}/p^r\ca{O}_{\overline{F}})
	\end{align}
	has kernel killed by $(\zeta_p-1)^q\ak{m}_F$ by \ref{thm:coh} and cokernel killed by $(\zeta_p-1)^{2q}p^2\ak{m}_F$ by \ref{thm:coh} and \ref{lem:pi-isom}. Therefore, the latter has bounded $p$-power torsion as $\widehat{\Omega}^q_{\ca{O}_F/\bb{Z}_p}$ does so (\ref{prop:non-ari-flat}). 
	
	Therefore, $H^q(G_F,\ca{O}_{\widehat{\overline{F}}})$ is also $p$-adically complete by \ref{lem:complete}. Then, \eqref{eq:thm:coh-1-int-4} coincides with the composition of the canonical morphisms
	\begin{align}\label{eq:thm:coh-1-int-5}
		\widehat{\Omega}^q_{\ca{O}_F/\bb{Z}_p}\{-q\}\longrightarrow H^q(G_F,\ca{O}_{\widehat{\overline{F}}}) \longrightarrow   \lim_{r\in\bb{N}} H^q(G_F,\ca{O}_{\overline{F}}/p^r\ca{O}_{\overline{F}})
	\end{align}
	where the first arrow is \eqref{eq:rem:prop:comp-2} and the second arrow is defined in \eqref{eq:thm:coh-1-int-3}. The conclusion follows from diagram chasing.
\end{proof}

\begin{myrem}\label{rem:htss}
	Let $K$ be a complete algebraically closed valuation field of height $1$ extension of $\bb{Q}_p$, $Y$ a proper $K$-scheme, $X$ the cofiltered limit of all the flat proper $\ca{O}_K$-models of $Y$ as a locally ringed space (\cite[9.10]{he2024perfd}). Consider the associated \'etale ringed site and Faltings ringed site defined by taking cofiltered limits, $(X_\et,\ca{O}_X)$ and $(\fal^\et_{Y\to X},\falb)$ (\cite[9.6]{he2024perfd}). There is a canonical morphism of ringed sites $\sigma:(\fal^\et_{Y\to X},\falb)\to (X_\et,\ca{O}_X)$. For any $j\in\bb{N}$, let $\Omega^j_{X/\bb{Z}_p}$ be the associated module of $j$-th differentials defined by taking filtered colimits on $X_\et$ (\cite[10.19]{he2024purity}). Then, as in \ref{para:notation-comp}, one can define a canonical morphism for any $r\in\bb{N}$,
	\begin{align}\label{eq:rem:htss-1}
		(\Omega^j_{X/\bb{Z}_p}/p^r\Omega^j_{X/\bb{Z}_p})\{-j\}\longrightarrow \rr^j\sigma_*(\falb/p^r\falb). 
	\end{align}
	Since the $p$-adic completion of each geometric stalk of $\ca{O}_X$ is a valuation ring extension of $\ca{O}_K$ (cf. \cite[10.3]{he2024purity}), reducing to the height-1 case as in \cite[5.16]{he2024falmain}, our computation of Galois cohomology of valuation rings \ref{thm:coh} implies that \eqref{eq:rem:htss-1} has kernel killed by $(\zeta_p-1)^j\ak{m}_K$ and cokernel killed by $(\zeta_p-1)^jp^2\ak{m}_K$. Consider the following commutative diagram of ringed sites
	\begin{align}
		\xymatrix{
			(\fal^{\et,\bb{N}}_{Y\to X},(\falb/p^r\falb)_{r\in\bb{N}})\ar[r]^-{\rr\lim}\ar[d]_-{\rr\sigma_*}&(\fal^\et_{Y\to X},\falb)\ar[d]^-{\rr\sigma_*}\\
			(X_\et^{\bb{N}},(\ca{O}_X/p^r\ca{O}_X)_{r\in\bb{N}})\ar[r]^-{\rr\lim}& (X_\et,\ca{O}_X)
		}
	\end{align}
	where $(-)^{\bb{N}}$ is taking the associated filbred site over $\bb{N}$ (cf. \cite[\Luoma{3}.7]{abbes2016p}). Then, there is a canonical convergent spectral sequence 
	\begin{align}
		H^i(X_\et^{\bb{N}},(\rr^j\sigma_*\falb/p^r\falb)_{r\in\bb{N}})\Rightarrow H^{i+j}(\fal^{\et,\bb{N}}_{Y\to X},(\falb/p^r\falb)_{r\in\bb{N}}).
	\end{align}
	Notice that the right hand side is almost isomorphic to $H^{i+j}_\et(Y,\bb{Z}_p)\otimes_{\bb{Z}_p}\ca{O}_K$ by Faltings' main comparison theorem (cf. \cite[6.4.4]{abbes2020suite}, \cite[5.17]{he2024falmain} and \cite[9.7]{he2024perfd}). Therefore, after inverting $p$, we obtain a canonical convergent spectral sequence as in \cite[6.4.6]{abbes2020suite},
	\begin{align}
		H^i(X_\et,\rr\lim_{r\in\bb{N}}(\Omega^j_{X/\bb{Z}_p}/p^r\Omega^j_{X/\bb{Z}_p})\{-j\})[1/p]\Rightarrow H^{i+j}_\et(Y,\bb{Q}_p)\otimes_{\bb{Q}_p}K,
	\end{align}
	which is an analog of Hodge-Tate spectral sequence for (non-smooth) proper schemes. Similarly, one can also make the same construction for (non-smooth) proper rigid analytic varieties over $K$. This should coincide with Guo's construction of the Hodge-Tate spectral sequence for (non-smooth) proper rigid analytic varieties \cite[1.1.3]{guo2019hodgetate}.
\end{myrem}

\section{Final Computation of Galois Cohomology}\label{sec:final}
We compute the Galois cohomology for geometric and arithmetic valuation fields by reducing to the case over the cyclotomic field (see \ref{thm:geo} and \ref{thm:ari}). As an application, we give a generalization of the perfectoidness criterion via Faltings extension \cite[4.23]{he2024perfd} (see \ref{prop:fal-ext-delta}).

\begin{mypara}\label{para:final-comp}
	In this section, we fix a Henselian valuation field $F$ of height $1$ extension of $\bb{Q}_p$ and an algebraic closure $\overline{F}$ of $F$. Let $(\zeta_{p^n})_{n\in\bb{N}}$ be a compatible system of primitive $p$-power roots of unity contained in $\overline{F}$. For any $n\in\bb{N}$, we put
	\begin{align}\label{eq:para:final-comp-1}
		F_n=F(\zeta_{p^n}),\quad F_\infty&=\bigcup_{n\in\bb{N}}F_n,
	\end{align}
	which do not depend on the choice of $(\zeta_{p^n})_{n\in\bb{N}}$. Note that $F_0=F$. We put $G_{F_n}=\gal(\overline{F}/F_n)$ the absolute Galois group of $F_n$ for any $n\in\bb{N}\cup\{\infty\}$. Let $v_p:\overline{F}\to \bb{R}\cup\{\infty\}$ be a valuation map with $v_p(p)=1$ and $v_p(0)=\infty$, and let
	\begin{align}\label{eq:para:final-comp-2}
		|\cdot|=p^{-v_p(\cdot)}:\overline{F}\longrightarrow \bb{R}_{\geq 0}
	\end{align}
	be the associated ultrametric absolute value.
\end{mypara}

\subsection{The geometric case}

\begin{mylem}\label{lem:non-ari-coh-twist}
	Assume that $F$ is not arithmetic. Let $M$ be a flat $\ca{O}_F/p^r\ca{O}_F$-module for some $r\in\bb{N}$. Then, for any $n\in\bb{Z}$ and $q\in\bb{N}$, there is a canonical isomorphism
	\begin{align}
		H^q(\gal(F_\infty/F),\ca{O}_{F_\infty}/p^r\ca{O}_{F_\infty})\otimes_{\ca{O}_F}M\{n\}\iso H^q(\gal(F_\infty/F),(\ca{O}_{F_\infty}\otimes_{\ca{O}_F}M)\{n\}),
	\end{align}
	where $M$ is endowed with the trivial action of $\gal(F_\infty/F)$ and $(\ca{O}_{F_\infty}\otimes_{\ca{O}_F}M)\{n\}=(\ca{O}_{F_\infty}\otimes_{\ca{O}_F}M)\otimes_{\ca{O}_{\widehat{F_\infty}}}T_p(\Omega^1_{\ca{O}_{F_\infty}/\bb{Z}_p})^{\otimes n}$ \eqref{eq:defn:int-tate-twist-1} is endowed with the diagonal action of $\gal(F_\infty/F)$.
\end{mylem}
\begin{proof}
	Note that the canonical morphism
	\begin{align}
		\ca{O}_{\widehat{F_\infty}}\otimes_{\ca{O}_{\widehat{F}}}T_p(\Omega^1_{\ca{O}_F/\bb{Z}_p})\longrightarrow T_p(\Omega^1_{\ca{O}_{F_\infty}/\bb{Z}_p})
	\end{align}
	is $\gal(F_\infty/F)$-equivariant. Moreover, it is an isomorphism by \ref{prop:tate-twist}. Therefore, $(\ca{O}_{F_\infty}\otimes_{\ca{O}_F}M)\{n\}\cong\ca{O}_{F_\infty}\otimes_{\ca{O}_F}(M\otimes_{\ca{O}_{\widehat{F}}}T_p(\Omega^1_{\ca{O}_F/\bb{Z}_p})^{\otimes n})=\ca{O}_{F_\infty}\otimes_{\ca{O}_F}M\{n\}$. As $M\{n\}$ is flat over $\ca{O}_F/p^r\ca{O}_F$ endowed with the trivial action of $\gal(F_\infty/F)$, we have $H^q(\gal(F_\infty/F),\ca{O}_{F_\infty}\otimes_{\ca{O}_F}M\{n\})=H^q(\gal(F_\infty/F),\ca{O}_{F_\infty}/p^r\ca{O}_{F_\infty})\otimes_{\ca{O}_F}M\{n\}$ (\cite[\Luoma{2}.3.15]{abbes2016p}), which completes the proof.
\end{proof}

\begin{mylem}\label{lem:geo-diff-bound}
	Assume that $F$ is geometric. Then, for any $r\in\bb{N}$ and $q\in\bb{N}_{>0}$, the cokernel of the canonical injection $\ca{O}_F/p^r\ca{O}_F\to H^0(\gal(F_\infty/F),\ca{O}_{F_\infty}/p^r\ca{O}_{F_\infty})$ and the $q$-th Galois cohomology group $H^q(\gal(F_\infty/F),\ca{O}_{F_\infty}/p^r\ca{O}_{F_\infty})$ are both killed by $\ca{O}_F\cap \ak{m}_F \scr{D}_{F_\infty/F}$.
\end{mylem}
\begin{proof}
	By taking filtered colimit (\cite[\Luoma{1}.\textsection2, Proposition 8]{serre2002galois}), it suffices to check that $\cok(\ca{O}_F/p^r\ca{O}_F\to H^0(\gal(F_n/F),\ca{O}_{F_n}/p^r\ca{O}_{F_n}))$ and $H^q(\gal(F_n/F),\ca{O}_{F_n}/p^r\ca{O}_{F_n})$ are both killed by $\ca{O}_F\cap \ak{m}_F \scr{D}_{F_\infty/F}$ for any $n\in\bb{N}$. This is verified in \ref{cor:coh-torsion} as $\ak{m}_F \scr{D}_{F_\infty/F}\subseteq \ak{m}_F \scr{D}_{F_n/F}$ (note that $F$ is non-discrete by \ref{lem:ari-geo-basic}.(\ref{item:lem:ari-geo-basic-1}) and \ref{lem:ari-geo-ex}).
\end{proof}

\begin{mylem}\label{lem:geo-diff-isom}
	Assume that $F$ is geometric. Let $\pi\in\ca{O}_F$ with norm $|\pi|<|\scr{D}_{F_\infty/F}|$, and $j,r\in\bb{N}$.
	\begin{enumerate}
		\renewcommand{\labelenumi}{{\rm(\theenumi)}}
		\item The canonical morphism $\Omega^j_{\ca{O}_F/\bb{Z}_p}/p^r\Omega^j_{\ca{O}_F/\bb{Z}_p}\to H^0(\gal(F_\infty/F),\Omega^j_{\ca{O}_{F_\infty}/\bb{Z}_p}/p^r\Omega^j_{\ca{O}_{F_\infty}/\bb{Z}_p})$ has kernel killed by $\pi^j$ and cokernel killed by $\pi^{j+1}$.\label{item:lem:geo-diff-isom-1}
		\item For any $i\in\bb{N}_{>0}$, $H^i(\gal(F_\infty/F),\Omega^j_{\ca{O}_{F_\infty}/\bb{Z}_p}/p^r\Omega^j_{\ca{O}_{F_\infty}/\bb{Z}_p})$ is killed by $\pi^{j+1}$.\label{item:lem:geo-diff-isom-2}
	\end{enumerate}
\end{mylem}
\begin{proof}
	For simplicity, we put $\Sigma=\gal(F_\infty/F)$. Consider the canonical exact sequence (\ref{thm:differential}.(\ref{item:thm:differential-1}))
	\begin{align}
		0\longrightarrow \ca{O}_{F_\infty}\otimes_{\ca{O}_F}\Omega^1_{\ca{O}_F/\bb{Z}_p}\longrightarrow \Omega^1_{\ca{O}_{F_\infty}/\bb{Z}_p}\longrightarrow \Omega^1_{\ca{O}_{F_\infty}/\ca{O}_F}\longrightarrow 0.
	\end{align}
	Note that $\Omega^1_{\ca{O}_{F_\infty}/\ca{O}_F}$ is killed by $\pi\in \ak{m}_F \scr{D}_{F_\infty/F}$ (\ref{thm:differential}.(\ref{item:thm:differential-3})). Thus, for any $i,j,r\in\bb{N}$, the canonical morphism
	\begin{align}
		H^i(\Sigma, \ca{O}_{F_\infty}\otimes_{\ca{O}_F}\Omega^j_{\ca{O}_F/\bb{Z}_p}/p^r)\to H^i(\Sigma, \Omega^j_{\ca{O}_{F_\infty}/\bb{Z}_p}/p^r)
	\end{align}
	has kernel and cokernel both killed by $\pi^j$ by \ref{lem:pi-isom}. As $F$ is not arithmetic, $\Omega^j_{\ca{O}_F/\bb{Z}_p}/p^r$ is flat over $\ca{O}_F/p^r$ by \ref{prop:non-ari-flat}. Thus, by \ref{lem:non-ari-coh-twist} we have
	\begin{align}
		H^i(\Sigma, \ca{O}_{F_\infty}\otimes_{\ca{O}_F}\Omega^j_{\ca{O}_F/\bb{Z}_p}/p^r)=H^i(\Sigma,\ca{O}_{F_\infty}/p^r)\otimes_{\ca{O}_F/p^r}\Omega^j_{\ca{O}_F/\bb{Z}_p}/p^r.
	\end{align}
	Notice that $\ca{O}_F/p^r\to H^0(\Sigma,\ca{O}_{F_\infty}/p^r)$ is injective with cokernel killed by $\pi$ by \ref{lem:geo-diff-bound}. We see that $\Omega^j_{\ca{O}_F/\bb{Z}_p}/p^r\to H^0(\Sigma, \ca{O}_{F_\infty}\otimes_{\ca{O}_F}\Omega^j_{\ca{O}_F/\bb{Z}_p}/p^r)$ is also injective with cokernel killed by $\pi$. On the other hand, $H^i(\Sigma,\ca{O}_{F_\infty}/p^r)$ is killed by $\pi$ for any $i>0$ by \ref{lem:geo-diff-bound}. The conclusion follows immediately.
\end{proof}

\begin{mythm}\label{thm:geo}
	Assume that $F$ is geometric. Let $\pi,\varpi\in\ca{O}_F$ with norm $|\pi|<|\scr{D}_{F_\infty/F}|$ and $|\varpi|\leq|\zeta_p-1|$. Then, for any $q,r\in\bb{N}$, the canonical morphism induced by the cup product of \eqref{eq:lem:comp-1},
	\begin{align}\label{eq:thm:geo-1}
		\Omega^q_{\ca{O}_F/\bb{Z}_p}/p^r\Omega^q_{\ca{O}_F/\bb{Z}_p}\longrightarrow H^q(G_F,(\ca{O}_{\overline{F}}/p^r\ca{O}_{\overline{F}})\{q\}),
	\end{align}
	has kernel killed by $\pi^q\varpi^q$ and cokernel killed by $p^{2(q+1)}\pi^{\frac{(q+1)(q+2)}{2}}\varpi^{q^2}$.
\end{mythm}
\begin{proof}
	We put $\Sigma=\gal(F_\infty/F)$ and consider the convergent spectral sequence (\cite[\Luoma{5}.11.7]{abbes2016p})
	\begin{align}\label{eq:thm:geo-2}
		E_2^{i,j}=H^i(\Sigma,H^j(G_{F_\infty},\ca{O}_{\overline{F}}/p^r\ca{O}_{\overline{F}}))\Rightarrow H^{i+j}(G_F,\ca{O}_{\overline{F}}/p^r\ca{O}_{\overline{F}}).
	\end{align}
	Recall that the canonical morphism \eqref{eq:thm:coh-1},
	\begin{align}\label{eq:thm:geo-3}
		(\Omega^j_{\ca{O}_{F_\infty}/\bb{Z}_p}/p^r\Omega^j_{\ca{O}_{F_\infty}/\bb{Z}_p})\{-j\}\longrightarrow H^j(G_{F_\infty},\ca{O}_{\overline{F}}/p^r\ca{O}_{\overline{F}}),
	\end{align}
	has kernel killed by $\varpi^j\ak{m}_F$ and cokernel killed by $p^2\varpi^j\ak{m}_F$ by \ref{thm:coh}. Note that it is injective when $j=0$. Therefore, we see that the canonical morphism
	\begin{align}\label{eq:thm:geo-4}
		(\Omega^j_{\ca{O}_F/\bb{Z}_p}/p^r\Omega^j_{\ca{O}_F/\bb{Z}_p})\{-j\}\longrightarrow H^0(\Sigma,H^j(G_{F_\infty},\ca{O}_{\overline{F}}/p^r\ca{O}_{\overline{F}}))
	\end{align}
	has kernel killed by $\pi^j\varpi^j$ and cokernel killed by $p^2\pi^{j+1}\varpi^j$ by \ref{lem:geo-diff-isom}.(\ref{item:lem:geo-diff-isom-1}). Moreover, $H^i(\Sigma,H^j(G_{F_\infty},\ca{O}_{\overline{F}}/p^r\ca{O}_{\overline{F}}))$ is killed by $p^2\pi^{j+1}\varpi^{2j}$ for any $i>0$ by \ref{lem:geo-diff-isom}.(\ref{item:lem:geo-diff-isom-2}) and \ref{lem:pi-isom}.
	
	The spectral sequence \eqref{eq:thm:geo-2} induces a finite decreasing filtration $\mrm{F}^\bullet$ on $H^q(G_F,\ca{O}_{\overline{F}}/p^r\ca{O}_{\overline{F}})$ whose graded pieces $\mrm{gr}^i_{\mrm{F}}$ identifies with a subquotient of $E_2^{i,q-i}$. Moreover, there is a canonical exact sequence
	\begin{align}
		0\longrightarrow \mrm{F}^1(H^q(G_F,\ca{O}_{\overline{F}}/p^r))\longrightarrow H^q(G_F,\ca{O}_{\overline{F}}/p^r)\longrightarrow H^0(\Sigma,H^q(G_{F_\infty},\ca{O}_{\overline{F}}/p^r\ca{O}_{\overline{F}})).
	\end{align}
	On the one hand, $\mrm{F}^1(H^q(G_F,\ca{O}_{\overline{F}}/p^r))$ is killed by $\prod_{j=0}^{q-1}p^2\pi^{j+1}\varpi^{2j}=p^{2q}\pi^{\frac{q(q+1)}{2}}\varpi^{q(q-1)}$. On the other hand, since \eqref{eq:thm:geo-4} canonically factors through $H^q(G_F,\ca{O}_{\overline{F}}/p^r)$ (when $j=q$), we see that the canonical morphism $(\Omega^q_{\ca{O}_F/\bb{Z}_p}/p^r)\{-q\}\to H^q(G_F,\ca{O}_{\overline{F}}/p^r)$ has kernel killed by $\pi^q\varpi^q$ and cokernel killed by $p^2\pi^{q+1}\varpi^q\cdot p^{2q}\pi^{\frac{q(q+1)}{2}}\varpi^{q(q-1)}=p^{2(q+1)}\pi^{\frac{(q+1)(q+2)}{2}}\varpi^{q^2}$.
\end{proof}

\begin{mycor}\label{cor:geo}
	Assume that $F$ is geometric. Let $\pi,\varpi\in\ca{O}_F$ with norm $|\pi|<|\scr{D}_{F_\infty/F}|$ and $|\varpi|\leq|\zeta_p-1|$. Then, for any $q\in\bb{N}$, the $q$-th Galois cohomology group $H^q(G_F,\ca{O}_{\widehat{\overline{F}}}\{q\})$ is $p$-adically complete. Moreover, the canonical morphism \eqref{eq:rem:prop:comp-2},
	\begin{align}\label{eq:cor:geo-1}
		\widehat{\Omega}^q_{\ca{O}_F/\bb{Z}_p}\longrightarrow H^q(G_F,\ca{O}_{\widehat{\overline{F}}}\{q\}),
	\end{align}
	has kernel killed by $\pi^q\varpi^q$ and cokernel killed by $p^{4(q+1)}\pi^{(q+1)(q+2)+2q}\varpi^{2q(q+1)}$. In particular, it induces a canonical isomorphism after inverting $p$,
	\begin{align}\label{eq:cor:geo-2}
		\widehat{\Omega}^q_{\ca{O}_F/\bb{Z}_p}[1/p]\iso H^q(G_F,\widehat{\overline{F}}(q)).
	\end{align}
\end{mycor}
\begin{proof}
	There is a canonical exact sequence (\cite[\Luoma{2}.3.10]{abbes2016p})
	\begin{align}\label{eq:cor:geo-3}
		0\to  \rr^1\lim_{r\in\bb{N}}H^{q-1}(G_F,\ca{O}_{\overline{F}}/p^r\ca{O}_{\overline{F}})\to H^q(G_F,\ca{O}_{\widehat{\overline{F}}})\to \lim_{r\in\bb{N}} H^q(G_F,\ca{O}_{\overline{F}}/p^r\ca{O}_{\overline{F}})\to 0.
	\end{align}
	As the inverse system $((\Omega^q_{\ca{O}_F/\bb{Z}_p}/p^r\Omega^q_{\ca{O}_F/\bb{Z}_p})\{-q\})_{r\in\bb{N}}$ satisfies the Mittag-Leffler condition, we see that $\rr^1\lim(\Omega^q_{\ca{O}_F/\bb{Z}_p}/p^r\Omega^q_{\ca{O}_F/\bb{Z}_p})\{-q\}=0$. Thus, $\rr^1\lim H^q(G_F,\ca{O}_{\overline{F}}/p^r\ca{O}_{\overline{F}})$ is killed by $p^{2(q+1)}\pi^{\frac{(q+1)(q+2)}{2}+q}\varpi^{q(q+1)}$ by \ref{thm:geo} and \ref{lem:pi-isom}. In particular, it is $p$-adically complete.
	
	On the other hand, we claim that $\lim_{r\in\bb{N}}H^q(G_F,\ca{O}_{\overline{F}}/p^r\ca{O}_{\overline{F}})$ is $p$-adically complete with bounded $p$-power torsion. Indeed, it is $p$-adically complete as each $H^q(G_F,\ca{O}_{\overline{F}}/p^r\ca{O}_{\overline{F}})$ is killed by $p^r$ (\cite[\href{https://stacks.math.columbia.edu/tag/0G1Q}{0G1Q}]{stacks-project}). Notice that the canonical morphism induced by \eqref{eq:thm:geo-1},
	\begin{align}\label{eq:cor:geo-4}
		\widehat{\Omega}^q_{\ca{O}_F/\bb{Z}_p}\{-q\}=\lim_{r\in\bb{N}}(\Omega^q_{\ca{O}_F/\bb{Z}_p}/p^r\Omega^q_{\ca{O}_F/\bb{Z}_p})\{-q\} \longrightarrow   \lim_{r\in\bb{N}} H^q(G_F,\ca{O}_{\overline{F}}/p^r\ca{O}_{\overline{F}})
	\end{align}
	has kernel killed by $\pi^q\varpi^q$ by \ref{thm:geo} and cokernel killed by $p^{2(q+1)}\pi^{\frac{(q+1)(q+2)}{2}+q}\varpi^{q(q+1)}$ by \ref{thm:geo} and \ref{lem:pi-isom}. Therefore, the latter has bounded $p$-power torsion as $\widehat{\Omega}^q_{\ca{O}_F/\bb{Z}_p}$ does so (\ref{prop:non-ari-flat}). 
	
	Therefore, $H^q(G_F,\ca{O}_{\widehat{\overline{F}}})$ is also $p$-adically complete by \ref{lem:complete}. Then, \eqref{eq:cor:geo-4} coincides with the composition of the canonical morphisms
	\begin{align}\label{eq:cor:geo-5}
		\widehat{\Omega}^q_{\ca{O}_F/\bb{Z}_p}\{-q\}\longrightarrow H^q(G_F,\ca{O}_{\widehat{\overline{F}}}) \longrightarrow   \lim_{r\in\bb{N}} H^q(G_F,\ca{O}_{\overline{F}}/p^r\ca{O}_{\overline{F}})
	\end{align}
	where the first arrow is \eqref{eq:rem:prop:comp-2} and the second arrow is defined in \eqref{eq:cor:geo-3}. The conclusion follows from diagram chasing.
\end{proof}

\begin{myprop}[{\cite[4.23, 4.25]{he2024perfd}}]\label{prop:fal-ext-delta}
	Assume that $F$ is geometric. If the canonical morphism $\widehat{\overline{F}}\otimes_{\widehat{F}}(\scr{E}_{\ca{O}_{\overline{F}}})^{G_F}\to \scr{E}_{\ca{O}_{\overline{F}}}$ {\rm(\ref{para:notation-fal-ext})} is an isomorphism, then $F$ is pre-perfectoid.
\end{myprop}
\begin{proof}
	We give another proof slightly different from \cite[4.23, 4.25]{he2024perfd} by our computation of Galois cohomology. The assumption implies that the canonical morphism $\jmath:\widehat{\overline{F}}\otimes_{\widehat{F}}(\scr{E}_{\ca{O}_{\overline{F}}})^{G_F}\to \widehat{\overline{F}}\otimes_F\Omega^1_{F/\bb{Q}_p}$ is surjective and thus so is $(\scr{E}_{\ca{O}_{\overline{F}}})^{G_F}\to \widehat{F}\otimes_F\Omega^1_{F/\bb{Q}_p}$. Hence, the coboundary map $\delta$ \eqref{eq:para:notation-fal-ext-3} is zero (see also \cite[4.22]{he2024perfd}). Then, we see that the image of  $\widehat{F}\otimes_F\Omega^1_{F/\bb{Q}_p}\to	\widehat{\Omega}^1_{\ca{O}_F/\bb{Z}_p}[1/p]$ is zero by \ref{lem:fal-ext-delta} and \eqref{eq:cor:geo-2}. However, it is also dense in $\widehat{\Omega}^1_{\ca{O}_F/\bb{Z}_p}[1/p]$ with respect to the $p$-adic topology defined by $\widehat{\Omega}^1_{\ca{O}_F/\bb{Z}_p}$. This forces $\widehat{\Omega}^1_{\ca{O}_F/\bb{Z}_p}[1/p]=0$ which implies that the torsion-free module $\widehat{\Omega}^1_{\ca{O}_F/\bb{Z}_p}$ is zero (\ref{prop:non-ari-flat}). Thus, $\Omega^1_{\ca{O}_F/\bb{Z}_p}$ is $p$-divisible so that $F$ is pre-perfectoid (\cite[6.6.6]{gabber2003almost}).
\end{proof}

\subsection{The arithmetic case}

\begin{mylem}\label{lem:n_0}
	We put 
	\begin{align}\label{eq:lem:n_0-0}
		n_p=\left\{\begin{array}{ll}
			2&\trm{if }p=2,\\
			1&\trm{if }p\neq 2.
		\end{array}\right.
	\end{align}
	Let $n_0\in\bb{N}\cup\{\infty\}$ be the maximal element such that $\zeta_{p^{n_0}}\in F_{n_p}$ (so that $n_0\geq n_p$ and $F_{n_0}=F_{n_p}$). Consider the canonical continuous group homomorphisms
	\begin{align}\label{eq:lem:n_0-1}
		\xymatrix{
			\gal(F_\infty/F)\ar[r]^-{\chi}&\bb{Z}_p^\times\ar[r]^-{\log}&\bb{Z}_p,
		}
	\end{align}
	where $\chi$ is the cyclotomic character characterized by $\sigma(\zeta_{p^n})=\zeta_{p^n}^{\chi(\sigma)}$ for any $\sigma\in\gal(F_\infty/F)$ and $n\in\bb{N}$, and $\log$ is the $p$-adic logarithm map {\rm(\cite[(3.6.3)]{he2022sen})}.
	\begin{enumerate}
		\renewcommand{\labelenumi}{{\rm(\theenumi)}}
		\item For any $n\in\bb{N}_{>0}$, the canonical morphisms \eqref{eq:lem:n_0-1} induce continuous homomorphisms
		\begin{align}\label{eq:lem:n_0-2}
			\xymatrix{
				\gal(F_\infty/F_n)\ar[r]^-{\chi}&1+p^n\bb{Z}_p\ar[r]^-{\log}&p^n\bb{Z}_p,
			}
		\end{align}
		which are isomorphisms if $n\geq n_0$. \label{item:lem:n_0-1}
		\item The Galois group $\gal(F_{n_0}/F)$ is finite cyclic of order dividing $n_p(p-1)$.\label{item:lem:n_0-2}
	\end{enumerate}
\end{mylem}
\begin{proof}
	Note that $\bb{Z}_p^\times\cong (1+p^{n_p}\bb{Z}_p)\times(\bb{Z}/p^{n_p}\bb{Z})^\times\cong  (1+p^{n_p}\bb{Z}_p)\times\bb{Z}/n_p\bb{Z}\times \bb{Z}/(p-1)\bb{Z}$. We remark that $\log:\bb{Z}_p^\times\to\bb{Z}_p$ annihilates the torsion part and induces a homomorphism $1+p^n\bb{Z}_p\to p^n\bb{Z}_p$ for any $n\in\bb{N}_{>0}$ which is an isomorphism if $n\geq n_p$.
	
	We put $K=\bb{Q}_p$, $K_n=\bb{Q}_p(\zeta_{p^n})$ and $K_\infty=\bigcup_{n\in\bb{N}}K_n$. Recall that $\chi$ induces isomorphisms $\gal(K_\infty/K)\iso \bb{Z}_p^\times$ and $\gal(K_\infty/K_n)\iso 1+p^n\bb{Z}_p$ for any $n\in\bb{N}_{>0}$ by \ref{lem:cyclotomic}.(\ref{item:lem:cyclotomic-2}). As $\chi:\gal(F_\infty/F_n)\to\bb{Z}_p^\times$ factors through $\gal(K_\infty/K_n)$, we obtain \eqref{eq:lem:n_0-2}. 
	
	Notice that all the nontrivial closed subgroups of $\gal(K_\infty/K_{n_p})\cong 1+p^{n_p}\bb{Z}_p\cong p^{n_p}\bb{Z}_p$ are given by $\gal(K_\infty/K_n)\cong 1+p^n\bb{Z}_p\cong p^n\bb{Z}_p$ ($n\in\bb{N}_{\geq n_p}$). Thus, there is a unique element $n_0\in\bb{N}_{\geq n_p}\cup\{\infty\}$ such that $K_{n_0}=K_\infty\cap F_{n_p}$ by Galois theory. We also see that this $n_0$ is also the maximal element in $\bb{N}\cup\{\infty\}$ such that $K_{n_0}\subseteq F_{n_p}$. Therefore, for any $n\in\bb{N}_{\geq n_0}$, the arrows in \eqref{eq:lem:n_0-2} are isomorphisms by the canonical isomorphism $\gal(F_\infty/F_n)\cong \gal(K_\infty/K_n)$ by Galois theory.
	
	On the other hand, as $\gal(K_{n_p}/K)\cong \bb{Z}_p^\times/(1+p^{n_p}\bb{Z}_p)\cong \bb{Z}/n_p\bb{Z}\times \bb{Z}/(p-1)\bb{Z}$ (\ref{lem:cyclotomic}.(\ref{item:lem:cyclotomic-2})) is a finite cyclic group of order $n_p(p-1)$, we see that its subgroup $\gal(F_{n_0}/F)$ is finite cyclic of order dividing $n_p(p-1)$.
\end{proof}

\begin{mylem}\label{lem:ari-bc}
	Assume that $F$ is arithmetic and contains $\zeta_{p^{n_p}}$ {\rm(\ref{lem:n_0})}. Let $n_0$ be the maximal integer such that $\zeta_{p^{n_0}}\in F$ and let $\pi\in\mrm{Ann}_{\ca{O}_F}(\Omega^1_{\ca{O}_F/\bb{Z}_p[\zeta_{p^{n_0}}]}[p^\infty])$. Then, for any integers $n\geq m\geq n_0$, we have
	\begin{align}
		\scr{D}_{F_n/F_m}&\subseteq p^{n-m}\pi^{-1}\ca{O}_{F_n},\label{eq:lem:ari-bc-1}\\
		p\pi\ak{m}_{F_n}\subseteq \ca{O}_{F_m}[\zeta_{p^n}]&=\bb{Z}_p[\zeta_{p^n}]\otimes_{\bb{Z}_p[\zeta_{p^m}]}\ca{O}_{F_m}\subseteq \ca{O}_{F_n}.\label{eq:lem:ari-bc-2}
	\end{align}
\end{mylem}
\begin{proof}
	Firstly, note that $n_0$ is well-defined as $F$ is arithmetic and we have $F=F_{n_0}$. We put $K=\bb{Q}_p$ and $K_n=\bb{Q}_p(\zeta_{p^n})$. Note that $\ca{O}_{K_n}=\bb{Z}_p[\zeta_{p^n}]$ and $\Omega^1_{\ca{O}_{K_n}/\ca{O}_{K_{n_0}}}\cong p^{-n}\ca{O}_{K_n}/p^{-n_0}\ca{O}_{K_n}$ by \ref{lem:cyclotomic}. As $F_n=K_n\otimes_{K_{n_0}}F$ for any $n\in\bb{N}_{\geq n_0}$ (\ref{lem:n_0}), we can apply the ramified base change theorem \ref{thm:ramified-bc} so that $\scr{D}_{F_n/F}\subseteq p^{n-n_0}\pi^{-1}\ca{O}_{F_n}$ and $p\pi\ak{m}_{F_n}\subseteq \ca{O}_F[\zeta_{p^n}]\subseteq \ca{O}_{F_n}$. Then, for any $n\geq m\geq n_0$, we have
	\begin{align}
		p^{m-n_0}\widetilde{\ak{m}}_{F_n}\scr{D}_{F_n/F_m}\subseteq \widetilde{\ak{m}}_{F_n}\scr{D}_{F_n/F_m}\scr{D}_{F_m/F}\subseteq \scr{D}_{F_n/F}\subseteq p^{n-n_0}\pi^{-1}\ca{O}_{F_n},
	\end{align}
	where the first inclusion follows from $p^{m-n_0}\ca{O}_{K_m}=\widetilde{\ak{m}}_{K_{n_0}}\scr{D}_{K_m/K_{n_0}}\subseteq\scr{D}_{F_m/F}$ (\ref{prop:ramified-bc}) and the second inclusion follows from $\scr{D}_{F_n/F}^{\al}=\scr{D}_{F_n/F_m}^{\al}\cdot\scr{D}_{F_m/F}^{\al}$ (\cite[4.1.25]{gabber2003almost}). Therefore, we have $\widetilde{\ak{m}}_{F_n}\scr{D}_{F_n/F_m}\subseteq p^{n-m}\pi^{-1}\ca{O}_{F_n}$ which implies \eqref{eq:lem:ari-bc-1} by \eqref{eq:para:notation-norm-3}. Finally, \eqref{eq:lem:ari-bc-2} follows from the fact that $\ca{O}_{F_m}[\zeta_{p^n}]=\ca{O}_F[\zeta_{p^n}]$ and $F_n=K_n\otimes_{K_m}F_m$.
\end{proof}

\begin{myprop}\label{prop:ari-diff}
	Assume that $F$ is arithmetic. Let $K_n=\bb{Q}_p(\zeta_{p^n})$ and $K_\infty=\bigcup_{n\in\bb{N}}K_n$. Then, there exists a nonzero element $\pi\in\ca{O}_F$ such that for any $n\in\bb{N}\cup\{\infty\}$, the canonical morphism
	\begin{align}\label{eq:prop:ari-diff-1}
		\ca{O}_{F_n}\otimes_{\ca{O}_F}\Omega^1_{\ca{O}_F/\bb{Z}_p}\longrightarrow \Omega^1_{\ca{O}_{F_n}/\ca{O}_{K_n}}
	\end{align}
	has kernel and cokernel both killed by $\pi$. In particular, the canonical morphism of $p$-adic completions
	\begin{align}\label{eq:prop:ari-diff-2}
		\ca{O}_{F_\infty}\widehat{\otimes}_{\ca{O}_F}\Omega^1_{\ca{O}_F/\bb{Z}_p}\longrightarrow \widehat{\Omega}^1_{\ca{O}_{F_\infty}/\bb{Z}_p}
	\end{align}
	has kernel and cokernel killed by $\pi^2$.
\end{myprop}
\begin{proof}
	Firstly, note that $F_{n_p}$ is arithmetic as a finite extension of $F$ (\ref{lem:n_0}) by \ref{prop:ari-geo-basic}.(\ref{item:prop:ari-geo-basic-2}). Let $n_0$ be the maximal integer such that $\zeta_{p^{n_0}}\in F_{n_p}$.
	Consider the canonical exact sequences (\ref{thm:differential}.(\ref{item:thm:differential-1}))
	\begin{align}
		0&\longrightarrow \ca{O}_{F_{n_0}}\otimes_{\ca{O}_F}\Omega^1_{\ca{O}_F/\bb{Z}_p}\longrightarrow \Omega^1_{\ca{O}_{F_{n_0}}/\bb{Z}_p}\longrightarrow \Omega^1_{\ca{O}_{F_{n_0}}/\ca{O}_F}\longrightarrow 0,\\
		0&\longrightarrow \ca{O}_{F_{n_0}}\otimes_{\ca{O}_{K_{n_0}}}\Omega^1_{\ca{O}_{K_{n_0}}/\bb{Z}_p}\longrightarrow \Omega^1_{\ca{O}_{F_{n_0}}/\bb{Z}_p}\longrightarrow \Omega^1_{\ca{O}_{F_{n_0}}/\ca{O}_{K_{n_0}}}\longrightarrow 0.
	\end{align}
	As $\Omega^1_{\ca{O}_{K_{n_0}}/\bb{Z}_p}$ and $\Omega^1_{\ca{O}_{F_{n_0}}/\ca{O}_F}$ are both killed by a power of $p$, we see that the kernel and cokernel of $\ca{O}_{F_{n_0}}\otimes_{\ca{O}_F}\Omega^1_{\ca{O}_F/\bb{Z}_p}\longrightarrow \Omega^1_{\ca{O}_{F_{n_0}}/\ca{O}_{K_{n_0}}}$ are killed by a power of $p$.
	
	Let $\pi\in\ca{O}_F$ with norm $|\pi|<|\mrm{Ann}_{\ca{O}_{F_{n_0}}}(\Omega^1_{\ca{O}_{F_{n_0}}/\ca{O}_{K_{n_0}}}[p^\infty])|$. Then, for any $n\in\bb{N}_{\geq n_0}$, we have $p\pi \ca{O}_{F_n}\subseteq \ca{O}_{K_n}\otimes_{\ca{O}_{K_{n_0}}}\ca{O}_{F_{n_0}}\subseteq \ca{O}_{F_n}$ by \eqref{eq:lem:ari-bc-2}. Thus, the kernel and cokernel of the canonical morphism $\ca{O}_{F_n}\otimes_{\ca{O}_{F_{n_0}}}\Omega^1_{\ca{O}_{F_{n_0}}/\ca{O}_{K_{n_0}}}\to \Omega^1_{\ca{O}_{F_n}/\ca{O}_{K_n}}$ are killed by $(p\pi)^{17}$ (\cite[7.4]{he2022sen}). This proves \eqref{eq:prop:ari-diff-1}. 
	
	The ``in particular" part follows from \ref{lem:pi-isom} and $\widehat{\Omega}^1_{\ca{O}_{F_\infty}/\ca{O}_{K_\infty}}=\widehat{\Omega}^1_{\ca{O}_{F_\infty}/\bb{Z}_p}$ \eqref{eq:prop:rhom-2}.
\end{proof}

\begin{mylem}\label{lem:ari-trace}
	Assume that $F$ is arithmetic and contains $\zeta_{p^{n_p}}$ {\rm(\ref{lem:n_0})}. Let $n_0$ be the maximal integer such that $\zeta_{p^{n_0}}\in F$ and let $\pi\in\mrm{Ann}_{\ca{O}_F}(\Omega^1_{\ca{O}_F/\bb{Z}_p[\zeta_{p^{n_0}}]}[p^\infty])$. For any $m\in\bb{N}_{\geq n_0}$, consider Tate's normalized trace map
	\begin{align}\label{eq:lem:ari-trace-1}
		\ca{T}_m=\colim_{n\in\bb{N}_{\geq m}}p^{m-n}\mrm{Tr}_{F_n/F_m}:F_\infty\longrightarrow F_m.
	\end{align}
	\begin{enumerate}
		\renewcommand{\labelenumi}{{\rm(\theenumi)}}
		\item For any $x\in F_\infty$, we have $|\ca{T}_m(x)|\leq |\pi\ak{m}_{F_m}|^{-1}\cdot|x|$.\label{item:lem:ari-trace-1}
		\item Let $\lambda\in \ca{O}_{F_m}$ with $|\lambda-1|<|p\pi\ak{m}_{F_m}|$ and let $\sigma_m\in \gal(F_\infty/F_m)$ be a topological generator \eqref{eq:lem:n_0-2}. Then, for any $x\in \ke(\ca{T}_m)$, we have $|\sigma_m(x)-\lambda x|\geq |p\pi\ak{m}_{F_m}|\cdot |x|$.\label{item:lem:ari-trace-2}
	\end{enumerate}
\end{mylem}
\begin{proof}
	(\ref{item:lem:ari-trace-1}) For any $n\in\bb{N}_{\geq m}$ and $x\in F_n$, we have
	\begin{align}
		|\ca{T}_m(x)|=|p^{m-n}\mrm{Tr}_{F_n/F_m}(x)|\leq |p^{m-n}\ak{m}_{F_m}^{-1}\ak{m}_{F_n}\scr{D}_{F_n/F_m}|\cdot|x|
	\end{align}
	by \ref{cor:different-trace}. Moreover, it is no more than $| \ak{m}_{F_m}^{-1}\ak{m}_{F_n}\pi^{-1}|\cdot|x|\leq | \pi\ak{m}_{F_m}|^{-1}\cdot|x|$ by \eqref{eq:lem:ari-bc-1}.
	
	(\ref{item:lem:ari-trace-2}) For any $n\in\bb{N}_{\geq m}$ and $x\in F_n$, we have
	\begin{align}
		|\ca{T}_m(x)-x|=|p^{m-n}\mrm{Tr}_{F_n/F_m}(x)-x|\leq |p^{m-n-1}\ak{m}_{F_m}^{-1}\ak{m}_{F_n}\scr{D}_{F_n/F_m}|\cdot |\sigma_m(x)-x|
	\end{align}
	by \ref{prop:trace-tau}. Moreover, it is no more than $|p^{-1}\ak{m}_{F_m}^{-1}\ak{m}_{F_n}\pi^{-1}|\cdot |\sigma_m(x)-x|\leq |p\pi\ak{m}_{F_m}|^{-1}\cdot|\sigma_m(x)-x|$ by \eqref{eq:lem:ari-bc-1}. In particular, for any $x\in\ke(\ca{T}_m)$, we have $|\sigma_m(x)-x|\geq |p\pi\ak{m}_{F_m}|\cdot |x|$. As $|x-\lambda x|<|p\pi\ak{m}_{F_m}|\cdot |x|$, we have $|\sigma_m(x)-\lambda x|=|\sigma_m(x)-x+x-\lambda x|=|\sigma_m(x)-x|\geq |p\pi\ak{m}_{F_m}|\cdot |x|$.
\end{proof}

\begin{mylem}\label{lem:ari-D}
	Assume that $F$ is arithmetic and contains $\zeta_{p^{n_p}}$ {\rm(\ref{lem:n_0})}. Let $n_0$ be the maximal integer such that $\zeta_{p^{n_0}}\in F$ and let $\pi\in\mrm{Ann}_{\ca{O}_F}(\Omega^1_{\ca{O}_F/\bb{Z}_p[\zeta_{p^{n_0}}]}[p^\infty])$. For any $m\in\bb{N}_{\geq n_0}$, we put $D_m=\ca{O}_{F_\infty}\cap \ke(\ca{T}_m)$, where $\ca{T}_m$ is Tate's normalized trace map \eqref{eq:lem:ari-trace-1}.
	\begin{enumerate}
		\renewcommand{\labelenumi}{{\rm(\theenumi)}}
		\item The canonical $G_F$-equivariant morphism of $\ca{O}_{F_m}$-modules
		\begin{align}\label{eq:lem:ari-D-1}
			\ca{O}_{F_m}\oplus D_m\longrightarrow \ca{O}_{F_\infty}
		\end{align}
		is injective with cokernel killed by $\pi\ak{m}_{F_m}$. \label{item:lem:ari-D-1}
		\item Let $\sigma_m\in \gal(F_\infty/F_m)$ be a topological generator \eqref{eq:lem:n_0-2}. The $\ca{O}_{F_m}$-linear endomorphism $\sigma_m-1$ on $\ca{O}_{F_\infty}$ annihilates $\ca{O}_{F_m}$ and stabilizes $D_m$. Moreover, it induces an injection
		\begin{align}\label{eq:lem:ari-D-2}
			\sigma_m-1:D_m\longrightarrow D_m
		\end{align}
		with cokernel killed by $p\pi\ak{m}_{F_m}$.\label{item:lem:ari-D-2}
		\item Let $\lambda$ be an element of $\ca{O}_{F_m}$ not a root of unity with $|\lambda-1|<|p\pi\ak{m}_{F_m}|$. Then, the $\ca{O}_{F_m}$-linear endomorphism $\sigma_m-\lambda$ on $\ca{O}_{F_\infty}$ induces an injection
		\begin{align}\label{eq:lem:ari-D-3}
			\sigma_m-\lambda:\ca{O}_{F_m}\oplus D_m\longrightarrow \ca{O}_{F_m}\oplus D_m
		\end{align}
		with cokernel killed by $\lambda-1$.\label{item:lem:ari-D-3}
	\end{enumerate}
\end{mylem}
\begin{proof}
	(\ref{item:lem:ari-D-1}) Since $\ca{T}_m=\colim_{n\in\bb{N}}p^{-n}(1+\sigma_m+\cdots+\sigma_m^{p^n-1}):F_\infty\to F_m$ commutes with the action of $G_F$ as $\gal(F_\infty/F)\cong\bb{Z}_p$ is commutative, $D_m=\ca{O}_{F_\infty}\cap\ke(\ca{T}_m)$ is a $G_F$-stable $\ca{O}_{F_m}$-submodule of $\ca{O}_{F_\infty}$. It is clear that $D_m[1/p]=\ke(\ca{T}_m)\subseteq F_\infty$. Thus, we see that $\ca{O}_{F_m}\oplus D_m\to \ca{O}_{F_\infty}$ is injective and induces the decomposition $F_m\oplus\ke(\ca{T}_m)=F_\infty$ (induced by the idempotent operator $\ca{T}_m$ on $F_\infty$) after inverting $p$. 
	
	For any $x\in F_\infty$, this decomposition gives $x=\ca{T}_m(x)+(1-\ca{T}_m)(x)$. As $|\ca{T}_m(x)|\leq |\pi\ak{m}_{F_m}|^{-1}\cdot|x|$ by \ref{lem:ari-trace}.(\ref{item:lem:ari-trace-1}), we have $\ca{T}_m(x)\in\ca{O}_{F_m}$ (and thus $(1-\ca{T}_m)(x)\in D_m$) if $x\in \pi\ak{m}_{F_m}$. This shows that the quotient of $\ca{O}_{F_m}\oplus D_m\subseteq \ca{O}_{F_\infty}$ is killed by $\pi\ak{m}_{F_m}$.
	
	(\ref{item:lem:ari-D-2}) The morphism $\sigma_m-1:D_m\to D_m$ is injective because the kernel of $\sigma_m-1:F_\infty\to F_\infty$ is $F_m$. It remains to show that $p\pi\ak{m}_{F_m}\cdot D_m\subseteq (\tau-1)(D_m)$. Recall that $\sigma_m-1$ induces an isomorphism $D_m[1/p]\iso D_m[1/p]$. Indeed, since $\sigma_m-1$ is an $F_m$-linear injective endomorphism of $D_m[1/p]=\colim_{n\in\bb{N}_{\geq m}}\ke(\ca{T}_m|_{\ca{F}_n})$, we see that it is an isomorphism by a dimension argument on each $G_F$-stable finite-dimensional subspace $\ke(\ca{T}_m|_{\ca{F}_n})$. Hence, for any $x\in D_m$ and $\epsilon\in\ak{m}_{F_m}$, there exists $y\in D_m[1/p]$ with $(\sigma_m-1)(y)=p\pi\epsilon x$. Since $|y|\leq |p\pi\ak{m}_{F_m}|^{-1}\cdot |\sigma_m(y)-y|\leq |x|\leq 1$ by \ref{lem:ari-trace}.(\ref{item:lem:ari-trace-2}), we have $y\in\ca{O}_{F_\infty}$. Thus, $y\in \ca{O}_{F_\infty}\cap D_m[1/p]=D_m$ which shows that the cokernel of $\sigma_m-1:D_m\to D_m$ is killed by $p\pi\epsilon$ for any $\epsilon\in\ak{m}_{F_m}$.
	
	(\ref{item:lem:ari-D-3}) Firstly, we claim that $\sigma_m-\lambda:F_n\to F_n$ is injective (so that isomorphic) for any $n\in\bb{N}_{\geq m}$. Indeed, if $\sigma_m(x)=\lambda x$ for some $x\in F_n$, then $x=\sigma_m^{p^n}(x)=\lambda^{p^n}x$ so that $x=0$ as $\lambda^{p^n}\neq 1$. In particular, the claim implies that $\sigma_m-\lambda:\ca{O}_{F_m}\oplus D_m\longrightarrow \ca{O}_{F_m}\oplus D_m$ is injective. Since the operator $\sigma_m-\lambda$ on $\ca{O}_{F_m}$ is the multiplication by $1-\lambda$, its cokernel is killed by $1-\lambda$. On the other hand, for any $x\in D_m$ and $\epsilon\in \ak{m}_{F_m}$, there exists $y\in D_m[1/p]$ with $(\sigma_m-\lambda)(y)=p\pi\epsilon x$. Since $|y|\leq |p\pi\ak{m}_{F_m}|^{-1}\cdot |\sigma_m(y)-\lambda y|\leq |x|\leq 1$ by \ref{lem:ari-trace}.(\ref{item:lem:ari-trace-2}), we have $y\in\ca{O}_{F_\infty}$. Thus, $y\in \ca{O}_{F_\infty}\cap D_m[1/p]=D_m$ which shows that the cokernel of $\sigma_m-\lambda:D_m\to D_m$ is killed by $p\pi\ak{m}_{F_m}$ and thus killed by $\lambda-1$.
\end{proof}

\begin{mylem}\label{lem:ari-sigma-coh}
	Assume that $F$ is arithmetic and contains $\zeta_{p^{n_p}}$ {\rm(\ref{lem:n_0})}. Let $n_0$ be the maximal integer such that $\zeta_{p^{n_0}}\in F$, $\pi\in\mrm{Ann}_{\ca{O}_F}(\Omega^1_{\ca{O}_F/\bb{Z}_p[\zeta_{p^{n_0}}]}[p^\infty])$, $m\in\bb{N}_{\geq n_0}$, $q,r\in\bb{N}$.
	\begin{enumerate}
		\renewcommand{\labelenumi}{{\rm(\theenumi)}}
		\item The canonical homomorphism
		\begin{align}\label{eq:lem:ari-sigma-coh-1}
			\wedge^q_{\ca{O}_{F_m}}\ho_{\bb{Z}_p}(\gal(F_\infty/F_m),\ca{O}_{F_m}/p^r\ca{O}_{F_m})\longrightarrow H^q(\gal(F_\infty/F_m),\ca{O}_{F_\infty}/p^r\ca{O}_{F_\infty})
		\end{align}
		has kernel killed by $\pi\ak{m}_{F_m}$ and cokernel killed by $p\pi^2\ak{m}_{F_m}^2$. Moreover, the source and target are both zero for $q\notin\{0,1\}$. \label{item:lem:ari-sigma-coh-1}
		\item For any nonzero integer $n$ such that $|np^m|< |p\pi\ak{m}_{F_m}|$, $H^q(\gal(F_\infty/F_m),(\ca{O}_{F_\infty}/p^r\ca{O}_{F_\infty})\{n\})$ is killed by $np^m\pi\ak{m}_{F_m}$ for $q\in\{0,1\}$ and is zero for $q\notin \{0,1\}$.\label{item:lem:ari-sigma-coh-2}
	\end{enumerate}
\end{mylem}
\begin{proof}
	(\ref{item:lem:ari-sigma-coh-1}) As $\gal(F_\infty/F_m)\cong \bb{Z}_p$ \eqref{eq:lem:n_0-2}, the canonical morphism induced by the cup product
	\begin{align}\label{eq:lem:ari-sigma-coh-2}
		\wedge^q_{\ca{O}_{F_m}}\ho_{\bb{Z}_p}(\gal(F_\infty/F_m),\ca{O}_{F_m}/p^r\ca{O}_{F_m})\longrightarrow H^q(\gal(F_\infty/F_m),\ca{O}_{F_m}/p^r\ca{O}_{F_m})
	\end{align}
	is an isomorphism (\cite[\Luoma{2}.3.30, \Luoma{2}.6.12]{abbes2016p}). Notice that the canonical morphism $\ca{O}_{F_m}\oplus D_m\to \ca{O}_{F_\infty}$ \eqref{eq:lem:ari-D-1} is injective with cokernel killed by $\pi\ak{m}_{F_m}$ by \ref{lem:ari-D}.(\ref{item:lem:ari-D-1}). Thus, the canonical morphism
	\begin{align}\label{eq:lem:ari-sigma-coh-3}
		H^q(\gal(F_\infty/F_m),\ca{O}_{F_m}/p^r)\oplus H^q(\gal(F_\infty/F_m),D_m/p^r)\longrightarrow H^q(\gal(F_\infty/F_m),\ca{O}_{F_\infty}/p^r)
	\end{align}
	has kernel and cokernel both killed by $\pi\ak{m}_{F_m}$ (\ref{lem:pi-isom}). As $\gal(F_\infty/F_m)= \bb{Z}_p\sigma_m$, $\rr\Gamma(\gal(F_\infty/F_m),D_m/p^rD_m)$ is represented by $D_m/p^rD_m\stackrel{\sigma_m-1}{\longrightarrow} D_m/p^rD_m$  and the ``moreover" part holds (\cite[\Luoma{2}.3.23]{abbes2016p}). Since $\sigma_m-1:D_m\to D_m$ is injective whose cokernel is killed by $p\pi\ak{m}_{F_m}$ by \ref{lem:ari-D}.(\ref{item:lem:ari-D-2}), the kernel and cokernel of $\sigma_m-1:D_m/p^rD_m\to D_m/p^rD_m$ are also killed by $p\pi\ak{m}_{F_m}$ by \ref{lem:pi-isom}. Then, the conclusion follows.
	
	(\ref{item:lem:ari-sigma-coh-2}) As $\gal(F_\infty/F_m)= \bb{Z}_p\sigma_m$, $\rr\Gamma(\gal(F_\infty/F_m),(\ca{O}_{F_\infty}/p^r\ca{O}_{F_\infty})\{n\})$ is represented by $\chi(\sigma_m)^n\cdot\sigma_m-1:\ca{O}_{F_\infty}/p^r\ca{O}_{F_\infty}\to \ca{O}_{F_\infty}/p^r\ca{O}_{F_\infty}$ (\cite[\Luoma{2}.3.23]{abbes2016p} and \eqref{eq:defn:int-tate-twist-3}), where $\chi:\gal(F_\infty/F_m)\iso 1+p^m\bb{Z}_p$ is the cyclotomic character \eqref{eq:lem:n_0-2}. As $\sigma_m$ is a topological generator of $\gal(F_\infty/F_m)$, we have $\chi(\sigma_m)^n\in 1+np^m\bb{Z}_p^\times$ as $\log(\chi(\sigma_m)^n)=n\log(\chi(\sigma_m))\in np^m\bb{Z}_p^\times$. We put $\lambda=\chi(\sigma_m)^{-n}$ so that $|\lambda-1|=|np^m|<|p\pi\ak{m}_{F_m}|$. Then, $\sigma_m-\lambda:\ca{O}_{F_m}\oplus D_m\longrightarrow \ca{O}_{F_m}\oplus D_m$ is injective with cokernel killed by $np^m$ by \ref{lem:ari-D}.(\ref{item:lem:ari-D-3}).
	Consider the commutative diagram
	\begin{align}
		\xymatrix{
			\ca{O}_{F_m}\oplus D_m\ar[d]\ar[rr]^-{\sigma_m-\lambda} &&\ca{O}_{F_m}\oplus D_m\ar[d]^-{\cdot \lambda^{-1}}\\
			\ca{O}_{F_\infty}\ar[rr]^-{\chi(\sigma_m)^n\cdot\sigma_m-1}&& \ca{O}_{F_\infty}
		}
	\end{align}
	where the vertical arrows are injective with cokernel killed by $\pi\ak{m}_{F_m}$ (\ref{lem:ari-D}.(\ref{item:lem:ari-D-1})). We see that $\chi(\sigma_m)^n\cdot\sigma_m-1:\ca{O}_{F_\infty}\to \ca{O}_{F_\infty}$ is injective (by inverting $p$) with cokernel killed by $np^m\pi\ak{m}_{F_m}$. Then, the kernel and cokernel of $\chi(\sigma_m)^n\cdot\sigma_m-1:\ca{O}_{F_\infty}/p^r\ca{O}_{F_\infty}\to \ca{O}_{F_\infty}/p^r\ca{O}_{F_\infty}$ are both killed by $np^m\pi\ak{m}_{F_m}$ by \ref{lem:pi-isom}.
\end{proof}

\begin{myrem}\label{rem:ari-D}
	In \eqref{eq:lem:ari-sigma-coh-1}, the $\ca{O}_{F_m}/p^r\ca{O}_{F_m}$-module $\ho_{\bb{Z}_p}(\gal(F_\infty/F_m),\ca{O}_{F_m}/p^r\ca{O}_{F_m})$ is free with basis $p^{-m}(\log\circ\chi)$. Indeed, $\log\circ\chi:\gal(F_\infty/F_m)\iso p^m\bb{Z}_p$ \eqref{eq:lem:n_0-2} is a canonical isomorphism and $\ho_{\bb{Z}_p}(p^m\bb{Z}_p,\ca{O}_{F_m}/p^r\ca{O}_{F_m})$ is free with basis $f:p^m\bb{Z}_p\to\ca{O}_{F_m}/p^r\ca{O}_{F_m}$ sending $p^m$ to $1$.
\end{myrem}

\begin{myprop}\label{prop:ari-D}
	Assume that $F$ is arithmetic. Then, there exists a nonzero element $\pi\in\ca{O}_F$ such that the following statements hold for any $m,r\in\bb{N}$.
	\begin{enumerate}
		\renewcommand{\labelenumi}{{\rm(\theenumi)}}
		\item For $q=0$, the canonical morphism 
		\begin{align}\label{eq:prop:ari-D-1}
			\ca{O}_{F_m}/p^r\ca{O}_{F_m}\longrightarrow H^0(\gal(F_\infty/F_m),\ca{O}_{F_\infty}/p^r\ca{O}_{F_\infty})
		\end{align}
		has kernel and cokernel killed by $\pi$.\label{item:prop:ari-D-1}
		\item For $q=1$, the canonical morphism
		\begin{align}\label{eq:prop:ari-D-2}
			\ca{O}_{F_m}/p^r\ca{O}_{F_m}\longrightarrow H^1(\gal(F_\infty/F_m),\ca{O}_{F_\infty}/p^r\ca{O}_{F_\infty}),
		\end{align}
		sending $1$ to $p^{-m}(\log\circ \chi)$ \eqref{eq:lem:n_0-2}, has kernel and cokernel killed by $\pi$.\label{item:prop:ari-D-2}
		\item For any $q\in\bb{N}_{\geq 2}$, $H^q(\gal(F_\infty/F_m),\ca{O}_{F_\infty}/p^r\ca{O}_{F_\infty})$ is killed by $\pi$.\label{item:prop:ari-D-3}
		\item For any nonzero integer $n$ and any $q\in\bb{N}$, $H^q(\gal(F_\infty/F_m),(\ca{O}_{F_\infty}/p^r\ca{O}_{F_\infty})\{n\})$ is killed by $n^2p^m\pi$.\label{item:prop:ari-D-4}
	\end{enumerate}
\end{myprop}
\begin{proof}
	 As $F_{n_p}$ is arithmetic (by \ref{lem:n_0} and \ref{prop:ari-geo-basic}.(\ref{item:prop:ari-geo-basic-2})), we fix a nonzero element $ \pi\in\mrm{Ann}_{\ca{O}_{F_{n_0}}}(\Omega^1_{\ca{O}_{F_{n_0}}/\bb{Z}_p[\zeta_{p^{n_0}}]}[p^\infty])$. We fix $l\in\bb{N}_{\geq n_0}$ such that $|p^l|<|p^2\pi|$. Then, for any integer $m\geq l$ (so that $|np^m|<|p\pi\ak{m}_{F_m}|$ for any $n\in\bb{Z}$), all the statements have been proved in \ref{lem:ari-sigma-coh}. It remains to consider the finitely many cases $m<l$. For any $n\in\bb{Z}$ and $r\in\bb{N}$, consider the canonical convergent spectral sequence (\cite[\Luoma{5}.11.7]{abbes2016p})
	 \begin{align}\label{eq:prop:ari-D-3}
	 	E_2^{i,j}=H^i(\gal(F_l/F_m),H^j(\gal(F_\infty/F_l),(\ca{O}_{F_\infty}/p^r)\{n\}))\Rightarrow H^{i+j}(\gal(F_\infty/F_m),(\ca{O}_{F_\infty}/p^r)\{n\}).
	 \end{align}
	 Note that $E_2^{i,j}=0$ for any $j\notin\{0,1\}$ by \ref{lem:ari-sigma-coh}. Moreover, if $n\neq 0$, then $n\cdot E_2^{i,j}$ is killed by a power of $p$ independent of $i,j,r$.
	 
	 (\ref{item:prop:ari-D-1}) Taking $i=j=n=0$, as $\ca{O}_{F_m}/p^r\ca{O}_{F_m}
	 \to H^0(\gal(F_l/F_m),\ca{O}_{F_l}/p^r\ca{O}_{F_l})$ and $\ca{O}_{F_l}/p^r\ca{O}_{F_l}\to H^0(\gal(F_\infty/F_l),\ca{O}_{F_\infty}/p^r)$ has kernel and cokernel killed by a power of $p$ independent of $r$ by \ref{cor:coh-torsion} and \ref{lem:ari-sigma-coh} respectively, so is $\ca{O}_{F_m}/p^r\ca{O}_{F_m}\to E_2^{0,0}=H^0(\gal(F_\infty/F_m),\ca{O}_{F_\infty}/p^r)$.

	 (\ref{item:prop:ari-D-2}) Taking  $i=n=0$ and $j=1$, as $\ca{O}_{F_m}/p^r\ca{O}_{F_m}
	 \to H^0(\gal(F_l/F_m),\ca{O}_{F_l}/p^r\ca{O}_{F_l})$ and $p^{-l}(\log\circ\chi):\ca{O}_{F_l}/p^r\ca{O}_{F_l}\to H^1(\gal(F_\infty/F_l),\ca{O}_{F_\infty}/p^r)$ has kernel and cokernel killed by a power of $p$ independent of $r$ by \ref{cor:coh-torsion} and \ref{lem:ari-sigma-coh} respectively, so is $p^{-l}(\log\circ\chi):\ca{O}_{F_m}/p^r\ca{O}_{F_m}\to E_2^{0,1}$ and thus so is $p^{-m}(\log\circ\chi):\ca{O}_{F_m}/p^r\ca{O}_{F_m}\to E_2^{0,1}$. Note that the latter factors through $H^1(\gal(F_\infty/F_m),\ca{O}_{F_\infty}/p^r)$.
	 
	 On the other hand, taking $i=1$ and $j=n=0$, as $0
	 \to H^1(\gal(F_l/F_m),\ca{O}_{F_l}/p^r\ca{O}_{F_l})$ and $\ca{O}_{F_l}/p^r\ca{O}_{F_l}\to H^0(\gal(F_\infty/F_l),\ca{O}_{F_\infty}/p^r)$ has kernel and cokernel killed by a power of $p$ independent of $r$ by \ref{cor:coh-torsion} and \ref{lem:ari-sigma-coh} respectively, so is $0\to E_2^{1,0}$. Therefore, the conclusion follows from the canonical exact sequence $0\to E_2^{1,0}\to H^1(\gal(F_\infty/F_m),\ca{O}_{F_\infty}/p^r)\to E_2^{0,1}$.
	 
	 (\ref{item:prop:ari-D-3}) Taking $n=0$, for any $q\in\bb{N}_{\geq 2}$, as $E_2^{q,0}$ and $E_2^{q-1,1}$ are killed by a power of $p$ independent of $r$ by a similar argument of (\ref{item:prop:ari-D-2}), the conclusion follows from the canonical exact sequence $E_2^{q,0}\to H^q(\gal(F_\infty/F_m),\ca{O}_{F_\infty}/p^r)\to E_2^{q-1,1}$.
	 
	 (\ref{item:prop:ari-D-4}) For $n\neq 0$, as $n\cdot E_2^{i,j}$ is killed by a power of $p$ independent of $i,j,r$, the conclusion follows from the canonical exact sequence $ E_2^{q,0}\to H^q(\gal(F_\infty/F_m),(\ca{O}_{F_\infty}/p^r)\{n\})\to E_2^{q-1,1}$ for any $q\in\bb{N}$.
\end{proof}

\begin{mycor}\label{cor:ari-D}
	Assume that $F$ is arithmetic. Then, for any $j\in\bb{N}$, there exists a nonzero element $\pi\in\ca{O}_F$ (depending on $j$) such that the following statements hold for any $r\in\bb{N}$.
	\begin{enumerate}
		\renewcommand{\labelenumi}{{\rm(\theenumi)}}
		\item For $q=0$, the canonical morphism 
		\begin{align}\label{eq:cor:ari-D-1}
			\Omega^j_{\ca{O}_F/\bb{Z}_p}/p^r\Omega^j_{\ca{O}_F/\bb{Z}_p}\longrightarrow H^0(\gal(F_\infty/F),\Omega^j_{\ca{O}_{F_\infty}/\bb{Z}_p}/p^r\Omega^j_{\ca{O}_{F_\infty}/\bb{Z}_p})
		\end{align}
		has kernel and cokernel killed by $\pi$.\label{item:cor:ari-D-1}
		\item For $q=1$, the canonical morphism
		\begin{align}\label{eq:cor:ari-D-2}
			\Omega^j_{\ca{O}_F/\bb{Z}_p}/p^r\Omega^j_{\ca{O}_F/\bb{Z}_p}\longrightarrow H^1(\gal(F_\infty/F),\Omega^j_{\ca{O}_{F_\infty}/\bb{Z}_p}/p^r\Omega^j_{\ca{O}_{F_\infty}/\bb{Z}_p}),
		\end{align}
		sending $\omega$ to $(\log\circ \chi)\cdot \omega$ \eqref{eq:lem:n_0-2}, has kernel and cokernel killed by $\pi$.\label{item:cor:ari-D-2}
		\item For any $i\in\bb{N}_{\geq 2}$, $H^i(\gal(F_\infty/F),\Omega^j_{\ca{O}_{F_\infty}/\bb{Z}_p}/p^r\Omega^j_{\ca{O}_{F_\infty}/\bb{Z}_p})$ is killed by $\pi$.\label{item:cor:ari-D-3}
		\item For any nonzero integer $n$ and any $i\in\bb{N}$, $H^i(\gal(F_\infty/F),(\Omega^j_{\ca{O}_{F_\infty}/\bb{Z}_p}/p^r\Omega^j_{\ca{O}_{F_\infty}/\bb{Z}_p})\{n\})$ is killed by $n^2\pi$.\label{item:cor:ari-D-4}
	\end{enumerate}	
\end{mycor}
\begin{proof}
	As $F$ is arithmetic, the canonical morphism $\ca{O}_{F_\infty}\widehat{\otimes}_{\ca{O}_F}\Omega^1_{\ca{O}_F/\bb{Z}_p}\to \widehat{\Omega}^1_{\ca{O}_{F_\infty}/\bb{Z}_p}$ has kernel and cokernel killed by a certain power of $p$ by \ref{prop:ari-diff}. Thus, $H^i(\gal(F_\infty/F),(\ca{O}_{F_\infty}\otimes_{\ca{O}_F}\Omega^j_{\ca{O}_F/\bb{Z}_p}/p^r)\{n\})\to H^i(\gal(F_\infty/F),(\Omega^j_{\ca{O}_{F_\infty}/\bb{Z}_p}/p^r)\{n\})$ has kernel and cokernel killed by a power of $p$ depending on $j$ but independent on $i,r,n$ by \ref{lem:pi-isom}.
	
	On the other hand, $\Omega^1_{\ca{O}_F/\bb{Z}_p}[p^\infty]$ is bounded so that $\Omega^j_{\ca{O}_F/\bb{Z}_p}\to Q^j=\wedge^j (\Omega^1_{\ca{O}_F/\bb{Z}_p}/\Omega^1_{\ca{O}_F/\bb{Z}_p}[p^\infty])$ and thus $H^i(\gal(F_\infty/F),(\ca{O}_{F_\infty}\otimes_{\ca{O}_F}\Omega^j_{\ca{O}_F/\bb{Z}_p}/p^r)\{n\})\to H^i(\gal(F_\infty/F),(\ca{O}_{F_\infty}\otimes_{\ca{O}_F}Q^j/p^r)\{n\})$ have kernel and cokernel killed by a certain power of $p$ (depending on $j$ but not on $i,r,n$) by \ref{lem:pi-isom}. As $Q^j$ is flat over $\ca{O}_F$, we have $H^i(\gal(F_\infty/F),(\ca{O}_{F_\infty}\otimes_{\ca{O}_F}Q^j/p^r)\{n\})=H^i(\gal(F_\infty/F),(\ca{O}_{F_\infty}/p^r)\{n\})\otimes_{\ca{O}_F}Q^j$ (\cite[3.15]{abbes2016p}). Therefore, the conclusion follows from \ref{prop:ari-D} (for the case $m=0$) and a diagram chasing.
\end{proof}

\begin{mythm}\label{thm:ari}
	Assume that $F$ is arithmetic. Let $q\in\bb{N}$.
	\begin{enumerate}
		\renewcommand{\labelenumi}{{\rm(\theenumi)}}
		\item The canonical morphism induced by the cup product of \eqref{eq:lem:comp-1},
		\begin{align}\label{eq:thm:ari-1}
			\Omega^q_{\ca{O}_F/\bb{Z}_p}/p^r\Omega^q_{\ca{O}_F/\bb{Z}_p}\longrightarrow H^q(G_F,(\ca{O}_{\overline{F}}/p^r\ca{O}_{\overline{F}})\{q\}),
		\end{align}
		has kernel and cokernel killed by a power of $p$ independent of $r$.
		\label{item:thm:ari-1}
		\item The canonical morphism induced by the cup product of \eqref{eq:thm:ari-1} with $\log\circ\chi\in H^1(G_F,\ca{O}_{\overline{F}}/p^r\ca{O}_{\overline{F}})$ \eqref{eq:lem:n_0-1},
		\begin{align}\label{eq:thm:ari-2}
			\Omega^{q-1}_{\ca{O}_F/\bb{Z}_p}/p^r\Omega^{q-1}_{\ca{O}_F/\bb{Z}_p}\longrightarrow H^q(G_F,(\ca{O}_{\overline{F}}/p^r\ca{O}_{\overline{F}})\{q-1\}),
		\end{align}
		has kernel and cokernel killed by a power of $p$ independent of $r$.\label{item:thm:ari-2}
		\item For any integer $n\notin\{q,q-1\}$, $H^q(G_F,(\ca{O}_{\overline{F}}/p^r\ca{O}_{\overline{F}})\{n\})$ is killed by a power of $p$ independent of $r$.\label{item:thm:ari-3}
	\end{enumerate}
\end{mythm}
\begin{proof}
	We put $\Sigma=\gal(F_\infty/F)$ and for any $n\in\bb{Z}$, consider the convergent spectral sequence (\cite[\Luoma{5}.11.7]{abbes2016p})
	\begin{align}\label{eq:thm:ari-3}
		E_2^{i,j}=H^i(\Sigma,H^j(G_{F_\infty},(\ca{O}_{\overline{F}}/p^r\ca{O}_{\overline{F}})\{n\}))\Rightarrow H^{i+j}(G_F,(\ca{O}_{\overline{F}}/p^r\ca{O}_{\overline{F}})\{n\}).
	\end{align}
	Recall that the canonical morphism \eqref{eq:thm:coh-1},
	\begin{align}\label{eq:thm:ari-4}
		(\Omega^j_{\ca{O}_{F_\infty}/\bb{Z}_p}/p^r\Omega^j_{\ca{O}_{F_\infty}/\bb{Z}_p})\{n-j\}\longrightarrow H^j(G_{F_\infty},(\ca{O}_{\overline{F}}/p^r\ca{O}_{\overline{F}})\{n\}),
	\end{align}
	has kernel and cokernel killed by a power of $p$ depending on $j$ but not on $r,n$. Therefore, the canonical morphisms
	\begin{align}
		\Omega^n_{\ca{O}_F/\bb{Z}_p}/p^r\Omega^n_{\ca{O}_F/\bb{Z}_p}&\longrightarrow E_2^{0,n}=H^0(\Sigma,H^n(G_{F_\infty},(\ca{O}_{\overline{F}}/p^r\ca{O}_{\overline{F}})\{n\}))\label{eq:thm:ari-5}\\
		\Omega^n_{\ca{O}_F/\bb{Z}_p}/p^r\Omega^n_{\ca{O}_F/\bb{Z}_p}&\longrightarrow E_2^{1,n}=H^1(\Sigma,H^n(G_{F_\infty},(\ca{O}_{\overline{F}}/p^r\ca{O}_{\overline{F}})\{n\}))\label{eq:thm:ari-6}
	\end{align}
	have kernel and cokernel killed by a power of $p$ depending on $n$ but not on $r$ by \ref{cor:ari-D}.(\ref{item:cor:ari-D-1}, \ref{item:cor:ari-D-2}). Moreover, if $i\notin\{0,1\}$ or $j\neq n$, $E_2^{i,j}$ is killed by a power of $p$ depending on $j,n$ not on $i,r$ by \ref{cor:ari-D}.(\ref{item:cor:ari-D-3}, \ref{item:cor:ari-D-4}). Note that the spectral sequence \eqref{eq:thm:ari-3} induces a finite decreasing filtration $\mrm{F}^\bullet$ on $H^q(G_F,(\ca{O}_{\overline{F}}/p^r\ca{O}_{\overline{F}})\{n\})$ whose graded pieces $\mrm{gr}^i_{\mrm{F}}$ identifies with a subquotient of $E_2^{i,q-i}$. 
	
	(\ref{item:thm:ari-1}) Taking $n=q$, consider the canonical exact sequence
	\begin{align}\label{eq:thm:ari-7}
		0\longrightarrow \mrm{F}^1(H^q(G_F,(\ca{O}_{\overline{F}}/p^r)\{q\}))\longrightarrow H^q(G_F,(\ca{O}_{\overline{F}}/p^r)\{q\})\longrightarrow E_2^{0,q}.
	\end{align}
	On the one hand, $\mrm{F}^1(H^q(G_F,(\ca{O}_{\overline{F}}/p^r)\{q\}))$ is killed by $\prod_{i=1}^q\mrm{Ann}_{\ca{O}_F}(E_2^{i,q-i})$ and thus by a power of $p$ depending on $q$ but not on $r$. On the other hand, \eqref{eq:thm:ari-5} canonically factors through $H^q(G_F,(\ca{O}_{\overline{F}}/p^r)\{q\})$. Therefore, we see that the canonical morphism $\Omega^q_{\ca{O}_F/\bb{Z}_p}/p^r\to H^q(G_F,(\ca{O}_{\overline{F}}/p^r)\{q\})$ \eqref{eq:thm:ari-1} has kernel and cokernel killed by a power of $p$ depending on $q$ but not on $r$.
	
	(\ref{item:thm:ari-2}) Taking $n=q-1$, consider the canonical exact sequences
	\begin{align}
		0\longrightarrow \mrm{F}^1(H^q(G_F,(\ca{O}_{\overline{F}}/p^r)\{q-1\}))&\longrightarrow H^q(G_F,(\ca{O}_{\overline{F}}/p^r)\{q-1\})\longrightarrow E_2^{0,q},\label{eq:thm:ari-8}\\
		0\longrightarrow \mrm{F}^2(H^q(G_F,(\ca{O}_{\overline{F}}/p^r)\{q-1\}))&\longrightarrow \mrm{F}^1(H^q(G_F,(\ca{O}_{\overline{F}}/p^r)\{q-1\}))\longrightarrow E_2^{1,q-1}.
	\end{align}
	On the one hand, similar as above, $E_2^{0,q}$ and $\mrm{F}^2(H^q(G_F,(\ca{O}_{\overline{F}}/p^r)\{q-1\}))$ are killed by a power of $p$ depending on $q$ but not on $r$. On the other hand, \eqref{eq:thm:ari-6} canonically factors through $\mrm{F}^1(H^q(G_F,(\ca{O}_{\overline{F}}/p^r)\{q-1\}))$. Therefore, we see that the canonical morphism $\Omega^{q-1}_{\ca{O}_F/\bb{Z}_p}/p^r\longrightarrow H^q(G_F,(\ca{O}_{\overline{F}}/p^r)\{q-1\})$ \eqref{eq:thm:ari-2} has kernel and cokernel killed by a power of $p$ depending on $q$ but not on $r$.
	
	(\ref{item:thm:ari-3}) Taking $n\in\bb{Z}\setminus\{q,q-1\}$, we see that $E_2^{i,q-i}$ is killed by a power of $p$ depending on $i,q,n$ but not on $r$. Thus, so is $H^q(G_F,(\ca{O}_{\overline{F}}/p^r)\{n\})$.
\end{proof}

\begin{mycor}\label{cor:ari}
	Assume that $F$ is arithmetic. Let $q\in\bb{N}$.
	\begin{enumerate}
		\renewcommand{\labelenumi}{{\rm(\theenumi)}}
		\item The $q$-th Galois cohomology group $H^q(G_F,\ca{O}_{\widehat{\overline{F}}}\{q\})$ is $p$-adically complete. Moreover, the canonical morphism \eqref{eq:rem:prop:comp-2},
		\begin{align}\label{eq:cor:ari-1}
			\widehat{\Omega}^q_{\ca{O}_F/\bb{Z}_p}\longrightarrow H^q(G_F,\ca{O}_{\widehat{\overline{F}}}\{q\}),
		\end{align}
		has kernel and cokernel killed by a power of $p$. In particular, it induces a canonical isomorphism after inverting $p$,
		\begin{align}\label{eq:cor:ari-2}
			\widehat{\Omega}^q_{\ca{O}_F/\bb{Z}_p}[1/p]\iso H^q(G_F,\widehat{\overline{F}}(q)).
		\end{align}
		\label{item:cor:ari-1}
		\item The canonical morphism induced by the cup product of \eqref{eq:cor:ari-1} with $\log\circ\chi\in H^1(G_F,\ca{O}_{\widehat{\overline{F}}})$ \eqref{eq:lem:n_0-1},
		\begin{align}\label{eq:cor:ari-3}
			\widehat{\Omega}^{q-1}_{\ca{O}_F/\bb{Z}_p}\longrightarrow H^q(G_F,\ca{O}_{\widehat{\overline{F}}}\{q-1\}),
		\end{align}
		has kernel and cokernel killed by a power of $p$. In particular, it induces a canonical isomorphism after inverting $p$, 
		\begin{align}\label{eq:cor:ari-4}
			\widehat{\Omega}^{q-1}_{\ca{O}_F/\bb{Z}_p}[1/p]\iso H^q(G_F,\widehat{\overline{F}}(q-1)).
		\end{align}
		\label{item:cor:ari-2}
		\item For any integer $n\notin\{q,q-1\}$, $H^q(G_F,\ca{O}_{\widehat{\overline{F}}}\{n\})$ is killed by a power of $p$. In particular,
		\begin{align}\label{eq:cor:ari-5}
			H^q(G_F,\widehat{\overline{F}}(n))=0.
		\end{align}
		 \label{item:cor:ari-3}
	\end{enumerate}
\end{mycor}
\begin{proof}
	For any $n\in\bb{Z}$, there is a canonical exact sequence (\cite[\Luoma{2}.3.10]{abbes2016p})
	\begin{align}\label{eq:cor:ari-6}
		0\to  \rr^1\lim_{r\in\bb{N}}H^{q-1}(G_F,(\ca{O}_{\overline{F}}/p^r)\{n\})\to H^q(G_F,\ca{O}_{\widehat{\overline{F}}}\{n\})\to \lim_{r\in\bb{N}} H^q(G_F,(\ca{O}_{\overline{F}}/p^r)\{n\})\to 0.
	\end{align}
	
	(\ref{item:cor:ari-1}) Taking $n=q$, as $H^{q-1}(G_F,(\ca{O}_{\overline{F}}/p^r)\{q\})$ is killed by a power of $p$ independent of $r$ by \ref{thm:ari}.(\ref{item:thm:ari-3}), so is $\rr^1\lim H^{q-1}(G_F,(\ca{O}_{\overline{F}}/p^r)\{q\})$. In particular, it is $p$-adically complete. On the other hand, we claim that $\lim_{r\in\bb{N}}H^q(G_F,(\ca{O}_{\overline{F}}/p^r)\{q\})$ is $p$-adically complete with bounded $p$-power torsion. Indeed, it is $p$-adically complete as each $H^q(G_F,(\ca{O}_{\overline{F}}/p^r)\{q\})$ is killed by $p^r$ (\cite[\href{https://stacks.math.columbia.edu/tag/0G1Q}{0G1Q}]{stacks-project}). Notice that the canonical morphism induced by \eqref{eq:thm:ari-1},
	\begin{align}\label{eq:cor:ari-7}
		\widehat{\Omega}^q_{\ca{O}_F/\bb{Z}_p}=\lim_{r\in\bb{N}}\Omega^q_{\ca{O}_F/\bb{Z}_p}/p^r\longrightarrow \lim_{r\in\bb{N}}H^q(G_F,(\ca{O}_{\overline{F}}/p^r)\{q\})
	\end{align}
	has kernel and cokernel killed by a power of $p$ by \ref{thm:ari}.(\ref{item:thm:ari-1}) and \ref{lem:pi-isom}. Therefore, the latter has bounded $p$-power torsion as $\widehat{\Omega}^q_{\ca{O}_F/\bb{Z}_p}$ does so (\ref{prop:non-ari-flat}). 
	
	Therefore, $H^q(G_F,\ca{O}_{\widehat{\overline{F}}}\{q\})$ is also $p$-adically complete by \ref{lem:complete}. Then, \eqref{eq:cor:ari-7} coincides with the composition of the canonical morphisms
	\begin{align}\label{eq:cor:ari-8}
		\widehat{\Omega}^q_{\ca{O}_F/\bb{Z}_p}\longrightarrow H^q(G_F,\ca{O}_{\widehat{\overline{F}}}\{q\}) \longrightarrow   \lim_{r\in\bb{N}} H^q(G_F,(\ca{O}_{\overline{F}}/p^r)\{q\}),
	\end{align}
	where the first arrow is \eqref{eq:rem:prop:comp-2} and the second arrow is defined in \eqref{eq:cor:ari-6}. The conclusion follows from diagram chasing.
	
	(\ref{item:cor:ari-2}) Taking $n=q-1$, as the inverse system $(\Omega^{q-1}_{\ca{O}_F/\bb{Z}_p}/p^r)_{r\in\bb{N}}$ satisfies the Mittag-Leffler condition, we have $\rr^1\lim\Omega^{q-1}_{\ca{O}_F/\bb{Z}_p}/p^r=0$. Thus, $\rr^1\lim H^{q-1}(G_F,(\ca{O}_{\overline{F}}/p^r)\{q-1\})$ is killed a power of $p$ by \ref{thm:ari}.(\ref{item:thm:ari-1}) and \ref{lem:pi-isom}. On the other hand, the canonical morphism induced by \eqref{eq:thm:ari-2},
	\begin{align}\label{eq:cor:ari-9}
		\widehat{\Omega}^{q-1}_{\ca{O}_F/\bb{Z}_p}=\lim_{r\in\bb{N}}\Omega^{q-1}_{\ca{O}_F/\bb{Z}_p}/p^r\longrightarrow \lim_{r\in\bb{N}}H^q(G_F,(\ca{O}_{\overline{F}}/p^r)\{q-1\})
	\end{align}
	has kernel and cokernel killed by a power of $p$ by \ref{thm:ari}.(\ref{item:thm:ari-2}) and \ref{lem:pi-isom}.
	
	Since \eqref{eq:cor:ari-9} coincides with the composition of the canonical morphisms
	\begin{align}\label{eq:cor:ari-10}
		\widehat{\Omega}^{q-1}_{\ca{O}_F/\bb{Z}_p}\longrightarrow H^q(G_F,\ca{O}_{\widehat{\overline{F}}}\{q-1\}) \longrightarrow   \lim_{r\in\bb{N}} H^q(G_F,(\ca{O}_{\overline{F}}/p^r)\{q-1\}),
	\end{align}
	where the first arrow is induced by the cup product of \eqref{eq:cor:ari-1} with $\log\circ\chi$, and the second arrow is defined in \eqref{eq:cor:ari-6}. The conclusion follows from diagram chasing.
	
	(\ref{item:cor:ari-3}) Taking $n\in\bb{Z}\setminus\{q,q-1\}$, similarly we see that $\rr^1\lim H^{q-1}(G_F,(\ca{O}_{\overline{F}}/p^r)\{n\})$ is killed a power of $p$ by \ref{thm:ari}.(\ref{item:thm:ari-2}) for $n=q-2$ and by \ref{thm:ari}.(\ref{item:thm:ari-3}) for $n\neq q-2$. We also see that $\lim H^q(G_F,(\ca{O}_{\overline{F}}/p^r)\{n\})$ is killed a power of $p$ by \ref{thm:ari}.(\ref{item:thm:ari-3}). The conclusion follows from \eqref{eq:cor:ari-6}.
\end{proof}

\section{Appendix: Bouis' Prismatic Computation after Bhatt}\label{sec:pris}
For valuation fields over a pre-perfectoid field, we present a prismatic approach to its Galois cohomology computation due to Bhatt (see \ref{cor:cartier-smooth-val}). The key ingredient is Bouis' computation of the prismatic cohomology for such valuation fields, which relies not only on the Hodge-Tate and de Rham comparisons for prismatic cohomology but also on Gabber-Ramero's work on valuation rings (see \ref{thm:cartier-smooth} and \ref{thm:cartier-smooth-val}).

\begin{mypara}\label{para:notation-prism}
	Let $(A,I)$ be a bounded prism, $\phi_A:A\to A$ its associated Frobenius morphism (\cite[3.2]{bhattscholze2019prisms}). To any $p$-complete $(A/I)$-algebra $R$, Bhatt-Scholze \cite[7.6]{bhattscholze2019prisms} associate functorially a derived $(p,I)$-complete commutative algebra $\prism_{R/A}$ in the derived $\infty$-category $\mbf{D}(A)$ equipped with an $A$-algebra homomorphism $\phi:\prism_{R/A}\to \phi_{A*}\prism_{R/A}$. 
	We put
	\begin{align}\label{eq:para:notation-prism-1}
		\prism^{(1)}_{R/A}=\phi_A^*\prism_{R/A}=\prism_{R/A}\widehat{\otimes}^{\dl}_{A,\phi_A}A
	\end{align}
	where the completion is the derived $(p,I)$-completion. Note that $\phi$ also corresponds to an $A$-algebra homomorphism $\prism^{(1)}_{R/A}\to \prism_{R/A}$ in $\mbf{D}(A)$. It induces a canonical morphism
	\begin{align}\label{eq:para:notation-prism-2}
		\widetilde{\phi}:\prism^{(1)}_{R/A}\longrightarrow \dl\eta_I \prism_{R/A},
	\end{align}
	where $\dl\eta_I $ is the d\'ecalage functor defined in \cite[\textsection6]{bhattmorrowscholze2018integral}. This fact relies on a deep computation of the Nygaard filtration on $\prism^{(1)}_{R/A}$, see \cite[15.3]{bhattscholze2019prisms}, \cite[5.1.1]{bhattlurie2022absolute}, and also \cite[2.13]{bouis2023cartier}. 
	
	We also put
	\begin{align}\label{eq:para:notation-prism-3}
		\overline{A}=A/I,\quad \overline{\prism}_{R/A}=\prism_{R/A}\otimes^{\dl}_A\overline{A},\quad \overline{\prism}^{(1)}_{R/A}=\prism_{R/A}^{(1)}\otimes^{\dl}_A\overline{A}.
	\end{align} 
	Recall that the de Rham comparison \cite[5.2.5]{bhattlurie2022absolute} (cf. \cite[15.4]{bhattscholze2019prisms}) is a canonical equivalence
	\begin{align}\label{eq:para:notation-prism-4}
		(\bb{L}\Omega^\bullet_{R/\overline{A}})^\wedge\iso \overline{\prism}^{(1)}_{R/A},
	\end{align}
	where $\bb{L}\Omega^\bullet_{R/\overline{A}}$ is the derived de Rham complex (obtained by the left Kan extension of the de Rham complexes for polynomial $\overline{A}$-algebras). The Hodge-Tate comparison \cite[4.1.7]{bhattlurie2022absolute} (cf. \cite[6.3]{bhattscholze2019prisms}) is a canonical equivalence for any $q\in\bb{N}$,
	\begin{align}\label{eq:para:notation-prism-5}
		(\bb{L}\Omega^q_{R/\overline{A}})^\wedge[-q]\iso \mrm{gr}^{\mrm{conj}}_q(\overline{\prism}_{R/A}\{q\}),
	\end{align}
	where $\bb{L}\Omega^q_{R/\overline{A}}$ is the derived $q$-th exterior power of the cotangent complex (obtained by the left Kan extension of $\Omega^q$ for polynomial $\overline{A}$-algebras) and $\overline{\prism}_{R/A}\{q\}=\overline{\prism}_{R/A}\otimes_A (I/I^2)^{\otimes q}$ is endowed with the conjugate filtration (obtained by the left Kan extension of the Postnikov filtration for polynomial $\overline{A}$-algebras). 
\end{mypara}

\begin{mythm}[{\cite[2.17]{bouis2023cartier}}]\label{thm:cartier-smooth}
	With the notation in {\rm\ref{para:notation-prism}}, assume that $R$ is $p$-Cartier smooth over $\overline{A}$ {\rm(\cite[2.5]{bouis2023cartier})}, i.e., the following conditions hold:
	\begin{enumerate}
		\renewcommand{\theenumi}{\roman{enumi}}
		\renewcommand{\labelenumi}{{\rm(\theenumi)}}
		\item The canonical morphism $R\otimes^{\dl}_{\overline{A}}\overline{A}/p\overline{A}\to R/pR$ is an isomorphism.\label{item:thm:cartier-smooth-1}
		\item The $R/pR$-module of differentials $\Omega^1_{(R/pR)/(\overline{A}/p\overline{A})}$ is flat and the homology groups of the cotangent complex $H_q(\bb{L}_{(R/pR)/(\overline{A}/p\overline{A})})=0$ for any $q\in\bb{N}_{>0}$.\label{item:thm:cartier-smooth-2}
		\item The inverse Cartier morphism associated to the morphism of $\bb{F}_p$-algebras $\overline{A}/p\overline{A}\to R/pR$,
		\begin{align}\label{eq:thm:cartier-smooth-1}
			C^{-1}:\Omega^q_{(R/pR)/(\overline{A}/p\overline{A})}\otimes_{A,\phi_A}A\longrightarrow H^q(\Omega^\bullet_{(R/pR)/(\overline{A}/p\overline{A})}),
		\end{align}
		is an isomorphism for any $q\in\bb{N}$, where $\Omega^\bullet_{(R/pR)/(\overline{A}/p\overline{A})}$ is the de Rham complex.\label{item:thm:cartier-smooth-3}
	\end{enumerate}
	Then, the canonical morphism
	\begin{align}\label{eq:thm:cartier-smooth-2}
		\widetilde{\phi}:\prism^{(1)}_{R/A}\longrightarrow \dl\eta_I \prism_{R/A}
	\end{align}
	is an equivalence.
\end{mythm}

\begin{mypara}
	We briefly sketch Bouis' proof here. Under the assumptions (\ref{item:thm:cartier-smooth-1}) and (\ref{item:thm:cartier-smooth-2}), the Hodge-Tate comparison \eqref{eq:para:notation-prism-5} induces an isomorphism of differential graded $\overline{A}$-algebras (\cite[2.15.(3)]{bouis2023cartier},
	\begin{align}\label{eq:thm:cartier-smooth-3}
		(\Omega^\bullet_{R/\overline{A}})^\wedge\iso H^\bullet(\overline{\prism}_{R/A}\{\bullet\}),
	\end{align}
	where $H^\bullet(\overline{\prism}_{R/A}\{\bullet\})$ is the complex with Bockstein differentials induced by the $I$-adic filtration on $\overline{\prism}_{R/A}$. In particular, we have $(\Omega^q_{R/\overline{A}})^\wedge=H^q(\overline{\prism}_{R/A}\{q\})$ coincides with the classical $p$-adic completion of $\Omega^q_{R/\overline{A}}$. Combining with a fundamental property for d\'ecalage functor \cite[6.12]{bhattmorrowscholze2018integral}, we obtain canonical equivalences in the derived category
	\begin{align}\label{eq:thm:cartier-smooth-4}
		(\Omega^\bullet_{R/\overline{A}})^\wedge\iso H^\bullet(\overline{\prism}_{R/A}\{\bullet\})\iso \dl\eta_I \prism_{R/A}\otimes^{\dl}_A\overline{A}.
	\end{align}
	Combing with the de Rham comparison \eqref{eq:para:notation-prism-4}, the equivalence of $\widetilde{\phi}$ is equivalent to the equivalence of the following canonical morphism by the derived Nakayama's lemma,
	\begin{align}
		(\bb{L}\Omega^\bullet_{R/\overline{A}})^\wedge\longrightarrow(\Omega^\bullet_{R/\overline{A}})^\wedge,
	\end{align}
	where its compatibility with $\widetilde{\phi}$ follows from \cite[5.2.3]{bhattlurie2022absolute}. Again by the derived Nakayama's lemma and the assumptions (\ref{item:thm:cartier-smooth-1}) and (\ref{item:thm:cartier-smooth-2}), it remains to check the equivalence of 
	\begin{align}
		\bb{L}\Omega^\bullet_{(R/pR)/(\overline{A}/p\overline{A})}\longrightarrow \Omega^\bullet_{(R/pR)/(\overline{A}/p\overline{A})}.
	\end{align}
	Combining the assumptions (\ref{item:thm:cartier-smooth-1}) and (\ref{item:thm:cartier-smooth-2}) with the derived inverse Cartier isomorphism \cite[3.5]{bhatt2012derham},
	\begin{align}
		C^{-1}:(\bb{L}\Omega^q_{(R/pR)/(\overline{A}/p\overline{A})}\otimes^{\dl}_{A,\phi_A}A)[-q]\iso \mrm{gr}^{\mrm{conj}}_q(\bb{L}\Omega^\bullet_{(R/pR)/(\overline{A}/p\overline{A})}),
	\end{align}
	we obtain a canonical isomorphism for any $q\in\bb{N}$,
	\begin{align}
		C^{-1}:\Omega^q_{(R/pR)/(\overline{A}/p\overline{A})}\otimes_{A,\phi_A}A\iso H^q(\bb{L}\Omega^\bullet_{(R/pR)/(\overline{A}/p\overline{A})}).
	\end{align}
	Then, the conclusion follows from the assumption (\ref{item:thm:cartier-smooth-3}). 
	\begin{align}
		\xymatrix{
			\prism^{(1)}_{R/A}\ar[dd]_-{\widetilde{\phi}}\ar@{~>}[rr]^-{\trm{de Rham}} \ar@{}[ddrr]|-{\trm{reduction modulo }I}&& (\bb{L}\Omega^\bullet_{R/\overline{A}})^\wedge\ar[dd]\ar@{~>}[r]\ar@{}[ddr]|-{\trm{reduction modulo }p}&\bb{L}\Omega^\bullet_{(R/pR)/(\overline{A}/p\overline{A})}\ar[dd]&H^q(\bb{L}\Omega^\bullet_{(R/pR)/(\overline{A}/p\overline{A})})\\
			&&&&\Omega^q_{(R/pR)/(\overline{A}/p\overline{A})}\otimes_{A,\phi_A}A\ar[u]^-{\wr}_-{C^{-1}}\ar[d]_-{\wr}^-{C^{-1}}\\
			\dl\eta_I \prism_{R/A}\ar@{~>}[rr]^-{\trm{Hodge-Tate}} &&(\Omega^\bullet_{R/\overline{A}})^\wedge\ar@{~>}[r]&\Omega^\bullet_{(R/pR)/(\overline{A}/p\overline{A})}&H^q(\Omega^\bullet_{(R/pR)/(\overline{A}/p\overline{A})}) 
		}
	\end{align}
\end{mypara}

\begin{mylem}\label{lem:Leta-pos}
	Let $A$ be a ring, $I$ an invertible ideal of $A$, $M$ a derived $I$-complete complex of $A$-modules. Assume that $M\otimes^{\dl}_AA/I\in\mbf{D}^{\geq 0}(A/I)$. Then, $M\in\mbf{D}^{\geq 0}(A)$ and $H^0(M)[I]=0$.
\end{mylem}
\begin{proof}
	The problem is local on $A$ so that we may assume that $I$ is generated by a nonzero divisor of $A$. Consider the distinguished triangle 
	\begin{align}
		(\tau^{\leq -1}M)\otimes^{\dl}_AA/I\longrightarrow M\otimes^{\dl}_AA/I\longrightarrow(\tau^{\geq 0}M)\otimes^{\dl}_AA/I\longrightarrow.
	\end{align}
	As $I$ is free over $A$, we have $(\tau^{\geq 0}M)\otimes^{\dl}_AA/I\in \mbf{D}^{\geq -1}(A/I)$. Then, the associated long exact sequence 
	\begin{align}
		\cdots\to H^{n-1}((\tau^{\geq 0}M)\otimes^{\dl}_AA/I)\to H^n((\tau^{\leq -1}M)\otimes^{\dl}_AA/I)\to  H^n(M\otimes^{\dl}_AA/I)\to \cdots
	\end{align}
	implies that $H^n((\tau^{\leq -1}M)\otimes^{\dl}_AA/I)=0$ for any integer $n\leq -1$. Therefore, $(\tau^{\leq -1}M)\otimes^{\dl}_AA/I=0$. As $\tau^{\leq -1}M$ is also derived $I$-complete (\cite[6.15]{bhattmorrowscholze2018integral}), it is zero by derived Nakayama's lemma \cite[\href{https://stacks.math.columbia.edu/tag/0G1U}{0G1U}]{stacks-project}. Hence, $M=\tau^{\geq 0}M\in \mbf{D}^{\geq 0}(A)$. Moreover, $H^0(M)[I]=H^{-1}(M\otimes^{\dl}_AA/I)=0$.
\end{proof}

\begin{mycor}\label{cor:cartier-smooth}
	With the notation in {\rm\ref{para:notation-prism}}, assume that $R$ is $p$-Cartier smooth over $\overline{A}$.
	\begin{enumerate}
		\renewcommand{\labelenumi}{{\rm(\theenumi)}}
		\item We have $\prism_{R/A}\in \mbf{D}^{\geq 0}(A)$ with $H^0(\prism_{R/A})[I]=0$.\label{item:cor:cartier-smooth-1}
		\item Assume that $I$ is generated by a nonzero divisor $d\in A$. Then, for any $q\in\bb{N}$, there exists a canonical morphism $\psi:\tau^{\leq q}\prism_{R/A}\to \tau^{\leq q}\prism^{(1)}_{R/A}$ in $\mbf{D}(A)$ such that $\psi\circ\phi=d^q\cdot \id_{\tau^{\leq q}\prism^{(1)}_{R/A}}$ and $\phi\circ\psi=d^q\cdot \id_{\tau^{\leq q}\prism_{R/A}}$.\label{item:cor:cartier-smooth-2}
	\end{enumerate}
\end{mycor}
\begin{proof}
	(\ref{item:cor:cartier-smooth-1}) It follows directly from the Hodge-Tate comparison \eqref{eq:thm:cartier-smooth-4} that $\overline{\prism}_{R/A}\in \mbf{D}^{\geq 0}(\overline{A})$. Thus, the conclusion follows from \ref{lem:Leta-pos}.
	
	(\ref{item:cor:cartier-smooth-2}) Notice that $\phi:\prism^{(1)}_{R/A}\to \prism_{R/A}$ canonically factors as (cf. \cite[2.13.(2)]{bouis2023cartier})
	\begin{align}
		\prism^{(1)}_{R/A}\stackrel{\widetilde{\phi}}{\longrightarrow} \dl\eta_I\prism_{R/A}\longrightarrow \prism_{R/A}
	\end{align}
	by \eqref{eq:para:notation-prism-2} and (\ref{item:cor:cartier-smooth-1}) (\cite[6.10]{bhattmorrowscholze2018integral}). Then, the conclusion follows from the equivalence of $\widetilde{\phi}$ by \ref{thm:cartier-smooth} and a basic property for the d\'ecalage functor \cite[6.9]{bhattmorrowscholze2018integral}.
\end{proof}

\begin{mypara}\label{notation:perfect-prism}
	Let $K$ be a pre-perfectoid field. Recall that $\ca{O}_{K^\flat}=\lim_{x\mapsto x^p}\ca{O}_K/p\ca{O}_K$ is a perfect valuation ring of characteristic $p$ (\cite[3.4]{scholze2012perfectoid}). We put $A=W(\ca{O}_{K^\flat})$ and let $d$ be a generator of the kernel of Fontaine's surjection $\vartheta:A\to \ca{O}_{\widehat{K}}$ (\cite[3.10]{bhattmorrowscholze2018integral}). Then, $(A,dA)$ is a perfect prism (\cite[3.10]{bhattscholze2019prisms}). Note that $\vartheta$ induces a canonical isomorphism
	\begin{align}\label{eq:notation:perfect-prism-1}
		\ca{O}_{K^\flat}/d\ca{O}_{K^\flat}\iso \ca{O}_K/p\ca{O}_K .
	\end{align}
	For any $n\in\bb{N}$, let $\pi_n\in\ca{O}_K$ be a $p^n$-th root of $p$ up to a unit (\cite[5.4]{he2024coh}). Then, we see that $\phi_A^{-n}(d)$ and $\pi_n$ generate the same ideal in $\ca{O}_{K^\flat}/d\ca{O}_{K^\flat}= \ca{O}_K/p\ca{O}_K$. Therefore, $\vartheta(\phi_A^{-n}(d))=u\cdot \pi_n$ for some unit $u\in \ca{O}_{\widehat{K}}^\times$.
	
	Let $R$ be a $p$-complete $\overline{A}$-algebra (where $\overline{A}=A/dA=\ca{O}_{\widehat{K}}$). Following \cite[8.2]{bhattscholze2019prisms}, we define the \emph{perfection} of $\prism_{R/A}$ to be
	\begin{align}
		\prism_{R/A,\mrm{perf}}=\left(\colim(\prism_{R/A}\stackrel{\phi}{\longrightarrow}\phi_{A*}\prism_{R/A}\stackrel{\phi}{\longrightarrow}\cdots)\right)^\wedge,
	\end{align}
	where the completion is the derived $(p,I)$-completion. We define the \emph{perfectoidization} of $R$ to be
	\begin{align}
		R_{\mrm{perfd}}=\prism_{R/A,\mrm{perf}}\otimes^{\dl}_A\overline{A}=\left(\colim(\overline{\prism}_{R/A}\stackrel{\phi}{\longrightarrow}\phi_{A*}\overline{\prism}_{R/A}\stackrel{\phi}{\longrightarrow}\cdots)\right)^\wedge.
	\end{align}
	where the completion is the derived $p$-completion.
\end{mypara}

\begin{myprop}\label{prop:cartier-smooth}
	With the notation in {\rm\ref{notation:perfect-prism}}, assume that $R$ is $p$-Cartier smooth over $\ca{O}_{\widehat{K}}$. Let $\pi\in\ca{O}_K$ such that $\pi^{p-1}\in p\ca{O}_K$ and let $\varphi:\overline{\prism}_{R/A}\to R_{\mrm{perfd}}$ be the canonical morphism. Then, for any $q\in\bb{N}$, there exists a canonical morphism $\psi:\tau^{\leq q}R_{\mrm{perfd}}\to \tau^{\leq q}\overline{\prism}_{R/A}$ such that $\psi\circ\varphi=\pi^q\cdot \id_{\tau^{\leq q}\overline{\prism}_{R/A}}$ and $\varphi\circ\psi=\pi^q\cdot \id_{\tau^{\leq q}R_{\mrm{perfd}}}$. In particular, the canonical morphism
	\begin{align}\label{eq:prop:cartier-smooth-1}
		\widehat{\Omega}^q_{R/\ca{O}_{\widehat{K}}}\{-q\}=H^q(\overline{\prism}_{R/A})\longrightarrow H^q(R_{\mrm{perfd}})
	\end{align}
	has kernel and cokernel killed by $\pi^q$.
\end{myprop}
\begin{proof}
	Since $\phi_A:A\to A$ is an isomorphism (as $(A,dA)$ is perfect), for any $q,n\in\bb{N}$, there exists a canonical morphism $\psi_{n+1,n}:\tau^{\leq q}\phi_{A*}^{n+1}\prism_{R/A}\to \tau^{\leq q}\phi_{A*}^n\prism_{R/A}$ such that $\psi_{n+1,n}\circ \phi=\phi_A^{-(n+1)}(d^q)\cdot \id_{\tau^{\leq q}\phi_{A*}^n\prism_{R/A}}$ and $\phi\circ\psi_{n+1,n}=\phi_A^{-(n+1)}(d^q)\cdot \id_{\tau^{\leq q}\phi_{A*}^{n+1}\prism_{R/A}}$ by \ref{cor:cartier-smooth}.(\ref{item:cor:cartier-smooth-2}).
	\begin{align}
		\xymatrix{
			\tau^{\leq q}\overline{\prism}_{R/A}\ar@<2pt>[r]^-{\phi}&\tau^{\leq q}\phi_{A*}\overline{\prism}_{R/A}\ar@<2pt>[r]^-{\phi}\ar@<2pt>[l]^-{\psi_{1,0}}&\tau^{\leq q}\phi_{A*}^2\overline{\prism}_{R/A}\ar@<2pt>[r]^-{\phi}\ar@<2pt>[l]^-{\psi_{2,1}}&\cdots\ar@<2pt>[l]^-{\psi_{3,2}}
		}
	\end{align}
	We put $\pi_n=\vartheta(\phi_A^{-n}(d))\in\ca{O}_{\widehat{K}}$. Note that its normalized valuation is $v_p(\pi_n)=1/p^n$. In particular, $v_p(\pi_1\cdots\pi_n)=1/p+\cdots +1/p^n<1/(p-1)$, i.e., $\pi\in \pi_1\cdots\pi_n\cdot\ak{m}_{\widehat{K}}$, where $\pi\in\ca{O}_K$ such that $\pi^{p-1}\in p\ca{O}_K$. We put $\psi_n=\pi^q\cdot (\pi_1\cdots\pi_n)^{-q} \cdot \psi_{1,0}\circ\cdots\circ\psi_{n,n-1}:\tau^{\leq q}\phi_{A*}^n\overline{\prism}_{R/A}\to \tau^{\leq q}\overline{\prism}_{R/A}$. 
	\begin{align}
		\xymatrix{
			\tau^{\leq q}\overline{\prism}_{R/A}\ar@<2pt>[r]^-{\phi^n}&\tau^{\leq q}\phi_{A*}^n\overline{\prism}_{R/A}\ar@<2pt>[l]^-{\psi_n}
		}
	\end{align}
	Then, we have $\psi_n\circ\phi^n=\pi^q\cdot \id_{\tau^{\leq q}\overline{\prism}_{R/A}}$, $\phi^n\circ \psi_n=\pi^q\cdot\id_{\tau^{\leq q}\phi_{A*}^n\overline{\prism}_{R/A}}$ and $\psi_{n+1}\circ\phi=\psi_n$. Taking colimit over $n\in\bb{N}$, we obtain canonical morphisms
	\begin{align}
		\xymatrix{
			\tau^{\leq q}\overline{\prism}_{R/A}\ar@<2pt>[r]^-{\varphi}&\tau^{\leq q}M\ar@<2pt>[l]^-{\psi}
		}
	\end{align}
	where $M=\colim(\overline{\prism}_{R/A}\stackrel{\phi}{\longrightarrow}\phi_{A*}\overline{\prism}_{R/A}\stackrel{\phi}{\longrightarrow}\cdots)$, such that $\psi\circ\varphi=\pi^q\cdot\id_{\tau^{\leq q}\overline{\prism}_{R/A}}$ and $\varphi\circ\psi=\pi^q\cdot\id_{\tau^{\leq q}M}$. In particular, the kernel and cokernel of $H^q(\overline{\prism}_{R/A})\to H^q(M)$ are killed by $\pi^q$. This shows that $M$ is derived $p$-complete as $\overline{\prism}_{R/A}$ is so by \cite[6.15]{bhattmorrowscholze2018integral}. Therefore, $M=R_{\mrm{perfd}}$ and the conclusion follows from the Hodge-Tate comparison \eqref{eq:thm:cartier-smooth-3}.
\end{proof}

\begin{mypara}\label{para:notation-arc}
	Let $((\mbf{fSch}_{/\bb{Z}_p})_{\mrm{arc}},\ca{O})$ be the ringed site of $p$-adic formal schemes endowed with the $p$-adic arc topology (\cite[8.7]{bhattscholze2019prisms}). Recall that the objects $\mrm{Spf}(R)$, where $R$ is a perfectoid ring {\rm(\cite[3.5]{bhattmorrowscholze2018integral})}, form a topological basis of $(\mbf{fSch}_{/\bb{Z}_p})_{\mrm{arc}}$ (\cite[8.8]{bhattscholze2019prisms}).
\end{mypara}

\begin{mythm}[{\cite[8.10, 8.11]{bhattscholze2019prisms}}]\label{thm:arc}
	With the notation in {\rm\ref{para:notation-arc}}, let $R$ be a $p$-complete ring and $\ak{X}=\mrm{Spf}(R)$.
	\begin{enumerate}
		\renewcommand{\labelenumi}{{\rm(\theenumi)}}
		\item If $R$ is perfectoid, then for any $r\in\bb{N}$, the canonical morphism
		\begin{align}\label{eq:thm:arc-1}
			R/p^rR\longrightarrow \rr\Gamma((\mbf{fSch}_{/\ak{X}})_{\mrm{arc}},\ca{O}/p^r\ca{O})
		\end{align}
		is an isomorphism.\label{item:thm:arc-1}
		\item The canonical morphism 
		\begin{align}\label{eq:thm:arc-2}
			\ca{O}\longrightarrow \rr\lim_{r\in\bb{N}}\ca{O}/p^r\ca{O}
		\end{align}
		is an isomorphism. In particular, $R=\rr\Gamma((\mbf{fSch}_{/\ak{X}})_{\mrm{arc}},\ca{O})$ if $R$ is perfectoid.\label{item:thm:arc-2}
		\item There is a canonical isomorphism
		\begin{align}\label{eq:thm:arc-3}
			R_{\mrm{perfd}}\iso \rr\Gamma((\mbf{fSch}_{/\ak{X}})_{\mrm{arc}},\ca{O}).
		\end{align}\label{item:thm:arc-3}
	\end{enumerate}
\end{mythm}
\begin{proof}
	(\ref{eq:thm:arc-1}) follows from the proof of \cite[8.10]{bhattscholze2019prisms}; (\ref{item:thm:arc-2}) follows from (\ref{eq:thm:arc-1}) and \cite[\href{https://stacks.math.columbia.edu/tag/0BKY}{0BKY}]{stacks-project}; and (\ref{item:thm:arc-3}) is proved in \cite[8.11]{bhattscholze2019prisms}.
\end{proof}

\begin{mypara}\label{notation:fal-site}
	Let $K$ be a pre-perfectoid field extension of $\bb{Q}_p$, $\eta=\spec(K)$, $S=\spec(\ca{O}_K)$. Let $(\falh_{\eta\to S},\ca{O})$ be the ringed site of $\eta$-integrally closed coherent $S$-schemes endowed with the v-topology (\cite[3.23]{he2024coh}). Recall that the objects $\spec(R)$, where $R$ is an almost pre-perfectoid $\ca{O}_K$-algebra {\rm(\ref{para:notation-perfd})}, form a topological basis of $\falh_{\eta\to S}$ (\cite[8.5, 8.10]{he2024coh}).
\end{mypara}

\begin{mythm}[{\cite[8.11]{he2024coh}}]\label{thm:fal-site}
	With the notation in {\rm\ref{notation:fal-site}}, let $R$ be an $\ca{O}_K$-algebra integrally closed in $R[1/p]$ and $X=\spec(R)$.
	\begin{enumerate}
		\renewcommand{\labelenumi}{{\rm(\theenumi)}}
		\item If $R$ is almost pre-perfectoid, then for any $r\in\bb{N}$, the canonical morphism
		\begin{align}\label{eq:thm:fal-site-1}
			R/p^rR\longrightarrow \rr\Gamma(\falh_{X_\eta\to X},\ca{O}/p^r\ca{O})
		\end{align}
		is an almost isomorphism.\label{item:thm:fal-site-1}
		\item Let $\widehat{\ca{O}}=\lim_{r\in\bb{N}}\ca{O}/p^r\ca{O}$ be the $p$-adic completion of $\ca{O}$. Then, the canonical morphism 
		\begin{align}\label{eq:thm:fal-site-2}
			\widehat{\ca{O}}\longrightarrow \rr\lim_{r\in\bb{N}}\ca{O}/p^r\ca{O}
		\end{align}
		is an almost isomorphism. In particular, the canonical morphism $\widehat{R}\to \rr\Gamma(\falh_{X_\eta\to X},\widehat{\ca{O}})$ is an almost isomorphism if $R$ is almost pre-perfectoid.\label{item:thm:fal-site-2}
	\end{enumerate}
\end{mythm}
\begin{proof}
	(\ref{item:thm:fal-site-1}) follows from \cite[8.11]{he2024coh}; (\ref{item:thm:fal-site-2}) follows from (\ref{item:thm:fal-site-1}) and \cite[\href{https://stacks.math.columbia.edu/tag/0BKY}{0BKY}]{stacks-project}.
\end{proof}

\begin{mylem}\label{lem:perfd}
	Let $K$ be a pre-perfectoid field. Consider the following properties of an $\ca{O}_K$-algebra $R$:
	\begin{enumerate}
		\renewcommand{\labelenumi}{{\rm(\theenumi)}}
		\item the $\ca{O}_K$-algebra $R$ is pre-perfectoid {\rm(\ref{para:notation-perfd})};\label{item:lem:perfd-1}
		\item the $p$-adic completion $\widehat{R}$ is a perfectoid ring {\rm(\cite[3.5]{bhattmorrowscholze2018integral})};\label{item:lem:perfd-2}
		\item the $\ca{O}_K$-algebra $R$ is almost pre-perfectoid {\rm(\ref{para:notation-perfd})}.\label{item:lem:perfd-3}
	\end{enumerate}
	Then, we have {\rm(\ref{item:lem:perfd-1})} $\Rightarrow$ {\rm(\ref{item:lem:perfd-2})} $\Rightarrow$ {\rm(\ref{item:lem:perfd-3})}.
\end{mylem}
\begin{proof}
	Let $\pi\in\ak{m}_K$ be a nonzero element with $p\in\pi^p\ca{O}_K$. If $R$ is pre-perfectoid, then $\widehat{R}$ is flat over $\ca{O}_{\widehat{K}}$ and the Frobenius induces an isomorphism $\widehat{R}/\pi\widehat{R}\iso \widehat{R}/\pi^p\widehat{R}$. Thus, $\widehat{R}$ is a perfectoid ring by \cite[3.10.(\luoma{2})]{bhattmorrowscholze2018integral}. If $\widehat{R}$ is perfectoid, then $\widehat{R}[p^\infty]$ is killed by $\ak{m}_K$ by \cite[2.1.3]{cesnaviciusscholze2019purity}, i.e., $\widehat{R}$ is almost flat over $\ca{O}_{\widehat{K}}$ (\cite[5.12]{he2024coh}). Thus, $R$ is almost pre-perfectoid.
\end{proof}

\begin{myprop}\label{prop:fal-site-perfd}
	Let $K$ be a pre-perfectoid field extension of $\bb{Q}_p$, $X$ a coherent $\ca{O}_K$-scheme, $\widehat{X}$ the associated $p$-adic formal scheme of $X$, $X_\eta$ the generic fibre of $X$, $X^\eta$ the integral closure of $X$ in $X_\eta$. Then, for any $r\in\bb{N}$, there is a canonical morphism
	\begin{align}\label{eq:prop:fal-site-perfd-1}
		\rr\Gamma((\mbf{fSch}_{/\widehat{X}})_{\mrm{arc}},\ca{O}/p^r\ca{O})\longrightarrow \rr\Gamma(\falh_{X_\eta\to X^\eta},\ca{O}/p^r\ca{O})
	\end{align}
	which induces an almost isomorphism $H^q((\mbf{fSch}_{/\widehat{X}})_{\mrm{arc}},\ca{O}/p^r\ca{O})\to H^q(\falh_{X_\eta\to X^\eta},\ca{O}/p^r\ca{O})$ for any $q\in\bb{N}$. In particular, the canonical morphism
	\begin{align}\label{eq:prop:fal-site-perfd-2}
		\rr\Gamma((\mbf{fSch}_{/\widehat{X}})_{\mrm{arc}},\ca{O})\longrightarrow \rr\Gamma(\falh_{X_\eta\to X^\eta},\widehat{\ca{O}})
	\end{align}
	is an almost isomorphism.
\end{myprop}
\begin{proof}
	Let $((\mbf{Sch}^{\mrm{coh}}_{/X})_{\mrm{v}},\ca{O})$ be the ringed site of coherent $X$-schemes endowed with $v$-topology (\cite[3.3]{he2024coh}). There are canonical functors 
	\begin{align}
		\mbf{fSch}_{/\widehat{X}}\longleftarrow\mbf{Sch}^{\mrm{coh}}_{/X}\longrightarrow \falh_{X_\eta\to X^\eta},\ \widehat{U}\mapsfrom U\mapsto U^\eta,
	\end{align}
	which are left exact (\cite[\Luoma{1}.1.4.6]{fujiwarakato2018rigid} and \cite[3.21]{he2024coh}) and continuous (\cite[3.4.(4), 3.15]{he2024coh}). Therefore, they define canonical morphisms of ringed sites
	\begin{align}
		((\mbf{fSch}_{/\widehat{X}})_{\mrm{arc}},\ca{O})\stackrel{\nu}{\longrightarrow} ((\mbf{Sch}^{\mrm{coh}}_{/X})_{\mrm{v}},\widehat{\ca{O}})\stackrel{\mu}{\longleftarrow} (\falh_{X_\eta\to X^\eta},\widehat{\ca{O}}).
	\end{align}
	Notice that the objects $U=\spec(R)$ such that $\widehat{R}$ is a perfectoid ring form a topological basis of $(\mbf{Sch}^{\mrm{coh}}_{/X})_{\mrm{v}}$ (\cite[3.14]{he2024coh}). For such $U$, we see that $\widehat{U}=\mrm{Spf}(\widehat{R})$ is the formal spectrum of a perfectoid ring $\widehat{R}$ and that $U^\eta=\spec(\overline{R})$ is the spectrum of an almost pre-perfectoid $\ca{O}_K$-algebra (\ref{lem:perfd} and \cite[5.30]{he2024coh}). Therefore, the canonical morphism $\ca{O}/p^r\ca{O}\to \rr\nu_*(\ca{O}/p^r\ca{O})$ is an isomorphism by \ref{thm:arc}.(\ref{item:thm:arc-1}) and the canonical morphism $\ca{O}/p^r\ca{O}\to \rr\mu_*(\ca{O}/p^r\ca{O})$ is an almost isomorphism by \ref{thm:fal-site}.(\ref{item:thm:fal-site-1}). Thus, we obtain canonical morphisms
	\begin{align}
		\rr\Gamma((\mbf{fSch}_{/\widehat{X}})_{\mrm{arc}},\ca{O}/p^r\ca{O})\longleftarrow \rr\Gamma((\mbf{Sch}^{\mrm{coh}}_{/X})_{\mrm{v}},\ca{O}/p^r\ca{O}) \longrightarrow \rr\Gamma(\falh_{X_\eta\to X^\eta},\ca{O}/p^r\ca{O}),
	\end{align}
	where the left arrow is an isomorphism and the right arrow is an almost isomorphism. The conclusion follows from taking $\rr\lim_{r\in\bb{N}}$ by \ref{thm:arc}.(\ref{item:thm:arc-2}) and \ref{thm:fal-site}.(\ref{item:thm:fal-site-2}).
\end{proof}

\begin{mycor}\label{cor:fal-site-perfd}
	Let $K$ be a pre-perfectoid field extension of $\bb{Q}_p$, $F$ a valuation field of height $1$ extension of $K$, $\overline{F}$ an algebraic closure of $F$, $\ca{O}_{\overline{F}}$ the integral closure of $\ca{O}_F$ in $\overline{F}$ endowed with the natural continuous action of the absolute Galois group $G_F=\gal(\overline{F}/F)$ of $F$. Then, there is a canonical isomorphism in the derived category of almost $\ca{O}_{\widehat{F}}$-modules $\mbf{D}(\ca{O}_{\widehat{F}}^{\al})$,
	\begin{align}
		\rr\Gamma(G_F,\ca{O}_{\widehat{\overline{F}}})\iso(\ca{O}_{\widehat{F}})_{\mrm{perfd}}.
	\end{align}
\end{mycor}
\begin{proof}
	We take $X=\spec(\ca{O}_F)$, $x=X_\eta=\spec(F)$ and $\overline{x}=\spec(\overline{F})$. There is a canonical left exact continuous functor (\cite[8.5]{he2024coh})
	\begin{align}
		X_{\eta,\profet}\longrightarrow \falh_{X_\eta\to X},\ V\mapsto X^V
	\end{align}
	where $X^V$ is the integral closure of $X$ in the pro-finite \'etale $X_\eta$-scheme $V$. Then, it defines a canonical morphism of ringed sites  
	\begin{align}
		\beta:(\falh_{X_\eta\to X},\ca{O})\longrightarrow (X_{\eta,\profet},\falb),
	\end{align}
	where $\falb$ is the sheaf associated to the presheaf sending $V$ to $\Gamma(X^V,\ca{O}_{X^V})$. As $\ca{O}_{\overline{F}}$ is pre-perfectoid, we see that the canonical morphism $\falb/p^r\falb\to \rr\beta_*(\ca{O}/p^r\ca{O})$ is an almost isomorphism by \ref{thm:fal-site}.(\ref{item:thm:fal-site-1}). Therefore, there are canonical morphisms
	\begin{align}
		\xymatrix{
			\rr\Gamma((\mbf{fSch}_{/\widehat{X}})_{\mrm{arc}},\ca{O}/p^r\ca{O})\ar[r]\ar@{.>}[d]& \rr\Gamma(\falh_{X_\eta\to X},\ca{O}/p^r\ca{O})&\rr\Gamma(X_{\eta,\profet},\falb/p^r\falb)\ar[l]\\
			\rr\Gamma(G_F,\ca{O}_{\overline{F}}/p^r\ca{O}_{\overline{F}})\ar[rr]^-{\sim}&&\rr\Gamma(X_{\eta,\fet},\falb/p^r\falb)\ar[u]_-{\wr}
		}
	\end{align}
	where the arrows on the top are almost isomorphisms, the right vertical arrow is the canonical isomorphism (cf. \cite[7.32]{he2024coh}) and the arrow on the bottom is the canonical isomorphism as $G_F=\pi_1(X_\eta,\overline{x})$ (\cite[\Luoma{6}.9.8.6]{abbes2016p}). The conclusion follows from taking $\rr\lim_{r\in\bb{N}}$ by \ref{thm:arc}.(\ref{item:thm:arc-2}, \ref{item:thm:arc-3}) and the fact that $\rr\lim_{r\in\bb{N}}\rr\Gamma(G_F,\ca{O}_{\overline{F}}/p^r\ca{O}_{\overline{F}})=\rr\Gamma(G_F,\ca{O}_{\widehat{\overline{F}}})$ (\cite[2.2]{jannsen1988cont}).
\end{proof}

\begin{mythm}[{\cite[3.1]{bouis2023cartier}}]\label{thm:cartier-smooth-val}
	Let $K\to F$ be an extension of valuation fields with $K$ pre-perfectoid. Then, $\ca{O}_F$ is $p$-Cartier smooth over $\ca{O}_K$.
\end{mythm}

\begin{mypara}
	We briefly sketch Bouis' proof here. The case for $K$ of characteristic $p$ is proved by Gabber, see \cite[A.4]{kerz2021vorst}. For $K$ of characteristic $0$, after reducing to the case where $F$ is of height $1$ and finitely generated over $K$, it suffices to construct a valuation field $L$ extension of the tilt $K^\flat$ with an isomorphism compatible with \eqref{eq:notation:perfect-prism-1},
	\begin{align}
		\ca{O}_L/d\ca{O}_L\iso \ca{O}_F/p\ca{O}_F.
	\end{align}
	Bouis took induction to construct a flat $\ca{O}_{K^\flat}/d^n\ca{O}_{K^\flat}$-algebra $\ca{O}_L/d^n\ca{O}_L$ by deformation theory. This is possible because the cotangent complex of the valuation ring extension $\ca{O}_K\to\ca{O}_F$ is concentrated in degree $0$ where it is given by a countably generated flat $\ca{O}_F$-module $\Omega^1_{\ca{O}_F/\ca{O}_K}$ (\ref{thm:differential} and \ref{cor:kaplansky}, cf. \cite[3.11]{bouis2023cartier}).
\end{mypara}

\begin{mycor}[{cf. \ref{cor:geo}}]\label{cor:cartier-smooth-val}
	Let $K$ be a pre-perfectoid field extension of $\bb{Q}_p$, $F$ a valuation field of height $1$ extension of $K$, $\overline{F}$ an algebraic closure of $F$, $\ca{O}_{\overline{F}}$ the integral closure of $\ca{O}_F$ in $\overline{F}$ endowed with the natural continuous action of the absolute Galois group $G_F=\gal(\overline{F}/F)$ of $F$. Then, for any $q\in\bb{N}$ and $\varpi\in\ca{O}_K$ such that $\varpi^{p-1}\in p\ca{O}_K$, there is a canonical morphism of $\ca{O}_{\widehat{F}}^{\al}$-modules
	\begin{align}
		\widehat{\Omega}^q_{\ca{O}_F/\ca{O}_K}\{-q\}\longrightarrow H^q(G_F,\ca{O}_{\widehat{\overline{F}}})
	\end{align}
	with kernel and cokernel both killed by $\varpi^q$.
\end{mycor}
\begin{proof}
	It follows directly from \ref{cor:fal-site-perfd} and \ref{prop:cartier-smooth} (whose assumptions hold by \ref{thm:cartier-smooth-val}).
\end{proof}

\bibliographystyle{myalpha}
\bibliography{bibli}
\end{document}